\documentclass{article}
\usepackage{amssymb}
\usepackage{amsmath}
\usepackage{amsthm}
\usepackage{amsbsy}
\usepackage{enumerate}
\usepackage{bm}
\usepackage{hyperref}
\date{\today}
\usepackage{cite}
\usepackage{array}
\usepackage{xcolor}
\usepackage{tikz}
\usepackage{graphics}
\usepackage{subcaption}
\usepackage{mathtools}
\usepackage{relsize}
\usepackage{soul}
\usepackage{mathdots}
\usepackage{multirow}
\usetikzlibrary{matrix}
\makeatletter
\newcommand{\distas}[1]{\mathbin{\overset{#1}{\kern\z@\sim}}}

\newcommand{\distras}[1]{
  \mathbin{\overset{#1}{\kern\z@\resizebox{\wd\mybox}{\ht\mysim}{$\sim$}}}
}

    \newtheorem{theorem}{Theorem}
    \newtheorem{lemma}[theorem]{Lemma}
    \newtheorem{proposition}[theorem]{Proposition}

\theoremstyle{definition} 
    \newtheorem{definition}[theorem]{Definition}

    \newtheorem{remark}[theorem]{Remark}
    \newtheorem{example}[theorem]{Example}
    \newtheorem{exercise}[theorem]{Exercise}
    
    \newtheorem{assumption}[]{Assumption}
    
    \newtheorem{notation}[theorem]{Notation}

\renewcommand{\theequation}{\arabic{equation}}

\def\iff{\Longleftrightarrow}

\def \CC{{\mathcal C}}

\newcommand{\E}{\mathbb{E}}

\def\Z{\mathbb{Z}}

\def\N{\mathbb{N}}
\def\PP{{\mathcal P}}

\def\WW{\mathcal{W}}
\def\CC{\mathcal{C}}
\def\ZZ{{\mathcal Z}}

\def\u{{\bf u}}

\def\<{\langle}
\def\>{\rangle}

\newcommand{\one}{{\mathbf 1}}

\newcommand\Tr{{\mbox{Tr}}}

\newcommand\mnote[1]{} 
\newcommand\be{\begin{equation*}}

\newcommand\ee{\end{equation*}}

\newcommand\ben{\begin{equation}}
\newcommand\een{\end{equation}}
\newcommand\bes{\begin{eqnarray*}}
\newcommand\ees{\end{eqnarray*}}

\newcommand\bex{\begin{exercise}}
\newcommand\eex{\end{exercise}}
\newcommand\beg{\begin{example}}
\newcommand\eeg{\end{example}}
\newcommand\benu{\begin{enumerate}}
\newcommand\eenu{\end{enumerate}}
\newcommand\beit{\begin{itemize}}
\newcommand\eeit{\end{itemize}}
\newcommand\berk{\begin{remark}}
\newcommand\eerk{\end{remark}}
\newcommand\bdefn{\begin{defintion}}
\newcommand\edefn{\end{definition}}
\newcommand\bthm{\begin{theorem}}
\newcommand\ethm{\end{theorem}}
\newcommand\bprf{\begin{proof}}
\newcommand\eprf{\end{proof}}
\newcommand\blem{\begin{lemma}}
\newcommand\elem{\end{lemma}}

\newcommand{\var}{\mbox{\rm Var}}

\newcommand{\sm}{{\raise0.3ex\hbox{$\scriptstyle \setminus$}}}

\def\CHI{\mathchoice%
{\raise2pt\hbox{$\chi$}}%
{\raise2pt\hbox{$\chi$}}%
{\raise1.3pt\hbox{$\scriptstyle\chi$}}%
{\raise0.8pt\hbox{$\scriptscriptstyle\chi$}}}
\def\smalloplus{\raise1pt\hbox{$\,\scriptstyle \oplus\;$}}

\usepackage{tikz}

\newcommand{\varA}{\var\left(\Tr\left(\frac{A_n}{\sqrt{N}}\right)^{p_n}\right)}
\newcommand{\varAp}{\var\left(\Tr\left(\frac{A_n}{\sqrt{N}}\right)^{p}\right)}

\usetikzlibrary{shapes.geometric, arrows}

\tikzstyle{startstop} = [rectangle, rounded corners, minimum width=3cm, minimum height=1cm,text centered, draw=black, fill=red!30]

\tikzstyle{process} = [rectangle, minimum width=3cm, minimum height=1cm, text centered, draw=black, fill=orange!30]

\tikzstyle{decision} = [diamond, minimum width=3cm, minimum height=1cm, text centered, draw=black, fill=green!30]

\tikzstyle{arrow} = [thick,->,>=stealth]

\tikzstyle{io} = [trapezium, trapezium left angle=70, trapezium right angle=110, minimum width=3cm, minimum height=1cm, text centered, draw=black, fill=blue!30]

\hypersetup{
	colorlinks=true, 
	linkcolor=blue,
	filecolor=dviolet,
	urlcolor=blue,}

\usepackage[left=3.00cm, right=3.00cm, top=2.00cm, bottom=2.00cm]{geometry}

\usepackage{authblk}

\title {CLT for generalized patterned random matrices: a unified approach} 

\date{}  

\author[]{Kiran Kumar A.S.\thanks{kiran.kumar.as.math@gmail.com $^{\dagger}$shambhumath4@gmail.com $^{\ddagger}$koushik.saha@iitb.ac.in}}
\author[]{Shambhu Nath Maurya$^{\dagger}$}
\author[]{Koushik Saha$^{\ddagger}$}
\affil[]{$^{\ast}$ Division of Science, New York University Abu Dhabi, Abu Dhabi, United Arab Emirates}
\affil[]{$^{\dagger}$ Statistics and Mathematics Unit, Indian Statistical Institute, Kolkata,\\ 700108, West Bengal, India}
\affil[]{$^{\ddagger}$ Department of Mathematics, Indian Institute of Technology Bombay, Mumbai, 400076, Maharashtra, India}
\date{\today}

\begin{document}
\maketitle
\begin{abstract}

In this paper, we derive a unified method for establishing the distributional convergence of linear eigenvalue statistics (LES) for generalized patterned random matrices. 
We prove that for an $N \times N$ generalized patterned random matrix with independent subexponential entries and even degree monomial test functions of degree  $p_n=o(\log N/\log \log N)$, the LES converges to standard Gaussian distribution. This generalizes the CLT results on Gaussian patterned random matrices in \cite{chatterjee2009fluctuations}, \cite{adhikari_saha2017}. As an application, new results on LES of Toeplitz, Hankel, circulant-type matrices and block patterned random matrices for varying test functions are derived.
For odd degree monomial test functions, we derive the limiting moments of LES and show that it may not converge to a Gaussian distribution.  
\end{abstract}


\section{  \large \textbf{  Introduction}} \label{sec:Intro}

Patterned or structured matrices have been extensively studied in mathematics since the 19th century, with their structure allowing efficient computations on them. Patterned matrices such as Toeplitz, Hankel and Circulant appear widely in different areas of mathematics, sciences and engineering. Toeplitz matrix is most well-known for Szeg\H{o}'s theorem, which connects functionals of a Fourier series to the linear eigenvalue statistics of the Hermitian Toeplitz matrix constructed using its Fourier coefficients \cite{szego_toeplitz_915}. Additionally, the covariance matrix of a weakly stationary process also takes the form of a Toeplitz matrix. Hankel matrix is used to check the solvability of the \textit{Hamburger moment problem} \cite{shohat_moment+problem_43} and in signal analysis \cite{jain_signal+analysis_15}. Another important class of patterned matrix, circulant matrix, is directly related to Weyl-Heisenberg groups and non-commutative geometry of quantum mechanics \cite{aldrovaandi_circulant_book_01}.  Circulant matrices are also used widely in the spectral analysis of time series \cite{fan+yao_time+series_book_03}. For an overview of the wide range of applications of patterned matrices, see \cite{philip+davis_circulant_book_79}, \cite{aldrovaandi_circulant_book_01} and \cite{Botcher_Toeplitz_book}.

The study of the random versions of patterned matrices was initiated in a review article by Z. D. Bai\cite{bai_zd_99}.
Consider a sequence of 
random variables $\left\{x_{i} : i \in \Z^d \right\}$ called the \textit{input sequence}. For positive integers $N$ and $d$, consider a function $L: \{1,2,\ldots, N\}^2 \rightarrow \Z^d$ called the \textit{link function}. A link function is called symmetric link function if $L(i,j)=L(j,i)$ for all $1 \leq i,j \leq N$. An $N \times N$ \textit{patterned random matrix} with link function $L$ and input sequence $\{x_i: i \in \Z^d\}$ is defined as $A=\left(x_{L(i,j)}\right)_{i,j=1}^N$.

In this paper, we focus on symmetric patterned random matrices with $\Theta(n)$\footnote[1]{$f(n)= \Theta (g(n))$ if there exist $k_1,k_2>0$ and $n_0 \in \N$ such that $k_1 g(n) \leq f(n) \leq k_2 g(n)$ for all $n \geq n_0$.} independent entries. Several important classes of random matrices fall into the category of patterned random matrices with symmetric link functions and $\Theta(n)$ independent entries. A few of the most well-known patterned random matrices of this form are Toeplitz matrix (with link function $L(i,j)=|i-j|$), Hankel matrix ($L(i,j)=i+j$), reverse circulant matrix ($L(i,j)=(i+j-2) \mod n$) and symmetric circulant matrix ($L(i,j)=n/2-|n/2-|i-j||$). 

In comparison to Wigner matrices, patterned random matrices with $\Theta(n)$ independent random variables requires less storage and allows better computational speeds.
As a result, in the last two decades, patterned random matrices turned out to be extremely useful objects in a wide range of areas in computer science, such as convolutional neural networks \cite{Liao_Yuan_Circconv_2019},\cite{Liu_Jiao_Lim_LU_Decompos_DL_2024}; signal reconstruction \cite{Toeplitz_compressed_sending_2007}, communication systems \cite{couillet_RMT+wireless_book_11} and quantum random number generation \cite{Ng_Yu_Toeplitz_Hashing_2024}. Randomized versions of patterned matrices are also widely used in non-parametric regression \cite{Yin_Liu_Mend_IEEE_Circulant_regression_2020}, solving stochastic evolution equations \cite{Qi_SEE_circulant_2022}, probabilistic error-correction \cite{Scott_Goldreich_Ron_error_correction_2020} and Johnson-Lindenstrauss embedding and dimension reduction \cite{Hui_Cheng_JL_embedding_circulant_2014}. Patterned matrices also appear in study of other random ensembles such as orthogonal ensembles and unitary ensembles \cite{Johansson_Berry_Essen_2024_EJP}.

Following the introduction of random Toeplitz and Hankel matrices in \cite{bai_zd_99}, their limiting spectral distributions (LSD) were derived independently by Bryc, Dembo, Jiang  \cite{bryc_lsd_06}, and Hammond, Miller\cite{hammond_miller_05}. Their LSDs are not among the well-known distributions and are determined in terms of their moments. These works portrayed the distinctiveness of models with $\Theta(n)$ independent entries from classical random matrices such as Wigner, GUE and GOE matrices. This led to an interest in the study of the limiting spectral distribution of other patterned random matrices \cite{Massey_Miller_2007_Palin_Toeplitz}, \cite{bose_sen_LSD_EJP}. The asymptotic behavior of patterned random matrices is relatively well-studied in the literature and their spectral properties, such as 
local spectral statistics \cite{Bogomolny_Spectrum_Random_Toep_2020_PRE},\cite{Ali_Srivastava_JPA_2022_levelspacing}, distribution of norm \cite{Nishry_Paquette_Complex_hankel_norm_2023}, \cite{Pan_John_random_circulant_norm_2015}, \cite{Bose_Hazra_saha_norm_circulant_2011}, distribution of singular values \cite{Barrera_Paulo_circulant_singular_2022}, \cite{Paulo_Toeplitz_conditionno_2022}, joint convergence \cite{bose_saha_patter_JC_annals} and their applications \cite{Goldreich_tal_Topleitz_rigidity_2018} continue to attract considerable attraction. 
	
This paper focuses on the linear eigenvalue statistics (LES) of a class of random matrices which includes patterned random matrices, as well as their block versions and banded versions. In \cite{chatterjee2009fluctuations}, Chatterjee introduced the method of second order Poincare inequalities to study the LES of Gaussian random matrices. He showed that the LES of Gaussian Toeplitz matrices for monomial test functions $\phi_n=x^{p_n}$ converges to the standard Gaussian distribution under total variation norm when $p_n=o(\log n/\log \log n)$. Later on in \cite{adhikari_saha2017}, the same method was employed to obtain similar results for symmetric circulant, reverse circulant and Hankel matrices with Gaussian entries. Recently in \cite{Fred_raja_Generalized_gaussian_fluctuations_2024}, the same method was used to prove the convergence of LES for patterned matrices with correlated Gaussian entries. 
The study of LES of patterned random matrices with general entries is more challenging. Resolvent techniques, common in the study of LES of Wigner and related matrices, becomes less useful for most patterned matrices due to the high dependence between the rows. This makes moment method combined with combinatorics as the popular method to study the LES of patterned random matrix. 

For random patterned matrices, with general entries, the first result was by Liu, Sun and Wang \cite{liu2012fluctuations}, where they established the distributional convergence to Gaussian distribution for Toeplitz matrices with zero diagonal  under uniformly bounded moments condition on the entries using moment method and heavy combinatorics. 
The combinatorial techniques in \cite{liu2012fluctuations} are sensitive to the structure of the matrix, and the same idea might not work even under slight modifications of the matrix structure. 

In this paper, we develop a unified method for studying the LES of generalized patterned random matrices and extend the results in \cite{chatterjee2009fluctuations},\cite{adhikari_saha2017}, to generalized patterned random matrices with subexponential entries. To the best of our knowledge, this is the first work dealing with the LES of any type of patterned random matrix with general input entries and varying test functions. In the next two subsections, we introduce the random matrix model we consider, called the generalized patterned matrices, and recall the basics of linear eigenvalue statistics. In Section \ref{subsec:main_results}, we discuss the main results of this paper and Section \ref{subsec:main_appl} provides some applications of the main theorems. An overview of our methods is given in Section \ref{subsec:overview}.


\subsection{ \large Generalized patterned  random matrices}

In this work, we look at sequences of symmetric random matrices with increasing dimension constructed in the following way: 
\vskip2pt
\noindent\textbf{Setup:} For $d \in \N$, consider an input sequence $\{x_i : i \in \Z^d\}$ and a sequence of symmetric link functions
\begin{equation}\label{eqn: link_fun}
	L_{n}:\{1,2, \ldots, N(n)\}^{2} \rightarrow \mathbb{Z}^{d}, n \geq 1,
\end{equation}
where $N: \mathbb{N} \rightarrow \mathbb{N}$ is a strictly increasing function. We consider the sequence of random matrices $A_n$ of dimension $N(n) \times N(n)$ of the form
\begin{equation}\label{eqn:A_ij_general}
	(A_n)_{ij}=\begin{cases}
		x_{L_n(i, j)} &\text{ if } (i,j) \in \Delta_n, \\
		0 &\text{ otherwise},
	\end{cases}
\end{equation}
where $\Delta_n$ is a sequence of subsets of $[N(n)]^2$ such that $(i,j) \in \Delta_n$ implies $(j,i) \in \Delta_n$.
We call a matrix of the form (\ref{eqn:A_ij_general}) as a \textit{generalized patterned}  random matrix. Here $[N]$ denotes the set $\{1,2, \ldots,N\}$.

For notational convenience, we shall suppress the dependence of $n$ on $N(n)$ and write $N$ to denote the dimension of the matrix $A_n$. We urge the readers to keep in mind that in all cases considered in this paper, $N$ always depends on $n$ and $N \rightarrow \infty$ as $n \rightarrow \infty$.

One of the motivations for this generalization is to bring well-known variations of patterned random matrices into the same framework.
 We present a few well-known models of random matrices which fall into this class.
\begin{enumerate}[(i)]
	
\item \textbf{Band patterned random matrices:} For a patterned random matrix, two types of banding can be performed. Choosing $\Delta_n =\{(i,j):i-b_n \leq j \leq i+b_n , i \in [N]\}$ gives the banding along the diagonal with bandwidth $b_n$, and such matrices are called band matrices. Special cases of band matrices include the tridiagonal matrices and pentadiagonal matrices. 
	Choosing $\Delta_n =\{(i,j):n+1-b_n \leq i+j \leq n+1+b_n , i \in [N]\}$ gives banding along the anti-diagonal, and such matrices are called anti-diagonal band matrices.
	
	\item \textbf{Block patterned random matrices:} 
	Let $\widetilde{L}: \{1,2,\ldots, M \}^2 \rightarrow \Z^d$ be a function and $\{B_{j} : j \in \Z^d \}$ be a sequence of independent symmetric patterned random matrices of dimension $m \times m$. Define a block patterned matrix as the $mM \times mM$ matrix with block structure
	\begin{equation}\label{eqn:defn_block_patterned}
		A= \sum_{i,j=1}^M E_{i,j} \otimes B_{\widetilde{L}(i,j)},
	\end{equation}
	where $\otimes$ denotes the tensor product, $E_{i,j}$ is the $m \times m$ matrix with $1$ at $(i,j)$ and zero everywhere else, and $\mathbf{1}$ denotes the indicator function. We emphasize that $A$ can be written as a symmetric patterned random matrix with link function depending on the function $\widetilde{L}$ and the link functions of $B_i$.
	
	For a concrete construction, consider the following example: Let $\{B_n: n \geq 0\}$ be a sequence of independent Wigner matrices of dimension $m$. Suppose for each $n \geq 0$, the input sequence of $B_n$ is given by $\{x_{n,i,j}:i,j \in \Z\}$. We construct a block Toeplitz matrix of dimension $N \times N$, for $N=mM$, as the matrix whose blocks are given by $(T_n)_{ij}= B_{|i-j|}$ and the corresponding link function is given by 
	$$L_n\left(x,y\right)= \left(|r_1-r_2|, \min \{s_1,s_2\}, \max\{s_1,s_2\}\right),$$
	where $x=(r_1-1)m+s_1$ and $y=(r_2-1)m+s_2$.
	

	\item  Let $\{x_i: i \in \Z^2\}$ be a sequence of independent Bernoulli random variables with parameter $p$, $L_n(i,j)=(\min(i,j),\max(i,j))$ and $\Delta_n=\{(i,j):1  \leq i \neq j \leq n\}$. Then $A_n=(x_{L_n(i,j)}\one_{\Delta_n}(i,j))$ is the adjacency matrix of the Erd\H{o}s-R\'enyi random graph of $n$ vertices with edge-probability $p$. 
	The adjacency matrices of generalizations of Erd\H{o}s-R\'enyi random graphs such as stochastic block models and inhomogeneous Erd\H{o}s-R\'enyi random graph also fall into this category. Further, the unsigned adjacency matrix of the \textit{Linial-Meshulam complex} \cite{Leibzirer_Rosenthal_Sparse_RSC_2022}, a higher-dimensional generalization of the Erd\H{o}s-R\'enyi random graph also falls into this class. 
		
	\item Another important subclass is the \textit{hollow version} of patterned random matrix obtained by choosing $\Delta_n=\{(i,i):1 \leq i \leq N\}^c$, i.e. the diagonal entries are chosen to be zero. Modification of patterned matrices to sparse matrices, skyline matrix and block matrices with zero blocks are also of the form \eqref{eqn:A_ij_general}.
\end{enumerate}

In this paper, we consider sequences of generalized patterned random matrices with link functions satisfying the following assumption:
\begin{assumption}\label{Condition B}
	Let $N: \N \rightarrow \N$ be an increasing function. A sequence of link functions $\{L_n:[N(n)] \rightarrow \Z^d\}_{ n \in \N}$ is said to satisfy Assumption \ref{Condition B} if there exists a constant $B < \infty$ such that
	\begin{equation}\label{eqn:propertyB}
		\sup _n \sup _{t \in \mathbb{Z}^d} \sup _{1 \leq k \leq N(n)} \#\{l: 1 \leq l \leq N(n), L_n(k, l)=t\} \leq B.
	\end{equation}
	A single link function $L:[N] \rightarrow
	\Z^d$ is said to satisfy Assumption \ref{Condition B} if the constant sequence $(L_n=L \text{ for all } n)$ obeys (\ref{eqn:propertyB}).
\end{assumption}
Note that all patterned matrices discussed till now obey Assumption \ref{Condition B}. For example, $B=1$ for symmetric circulant matrix and reverse circulant matrix, and $B=2$ for symmetric Toeplitz matrix and Hankel matrix.

\subsection{ \large Linear eigenvalue statistics}
For an $n \times n$ matrix $A_n$, and a test function $\phi$, the \textit{linear eigenvalue statistics} (LES) of $A_n$ is defined as 
\begin{equation} \label{eqn:LES}
	\mathcal{A}_n(\phi)= \sum_{i=1}^{n} \phi(\lambda_i),
\end{equation} 
where $\lambda_1 , \lambda_2 , \ldots , \lambda_n$ are the eigenvalues of $A_n$. Note that if $A_n$ is a random matrix, then $\mathcal{A}_{n}(\phi)$ is a random variable. Hence, the obvious question arises about the fluctuations of $\mathcal{A}_{n}(\phi)$ in different modes of convergence.
The study of LES of random matrices focused on one particular matrix at a time, and the proof often does not follow through even under slight changes in the structure of the matrix.



This article provide a unified approach to study the LES of random matrices considered in \eqref{eqn:A_ij_general} for monomial test functions. Recall that for $n \in \N$, the dimension of the random matrix $A_n$ is $N(n)$, simply written as $N$. In order to study the fluctuations of LES of a random matrix $A_n/\sqrt{N}$, where $A_n$ is of the form \eqref{eqn:A_ij_general}, we consider the following scaled and centered version of the LES:
\begin{equation}\label{eqn:eta_p}
	\eta_{p_n}=\frac{\operatorname{Tr}\left(\frac{A_n}{\sqrt{N}}\right)^{p_n}-\E \operatorname{Tr}\left(\frac{A_n}{\sqrt{N}}\right)^{p_n}}{\sqrt{\var \left(\operatorname{Tr}\left(\frac{A_n}{\sqrt{N}}\right)^{p_n}\right)}}.
\end{equation}
Note that $\eta_{p_n}$ depends on $ \Delta_n$ and link function $L_n$. To keep the notations simpler, we are not including these dependencies in the notation of $\eta_{p_n}$. We request the readers to keep in mind these additional dependencies.

For stating our main result, we make the following assumption on the entries.

\begin{assumption}\label{assum: subexponential}
$\{x_i: i \in \Z^d\}$ are independent random variables with mean zero $\sup \E |x_i|^k \leq (\alpha k)^k$ for some constant $\alpha>0$, for all $k$.
\end{assumption}
Assumption \ref{assum: subexponential} is equivalent to saying that the input entries are subexponential with uniform moment bound. The class of subexponential distributions includes most of the well-studied heavy tailed distributions as well \cite{Vershynin_High_dim_prob_book}.

\subsection{Main results: LES for monomial test function}\label{subsec:main_results}
Now we state our main results.  The first main theorem states that to establish CLT for LES of generalized patterned random matrices for even degree monomial test functions, it is sufficient to check the convergence of its second moments.

\begin{theorem}\label{thm:converg_LES_even}
	Let $\left\{A_n=\left(x_{L_n(i, j)} \one_{\Delta_n}(i,j)\right)_{i, j=1}^N\right\}$ be a sequence of $N \times N$ symmetric random matrices and $(p_n)$ be a sequence of even positive integers such that
	\begin{enumerate}[(i)]
		\item the sequence of link function $\{L_n\}$ satisfies Assumption \ref{Condition B},
		\item the input sequence $\left\{x_i\right\}$ satisfies Assumption \ref{assum: subexponential},
		\item there exists a constant $c>0$ such that $c^{p_n}N \leq \var \left(\operatorname{Tr}\left(\frac{A_n}{\sqrt{N}}\right)^{p_n}\right)$ for all $n$.
	\end{enumerate} 
	Then,
	\begin{equation}
		\eta_{p_n} \stackrel{d}{\longrightarrow} N(0,1).
	\end{equation}
\end{theorem}

Our method also gives the convergence of LES for fixed $p$ case, even under further relaxation on the input entries. For fixed $p$, the condition on input entries can be relaxed to the following assumption.

\begin{assumption}\label{assum: 4 moments}
	$\{x_i: i \in \Z^d\}$ are independent random variables with mean zero and uniformly bounded moments of all order.
\end{assumption} 

The following is our main result for fixed $p$ case:
\begin{theorem}\label{thm:unified_moments_p_even}
	Suppose $\left\{A_n=\left(x_{L_n(i, j)} \one_{\Delta_n}(i,j)\right)_{i, j=1}^N\right\}$ is a sequence of $N \times N$ symmetric random matrices and $p$ is a fixed even positive integer such that 
	\begin{enumerate}[(i)]
		\item the sequence of link function $\{L_n\}$ satisfies Assumption \ref{Condition B},
		\item the input sequence $\left\{x_i\right\}$ satisfies Assumption \ref{assum: 4 moments},
		\item there exists a constant $c>0$ such that $cN \leq \var \left(\operatorname{Tr}\left(\frac{A_n}{\sqrt{N}}\right)^{p}\right)$.
	\end{enumerate} 
	Then,
	\begin{align*}
		\eta_p \stackrel{d}{\longrightarrow} N(0,1).
	\end{align*}
\end{theorem}

Theorems \ref{thm:converg_LES_even} and  \ref{thm:unified_moments_p_even} prompts an immediate inquiry: Is it possible to obtain an easily verifiable sufficient condition for condition (iii) of Theorem \ref{thm:converg_LES_even}? Our next theorem provides exactly that. To state the result,
we introduce the following convention: 
\begin{notation}
	For $n \in \N$ and a link function $L_n$, we define the following non-empty sequences of sets 
	\begin{itemize}
		\item $\mathcal{Z}_n=\{r_1,r_2,\ldots\}$ is any non-empty subset of $\operatorname{Im}(L_n)$, where $\operatorname{Im}(L_n)$ denotes the image set of $L_n$.
		\item For $\mathcal{Z}_n \subseteq \operatorname{Im}(L_n)$, consider $\mathcal{S}_n \subset \{(i,j) \in \Delta_n: L_n(i,j)=r_u \text{ for some } r_u \in \mathcal{Z}_n\}$ such that $(i,j) \in \mathcal{S}_n \implies (j,i) \in \mathcal{S}_n$.
	\end{itemize}
\end{notation}

The following theorem gives a sufficient condition based only on the link functions $\{L_n\}$ for the LES to converge to a normal distribution.

\begin{theorem}\label{thm: always normal}
	Suppose $\left\{A_n=\left(x_{L_n(i, j)}\one_{\Delta_n}(i,j)\right)_{i, j=1}^N\right\}$ is a sequence of $N \times N$ symmetric random matrices obeying conditions (i)-(ii) of Theorem \ref{thm:converg_LES_even} (or Theorem \ref{thm:unified_moments_p_even}). Additionally, assume that there exist constants $c_1,c_2>0$ and non-empty sets $\mathcal{Z}_n \subseteq \operatorname{Im}(L_n)$ such that
	\begin{enumerate}[(i)]
		\item for each $n \in \N$ and $r_u \in \mathcal{Z}_n$, $\# \{(i,j) \in \mathcal{S}_n:L_n(i,j)=r_u\} \geq c_1 N$,
		\item if for some row $R_i$ of the matrix $A_n$, $\mathcal{S}_n \cap R_i \neq \phi$, then $\#(\mathcal{S}_n \cap R_i) \geq c_2 N$. 
	\end{enumerate}
	Then there exists a constant $c>0$ such that $c^p N \leq \var \left(\operatorname{Tr}\left(\frac{A_n}{\sqrt{N}}\right)^{p}\right)$ for all $p$. As a consequence, $\eta_{p_n}$ defined in (\ref{eqn:eta_p}) converges in distribution to the standard Gaussian distribution for $p_n=o(\log N/ \log \log N)$ (respectively fixed $p$).
\end{theorem}

\begin{remark}\label{rem:normalization sqrt N}
	\noindent(a) For odd degree monomial test functions, a CLT result does not necessarily hold for general input entries. For example, consider a random Toeplitz matrix matrix with input sequence $\{x_i\}$ and the test function $\phi(x)=x$, then the LES is $x_0$.
	\vskip2pt
	\noindent(b) When $p$ is a fixed even positive integer, under the assumption that $\frac{1}{N}\var \left(\Tr \left(\frac{A_n}{\sqrt{N}}\right)^p\right)$ converges, and input entries obey Assumption \ref{assum: 4 moments},
	\begin{equation*}
		\zeta_p=\frac{\Tr\left(\frac{A_n}{\sqrt{N}}\right)^p-\E \Tr \left(\frac{A_n}{\sqrt{N}}\right)^p}{\sqrt{N}} \stackrel{d}{\longrightarrow} N(0,\sigma_p^2)
	\end{equation*}
	where $\sigma_p^2=\lim_{n \rightarrow \infty} \frac{1}{N}\var \left(\Tr \left(\frac{A_n}{\sqrt{N}}\right)^p\right)$. The existence of $\sigma_p^2$ is discussed in the Appendix. 
\end{remark}

Now, we deal with the LES of generalized patterned random matrices for odd degree monomial test function ($p$ is odd). We start by stating an assumption on the input sequence which is necessary to state the result for odd $p$ case.
\begin{assumption}\label{assump: x_i all moments}
	$\left\{x_i: i \in \Z^d\right\}$ are i.i.d. random variables with mean zero, variance one and $-\infty < \E(x_i^k)=m_k < \infty$ for all $k \geq 3$.
\end{assumption}

 For an $N \times N$ matrix $A_n$ and $p \in \Z_+$, define
\begin{equation*}
	\xi_p=\frac{\Tr\left(\frac{A_n}{\sqrt{N}}\right)^p-\E \Tr\left(\frac{A_n}{\sqrt{N}}\right)^p}{\sqrt{N}}.
\end{equation*}
For odd values of $p$, we have the following theorem for LES.
\begin{theorem}\label{thm:unified_moments_p_odd}
	Let $\left\{A_n=\left(x_{L_n(i, j)} \one_{\Delta_n}(i,j)\right)_{i, j=1}^N\right\}$ be a sequence of $N \times N$ symmetric random matrices and $p$ be an odd positive integer. Suppose 
	\begin{enumerate}[(i)]
		\item the sequence of link function $\{L_n\}$ satisfies Assumption \ref{Condition B},
		\item the input sequence $\{x_i: i \in \Z^d\}$ satisfies Assumption \ref{assump: x_i all moments},
		\item for each sentence $W \in SP(p,k)$, the limit $\theta(W)=\lim_{n \rightarrow \infty} \frac{1}{N^{\frac{pk+k}{2}}}\#\Pi_{L_n,\Delta_n}\left(W\right) $ exist.
	\end{enumerate} 
	Then the limiting moments of $\xi_p$ are given by
	\begin{align*}
		\lim _{n \rightarrow \infty} \mathbb{E}\left[\xi_p^{k}\right] &= \displaystyle \sum_{W \in SP(p,k)}   \prod_{i=1}^s m_{r_i} \theta(W_{C_i}),
	\end{align*}
	where $``s"$ is the number of clusters of $G_W$ (as defined in Definition \ref{defn_cluster}), $r_i=|C_i|$ is the size of the $i^{th}$ cluster, $m_{r_i}$ is as in Assumption \ref{assump: x_i all moments}, $W_{C_i}$ is as given in Definition \ref{defn: subsentence}, and $SP(p,k)$ as in Definition \ref{def: special partition}.
\end{theorem}

\begin{remark}
	\noindent (a) Note from Theorem \ref{thm:unified_moments_p_odd} the for odd $p$, the limiting moments are dependent on the moments of the input sequence. As a result, the LES of odd degree monomial test functions do not have a universal limit, provided $\theta(W_{C})$ is non-zero. For Toeplitz matrices, this result was proved in \cite{liu2012fluctuations}.
	\vskip2pt
	\noindent (b) In Theorem \ref{thm:unified_moments_p_odd}, only the limiting moments are calculated. To conclude convergence in distribution using moment method, it is additionally required to show that the limiting moment sequence uniquely determine a distribution \textit{(i.e. the moment sequence is M-determinate)}. But that does not hold for all patterned random matrices. For symmetric circulant matrix, the  limiting moment sequence is M-determinate, in fact the limit is Gaussian \cite{m&s_RcSc_jmp}. Whereas for symmetric Toeplitz matrix, the limiting moment sequence does not uniquely determine a distribution \textit{(M-indeterminate)} \cite{liu2012fluctuations}. Hence convergence in distribution cannot be concluded for all matrices of the form \eqref{eqn:A_ij_general}. In Theorem \ref{thm:unified_moments_p_odd}, if we additionally assume that the limiting moment sequence is M-determinate, say with distribution $\mathcal{F}$, then $\xi_p$ converges in distribution to  $\mathcal{F}$.  
\end{remark}

\subsection{Applications of our work}\label{subsec:main_appl}
 Linear eigenvalue statistics (LES) of different patterned random matrices have been studied separately in the literature, making use of their explicit trace formula. A drawback of this technique is that the combinatorial calculations are sensitive to changes in the structure of the matrices. In this paper, we attempt to build a unified framework for studying the LES of patterned random matrices for monomial (and easily generalized to polynomial) test functions, where this step is circumvented by building on the techniques and philosophy of \cite{bose_sen_LSD_EJP}. 
 
 The unified approach developed in this paper gives a variety of results for patterned random matrices. Some of the important cases, for the fixed even degree monomial test functions case are summarized in Table \ref{tab:Classical_new}.

\begin{table}[]
	\centering
	\begin{tabular}{|c|c|c|c|} 		
		\hline
		& \shortstack{Full Matrix \\$\Delta_n =[n]^2$} & \shortstack{Banding \\ $b_n/b \rightarrow c>0$}  & \shortstack{Anti-Banding \\ $b_n/b \rightarrow c>0$} \\ [2pt]
		\hline 
		Wigner & Girko \cite{Gikko_book} & Shcherbina \cite{shcherbina2015}   & Not Known \\ 
		Symmetric Toeplitz & Liu et al. \cite{liu2012fluctuations} &  Liu et al. \cite{liu2012fluctuations} & $\ast$ \\ 
		Palindromic Toeplitz & $\ast$  & $\ast$   & $\ast$  \\
		Generalized Toeplitz & $\ast$  & $\ast$   & $\ast$  \\
		Hankel &  Liu et al. \cite{liu2012fluctuations} & $\ast$   & Liu et al. \cite{liu2012fluctuations} \\
		Palindromic Hankel& $\ast$  & $\ast$   & $\ast$  \\
		Generalized Hankel & $\ast$  & $\ast$   & $\ast$  \\
		Symmetric circulant& Maurya and Saha \cite{m&s_RcSc_jmp} & $\ast$   & $\ast$  \\
		Reverse circulant& Maurya and Saha \cite{m&s_RcSc_jmp} & $\ast$   & $\ast$  \\
		\hline
	\end{tabular}
	\caption{Table showing different models of classical patterned matrices, along with reference to the first known proof of Gaussian fluctuation of the LES for fixed even degree monomial test functions. The symbol $\ast$ denotes the new results in this paper.}\label{tab:Classical_new}
\end{table}

\begin{table}[h!]
	\centering
	\begin{tabular}{||c|c|c||}
		\hline \hline
		& & \\
		\multirow{2}{10em}{$A=(B_{L_1(i,j)})_{i,j=1}^d$, $B_{\ell}=(x_{\ell,L_2(i,j)})_{i,j=1}^m$}
		& & \multirow{2}{8em}{$B_\ell$ is any other patterned matrix}\\
		& $B_\ell$ is Wigner &\\
		& & \vspace{-2mm}\\
		\hline \hline
		\vspace{-2mm}
		& &\\
		\shortstack{Block Wigner matrices \\ ($L_1(i,j)=(\min (i,j),\max(i,j))$)} & N.A. & $d$-finite, $m \rightarrow \infty$ \\
		&& \vspace{-2mm}\\
		\hline
		& &\\
		\shortstack{Other Block patterned matrices } & $m$-finite, $d \rightarrow \infty$ & All cases \\
		&&  \\
		\hline \hline
	\end{tabular}
	\caption{Table showing different models of block patterned random matrices and the cases when convergence of LES to Gaussian random variable occur. Here other patterned matrices include Toeplitz, Hankel and circulant-type matrices.}\label{table:block}
\end{table}


Block versions of patterned random matrices also important models of random matrices. The LSD of different block patterned matrices were studied in  \cite{Oraby_block_circulant},  \cite{Basu_Bose_Ganguly_Block_Toeplitz_LSD} and \cite{Murat_Kopp_Miller_Block_Circulant_2013}.
Analyzing the LES of block patterned random matrices presents additional challenges.
In particular, focusing on the explicit link function complicates the combinatorial calculations involved in determining the LES for these matrices.
 Block patterned matrices fit into the general framework described in this paper and we  verify Theorem \ref{thm: always normal} for several well-known block patterned matrices, namely, Block Wigner, Block Toeplitz, Block Hankel and Block versions of circulant-type matrices. 
A detailed discussion is given in see Section \ref{subsec:block_patterned_eg}.  
Table \ref{table:block} gives a concise form of the block patterned matrices for which Theorem \ref{thm:converg_LES_even} is applicable.

An important contribution of our work is the study of LES of patterned random matrices for varying test functions. This generalizes the results in \cite{chatterjee2009fluctuations} and \cite{adhikari_saha2017}, where Gaussian fluctuations was proved for different classical patterned random matrices with standard Gaussian entries. To the best of our knowledge, this is the first work dealing with the study of LES of patterned random matrices with general entries for varying test functions. 

Apart from these patterned and their block-version of matrices, our theorem can also be applied to other models of random matrices. $d$-disco, Checkerboard, Swirl, Hankel-type matrices are a few of them for which the conditions of Theorem \ref{thm: always normal} can be verified. For details on LES results of these matrices, see Section \ref{subsec:other_model_eg}. 

\subsection{Technical overview}\label{subsec:overview}

Now we briefly outline the structure of this manuscript. In Section \ref{sec:Prelim}, we introduce the combinatorial preliminaries required in this paper. In Section \ref{sec:LES_ind}, we provide a sketch of the central idea used to deal with the LES of any generalized patterned random matrix, and then we prove Theorems \ref{thm:converg_LES_even}, \ref{thm:unified_moments_p_even}, \ref{thm: always normal} and \ref{thm:unified_moments_p_odd} in separate subsections. This is done by assuming the required preliminary results, which we prove in the succeeding sections. In Section \ref{sec:proof of Bpl_prop}, we prove Propositions \ref{prop: W_k,p size 2} and \ref{lemma: Bpl general} along with the lemmas required to prove them. Finally, in Section \ref{sec:LES+example_ind}, we discuss the applications of our main theorem, proving LES results for several patterned random matrices. The convergence of LES for the fixed $p$ case is discussed in Appendix.

The logical structure of this paper is captured in the following diagram.

\vskip7pt

\begin{tikzpicture}[node distance=2cm]
	\node (lem+Qt) [startstop] {Lemma \ref{lemma:choices pi t}, Lemma \ref{lemma:choice of pi_t reordered}};
	\
	\node (lem+ste1a) [startstop, above of=lem+Qt,  xshift=7cm] {Lemma \ref{prop: Pi_L F-R}};
	
	\node (lem+ste1) [startstop, right of=lem+Qt,  xshift=5cm] {Lemma \ref{prop:Pi C_k,p}};
	
	\node (lem+ste2) [startstop, right of=lem+ste1,  xshift=2cm] {Lemma \ref{prop:existence_Type1subg}};
	
	\node (pro1) [io, below of=lem+Qt] {Proposition \ref{prop: W_k,p size 2}};
	\node (pro2) [io, right of=pro1,  xshift=7cm] {Proposition \ref{lemma: Bpl general}};
	\node (thm4) [decision, below of=pro1,  yshift=-1.5cm,] {Theorem \ref{thm: always normal}};
	\node (thm12) [decision, right of=pro1, yshift=-3.5cm, xshift=3cm] {Theorems \ref{thm:converg_LES_even}, \ref{thm:unified_moments_p_even}, \ref{thm:unified_moments_p_odd}};
	

	\draw [arrow] (lem+Qt) -- (pro1);
	\draw [arrow] (pro1) -- (thm12);
	\draw [arrow] (pro1) -- (thm4);
	
	\draw [arrow] (lem+ste1a) -- (lem+ste1);
	\draw [arrow] (lem+ste1) -- (pro2);
	\draw [arrow] (lem+ste2) -- (pro2);

	\draw [arrow] (pro2) -- (thm12);
	
	\label{fig:Overview}
	
\end{tikzpicture}

\section{\large \bf \textbf{Preliminaries}}\label{sec:Prelim}

\subsection{Words, sentences and clusters} \label{sec:words}
In this section, we introduce essential combinatorial objects necessary through out the paper.
\begin{definition}\label{defn:word}
	Let $[p]=\{1,2,\ldots,p\}$. We can label partitions of $[p]$ by words of length $p$ such that the first occurrence of each letter is in alphabetical order. The set of all words of length $p$ is denoted by $\WW_p$.
	
	 For example, for $p=5$, the partition $\{\{1,3,5\},\{2,4\}\}$ is represented by the word \textit{ababa}. For a word $w$, the number of distinct letters of $w$ is denoted by $|w|$.
	
	A word $w \in \WW_p$ is called a pair-partition if each letter in $w$ repeats twice. The set of all pair-partitions of $[p]$, is denoted by $\mathcal{P}_2(p)$. Note that $\mathcal{P}_2(p)$ is an empty set if $p$ is odd.
\end{definition}
	Now, we introduce the combinatorial notion of a sentence.

\begin{definition}\label{def:sentence}
	For $k \geq 1$, a partition of $\{1,2,\ldots, k\} \times \{1,2,\ldots,p\}$ is represented by a sentence with $k$ words each of length $p$, by assigning different letters to different blocks of the partition so that the first occurrence of the letters in the sentence is in the alphabetical order. In particular, the $i$-th word of a sentence represent the elements in $\{i\} \times [p]$. For example, the partition $\{\{(1,1),(2,1)\},\{(1,2),(1,4)\},\{(1,3),(2,4)\},\{(2,2),(2,3)\}\}$ of $\{1,2\} \times \{1,2,3,4\}$, can be represented as $(abcb,addc)$. The set of all sentences of length $k$, with each word having length $p$, is denoted by $\mathcal{W}_{p,k}$. For $k=0$, we take $\mathcal{W}_{p,k}$ as the empty set and for $k=1$, $\mathcal{W}_{p,k}=\mathcal{W}_p$. 
\end{definition}


 Next, we recall the definition of multiset. 
\begin{definition}\label{defn:multiset}
	A multiset is a tuple $S=(A,m_S)$ where $A$ is a set and $m_S: A \rightarrow \Z_+$ is called the \textit{multiplicity} function. $A$ is the set of distinct elements in the multiset and for $a \in A$, $m_S(a)$ gives the number of occurrences of $a$ in the multiset. By a slight abuse in notation, we shall use $a \in S$ to denote that an element $a$ belongs to the underlying set.
	Throughout this paper, for a given multiset $S$ and $a \in S$, we shall use $m_S(a)$ to denote the multiplicity of $a$ in $S$. By convention, $m_S(a)=0$ if $a \notin S$.
\end{definition}
Consider the following operation defined on multisets: For multisets $S_1=(A_1,m_{S_1})$ and $S_2=(A_2,m_{S_2})$, we define the multiset $S_1 \uplus S_2= \left(A_1 \cup A_2, m_{S_1}+m_{S_2}\right)$, where $\cup$ denotes the usual set union.
\begin{definition}\label{defn:SW}
	For a word $w \in \WW_p$, we define by $S_w$ the multiset of letters in $w$. For a sentence $W=(w_1,w_2,\ldots ,w_k) \in \WW_{p,k}$, we define 
	\begin{equation*}
		S_W= S_{w_1}\uplus S_{w_2}\uplus\cdots \uplus S_{w_k}.
	\end{equation*}
		For a sentence $W=(w_1,w_2,\ldots ,w_k)$, a letter $`x$' is called a cross-matched letter if $x \in S_{w_i}$ for at least two values of $i$. 
\end{definition}
Next, we define the notion of sub-sentences.
\begin{definition}\label{defn: subsentence}
	For a sentence $W=(w_1,w_2,\ldots,w_k) \in \WW_{p,k}$ and set $\CC \subset [k]$, the sub-sentence corresponding to $\CC$ is defined as $W_{\CC}=(w_i)_{i \in \CC}$.
\end{definition}

The following partitions of $[2] \times [p]$ are common in the study of LES of patterned random matrices \cite{liu2012fluctuations}. Here, we introduce the partitions in terms of sentences.
\begin{definition}\label{defn: P_2,2 , P2,4}
	We consider the following type of sentences:
		\vskip2pt
		\noindent (a) For a positive integer $p$, we denote a subset of $\WW_{p,2}$ by $\mathcal{P}_2(p, p)$, which consists of such sentences $W=(w_1,w_2) \in \WW_{p,2}$ such that
		\begin{enumerate}[(i)]
			\item $m_{S_W}(x)=2$ for all $x \in S_W$,
			\item $S_{w_1} \cap S_{w_2} \neq \phi$.
		\end{enumerate}
		\vskip2pt
		\noindent (b) For even positive integer $p$, we denote a subset of $\mathcal{P}_2(2p)$ by $\mathcal{P}_{2,4}(p, p)$, which consists of all sentences $W=(w_1,w_2) \in \WW_{p,2}$ satisfying the following conditions:
		\begin{enumerate}[(i)]
			\item $m_{S_W}(x)=4$ for some $x \in S_W$ and $m_{S_W}(y)=2$ for all $y \neq x$,
			\item the multiset intersection of $S_{w_1}$ and $S_{w_2}$ is $\{x,x\}$.
		\end{enumerate}
		For odd $p$, we assume $\mathcal{P}_{2,4}(p, p)$ is an empty set.
\end{definition} 

Now, we define the notion of a graph associated with sentences.

\begin{definition}\label{defn_cluster}
	Let $W=(w_1,w_2,\ldots,w_k) \in \mathcal{W}_{p,k}$ be a sentence. We  construct a simple graph $G_{W}=(V_{W},E_{W})$ associated with $W$ in the following fashion: We define $V_{W}=\{ 1, 2,\ldots , k\}$ and for $r \neq s$, $\{r,s\} \in E_{W}$ if $S_{w_r} \cap S_{w_s} \neq \phi$. A maximal connected subgraph of $G_W$ is called a \textit{cluster}.
\end{definition}
Note that for any sentence $W$ and any cluster $C$ of $G_W$, $G_{W_{V(C)}}=C$. Thus, without loss of generality, we shall represent cluster as the graph of a sentence $W^\prime$ such that $G_{W^\prime}$ is connected.
We now define a special type of cluster called clique cluster.
\begin{definition}\label{defn:special_cluster}
	For $p \geq 1$, $k \geq 3$ and $W \in \mathcal{W}_{p,k}$, the graph $G_W$ is called a \textit{clique cluster} if
	\begin{enumerate}[(i)]
		\item $G_W$ is a complete graph, and
		\item there exists a letter $x \in S_W$ such that $S_{w_i} \cap S_{w_j}=\{x\}$ for all distinct $i,j$.
	\end{enumerate}
	A sentence $W \in \mathcal{W}_{p,k}$ is called a \textit{clique sentence} if $G_W$ is a clique cluster, the common letter appears exactly $k$ times and every other letter appears exactly twice in $S_W$. 
\end{definition}

\begin{definition}\label{def: special partition}
	For $p \geq 1$ and $k \geq 3$, a sentence $W \in \mathcal{W}_{p,k}$ is called a \textit{special partition} if for each cluster $C$ of $G_W$,
	\begin{enumerate}[(i)]
		\item $|V(C)|\geq 2$,
		\item if $|V(C)| =2$, then $W_C$ belongs to either $\PP_2(p,p)$ or $\PP_{2,4}(p,p)$, and
		\item  if $|V(C)| \geq 3$, then $W_C$ is a clique sentence,
	\end{enumerate}
	where $V(C)$ denotes the vertex set of $C$.
	The set of all special partitions in $\mathcal{W}_{p,k}$ is denoted by $SP(p,k)$.
\end{definition}
Observe the following remark from the definition of $SP(p,k)$.
\begin{remark}\label{rem:W_in_SP_peven}
	\noindent (a) For even $p$, it can be seen that for any $W \in SP(p,k)$, all clusters of $G_W$ has size 2. To see this, note that if $C$ is a cluster of size greater than two, then $W_C$ is a clique sentence with common letter $`c$', say. For a word $w_i$, $i \in V(C)$, it follows that $`c$' appears exactly once in $w_i$ and all other letters must appear exactly twice in $w_i$. Note that this is possible only when $p$ is odd.\\
\noindent (b) As a consequence of (a), it follows that for $p$ even and $k$ odd, $SP(p,k)$ is empty.
\end{remark}
\subsection{Circuits}Next, we define the notion of circuits.
Recall the definition of link function $L: [N] \rightarrow \Z^d$, defined in \eqref{eqn: link_fun}. 
\begin{definition}\label{defn:circuit}
	Consider a link function $L: [N] \rightarrow \Z^d$, where $d \in \N$. A function $\pi:\{0,1,2, \ldots, p\} \rightarrow[N]$ with $\pi(0)=\pi(p)$ is called a \textit{circuit} of length $p$. Two circuits of length $p$, $\pi_1$ and $\pi_2$ are said to be equivalent if their $L$-values match at the same locations. That is, for all $1 \leq i, j \leq p$, 
	\begin{align}
		L(\pi_1(i-1), \pi_1(i)) =L(\pi_1(j-1), \pi_1(j)) 
		\Longleftrightarrow  L(\pi_2(i-1), \pi_2(i)) =L(\pi_2(j-1), \pi_2(j)).
	\end{align}
\end{definition}
The above condition gives an equivalence relation on the set of all circuits of length $p$. It is immediate that any equivalence class can be indexed by a partition of $\{1,2, \ldots, p\}$ where each block of the given partition identifies the positions where the $L$-matches take place. Therefore, each equivalence class can be uniquely expressed by a word defined in Definition \ref{defn:word}. For example, if $p=5$, then the partition $\{\{1,3,5\},\{2,4\}\}$ is represented by the word \textit{ababa}. This partition identifies all circuits $\pi$ for which $L(\pi(0), \pi(1))=L(\pi(2), \pi(3))=L(\pi(4), \pi(5))$ and $L(\pi(1), \pi(2))=L(\pi(3), \pi(4))$.


Let $w$ be a word and let $w[i]$ denote the $i$-th letter of $w$.
For a link function $L$ and a word $w$, we denote the equivalence class of circuits corresponding to $w$ by
\begin{equation*}
	\Pi_L(w)=\left\{\pi: w[i]=w[j] \Leftrightarrow L(\pi(i-1), \pi(i))=L(\pi(j-1), \pi(j))\right\}.
\end{equation*}
Similarly, for a link function $L$ and a sentence $W=(w_1,w_2,\ldots ,w_k ) \in \WW_{p,k}$, we define 
\begin{equation}\label{defn: Pi L(W)}
	\Pi_L(W)=\{(\pi_1,\pi_2,\ldots ,\pi_k ): w_r[i]=w_s[j] \Leftrightarrow L(\pi_r(i-1), \pi_r(i))=L(\pi_s(j-1), \pi_s(j))\}.
\end{equation}
Unfortunately, $\# \Pi_L(W)$ is harder to compute and therefore for convenience in computation, we define a slightly larger set $\Pi_L^*(W)$ given by
$$	\Pi_L^*(W)=\{(\pi_1,\pi_2,\ldots ,\pi_k ): w_r[i]=w_s[j] \implies L(\pi_r(i-1), \pi_r(i))=L(\pi_s(j-1), \pi_s(j))\}.$$
It easily follows that if $(\pi_1,\pi_2,\ldots, \pi_k) \in \Pi_L(W)$, then  $(\pi_1,\pi_2,\ldots, \pi_k) \in \Pi_L^*(W)$ and therefore $\#\Pi_L(W) \leq \#\Pi_L^*(W)$ for every sentence $W$. 
For a set $\Delta \subset [N]^2$, we define the following sets 
\begin{align}\label{eqn: Pi L,D(W)}
	\Pi_{L,\Delta}(W) &=\{(\pi_1,\pi_2,\ldots ,\pi_k ) \in \Pi_{L}(W) :  (\pi_u(i-1), \pi_u(i)) \in \Delta \ \forall \ 1 \leq u \leq k , 1 \leq i \leq p \}, \nonumber \\
	\Pi^*_{L,\Delta}(W)&=\{(\pi_1,\pi_2,\ldots ,\pi_k ) \in \Pi^*_{L}(W) :  (\pi_u(i-1), \pi_u(i)) \in \Delta \ \forall \ 1 \leq u \leq k , 1 \leq i \leq p \}. 
\end{align}
Note that in \eqref{eqn: Pi L,D(W)}, the circuits are maps from $\{0,1,\ldots, p\}$ to $[N]$ and hence the cardinality of $\Pi_{L,\Delta}$ depends on $p$ and $N$ as well. To make the notations simpler, we are suppressing this dependence in the notation.

In general, for a sequence of link functions (given by $L_ n$) and a sequence of subsets $(\Delta_n)$, we are interested in finding the cardinality of $\Pi_{L_n,\Delta_n}(W)$. The following inequality directly follows from \eqref{eqn: Pi L,D(W)} and shall play an important role in computations at later stages.
\begin{equation}\label{eqn: Pi Ln,D_n inequality}
	\# \Pi_{L_n, \Delta_n}(W) \leq \# \Pi_{L_n}(W) \leq \# \Pi_{L_n}^*(W) \mbox{ for all $W \in \WW_{p,k}$ and $\Delta_n \subseteq [N]^2$}.
\end{equation}

Let $W$ be a sentence such that $G_W$ has clusters $C_1,C_2,\ldots, C_m$ with $\# V(C_i)=k_i$. Then we have
\begin{equation}\label{eqn: PiL star product}
	\frac{1}{N^{\frac{pk+k}{2}}} \#\Pi_{L_n, \Delta_n}^*(W) = \prod_{i=1}^{m}\frac{1}{N^{\frac{pk_i+k_i}{2}}} \#\Pi_{L_n,\Delta_n}^*(W_{C_i}), 
\end{equation}
where $W_{C_i}$ is as defined in Definition \ref{defn: subsentence}.

\section{\large \bf \textbf{LES for independent entries}} \label{sec:LES_ind}
In this section, we present a unified technique to derive the limiting moment sequence of the LES of any symmetric generalized patterned random matrix for monomial test functions. In Section \ref{sec:gen+setup_ind}, we sketch a unified technique used to deal with the limiting moment sequence of the LES of generalized patterned random matrices and in Section \ref{subsec:even_p}, we prove Theorems \ref{thm:converg_LES_even} and \ref{thm:unified_moments_p_even}. Finally, in  Sections \ref{subsec:odd_p} and \ref{sec:LES_nestlink+ind}, we prove Theorems \ref{thm:unified_moments_p_odd} and \ref{thm: always normal}, respectively.
\subsection{\large \bf \textbf{General setup for independent entries:}} \label{sec:gen+setup_ind}
In this section, we present the unifying thread common in the proofs of Theorems \ref{thm:converg_LES_even}, \ref{thm:unified_moments_p_even} and \ref{thm:unified_moments_p_odd}. Let $\eta_p$ be as defined in (\ref{eqn:eta_p}). First, we discuss how $\eta_p$ can be connected to the sets $\Pi_L(W)$ defined in (\ref{defn: Pi L(W)}). Recall the definition of circuit from Definition \ref{defn:circuit}. For the time being, in dealing with $\eta_p$ we work with the normalization $1/\sqrt{N}$ instead of $1/\varA$, as this does not significantly change the arguments. It follows from (\ref{eqn:eta_p}) that for all $k \in \N$,
\begin{equation}\label{eqn: eta p power k}
	\eta_{p_n}^k= \frac{1}{N^{\frac{p_nk+k}{2}}}\displaystyle \sum_{\pi_1,\ldots , \pi_k} \big[(x_{\pi_1}- \E x_{\pi_1})(x_{\pi_2}- \E x_{\pi_2}) \cdots (x_{\pi_k}- \E x_{\pi_k}) \hspace{-2mm}\prod_{u=1,2,\ldots , k \atop i=1,2,\ldots , p_n} \hspace{-2mm}\one_{\Delta_n}(\pi_u(i-1),\pi_u(i))\big],
\end{equation}
where for $u=1,2, \ldots,k$; $\pi_u:\{0,1,2,\ldots, p_n\} \rightarrow \{1,2,\ldots, N\}$ are circuits of length $p_n$ and $x_{\pi_u}$ is defined as  
$$x_{\pi_u}=x_{L_n(\pi_u(0),\pi_u(1))}x_{L_n(\pi_u(1),\pi_u(2))}\cdots x_{L_n(\pi_u(p_n-1),\pi_u(p_n))}.$$ 
The above summation can be split further as
\begin{align}\label{eqn:E w_p^k}
	\E (\eta_{p_n}^k) &=\frac{1}{N^{\frac{p_nk+k}{2}}}  \sum_{W \in \WW_{p_n,k}} \sum_{(\pi_1,\pi_2,\ldots ,\pi_k) \in \Pi_{L_n,\Delta_n}(W)} \mathbb{E} \big( \prod_{i=1}^k (x_{\pi_{i}}- \mathbb{E}x_{\pi_{i}}) \big),
\end{align}
where $\Pi_{L_n,\Delta_n}(W)$ is defined in \eqref{eqn: Pi L,D(W)}.

Suppose $\{x_i\}$ obeys one of Assumptions \ref{assum: subexponential}, \ref{assum: 4 moments} or \ref{assump: x_i all moments}. First we make the following observations.
\begin{enumerate}[(i)]
	\item If there exists $\pi_{t}$ which is not connected with any other $\pi_{s}$ (i.e. $\operatorname{Im}(\pi_t) \cap \operatorname{Im}(\pi_s)=\phi$ for all $s$), then due to the independence of the input sequence, the expectation term will be zero.
	\item The same conclusion is true also if a letter appears only once in the sentence $W$.
\end{enumerate} 
Note that for $W \in SP(p,k)$, both of the above conditions do not hold. Therefore, for proving Theorem \ref{thm:converg_LES_even}, it is sufficient to show that for $p_n=o(\log N / \log \log N)$, only sentences in $SP(p_n,k)$ contribute to $\lim_{n \rightarrow \infty} \E (\eta_{p_n}^k)$, i.e.,
\begin{equation}\label{eqn: Card Pi_L(W) zero}
	\theta_2= \frac{1}{N^{\frac{p_nk+k}{2}}} \sum_{W \notin SP(p_n,k)} \sum_{(\pi_1,\pi_2,\ldots ,\pi_k) \in \Pi_{L_n,\Delta_n}(W)}  \mathbb{E} \big( \prod_{i=1}^k (x_{\pi_{i}}- \mathbb{E}x_{\pi_{i}} ) \big)=o(1).	
\end{equation}
For the fixed $p$ case (Theorems \ref{thm:unified_moments_p_even} and \ref{thm:unified_moments_p_odd}), the calculations are simpler since $SP(p,k)$ is a finite set in this case. But for varying $p$, the size of $SP(p,k)$ also grows with $p$.


Observe that $ \mathbb{E} \big( \prod_{i=1}^k (x_{\pi_{i}}- \mathbb{E}x_{\pi_{i}} )\big)$ is uniformly bounded for all sentences $W \in \WW_{p,k}$ and therefore
\begin{equation*}
	|\theta_2| \leq \frac{1}{N^{\frac{p_nk+k}{2}}} \sum_{W \notin SP(p_n,k)} \sum_{\Pi_{L_n,\Delta_n}(W)}  |\mathbb{E} \big( \prod_{i=1}^k (x_{\pi_{i}}- \mathbb{E}x_{\pi_{i}} )\big)|   \leq	\frac{1}{N^{\frac{p_nk+k}{2}}} \sum_{W \notin SP(p_n,k)} M_W \#\Pi_{L_n, \Delta_n}^*(W),
\end{equation*}
where $M_W$ is a constant that depends only on the sentence $W$. Hence proving \eqref{eqn: Card Pi_L(W) zero} is reduced to proving that
\begin{equation}\label{eqn: Pi L_n,D_n theta 2 relation}
	\sum_{W \notin SP(p_n,k) \atop S_W(x) \neq 1 \forall x} M_{W}\# \Pi_{L_n, \Delta_n}^*(W)=o(N^{\frac{p_nk+k}{2}}).
\end{equation}
Observe that both $M_W$ and $\#\Pi_{L_n, \Delta_n}^*(W)$ depend on $p_n$, and we need appropriate bounds on them to establish \eqref{eqn: Pi L_n,D_n theta 2 relation}. We use the moment bound on the entries to obtain a bound on $M_W$ and the bound on $\#\Pi_{L_n, \Delta_n}^*(W)$ is achieved by the following two propositions.

\begin{proposition}\label{prop: W_k,p size 2}
Let $\{L_n:{[N(n)]^2 \rightarrow \Z^d}\}$ be a sequence of link functions obeying Assumption \ref{Condition B} with the constant $B$. Then
	\begin{enumerate}[(a)]
		\item  for every sentence  $W \in \mathcal{P}_2(p,p) \cup \mathcal{P}_{2,4}(p,p)$,
		\begin{equation}
			\frac{1}{N^{p+1}}\# \Pi_{L_n}(W) \leq  \frac{1}{N^{p+1}}\# \Pi_{L_n}^*(W) \leq B^{2p}.
		\end{equation}
		For all other $W\in \WW_{p,2} \setminus \left(\mathcal{P}_2(p,p) \cup \mathcal{P}_{2,4}(p,p)\right)$ with at least one cross-matched element, and each element in $S_W$ appearing at least twice,
		\begin{equation*}
			\frac{1}{N^{p+1}}\# \Pi_{L_n}^*(W) \leq \frac{B^{2p}}{\sqrt{N}}.
		\end{equation*}
		\item 	For $k \geq 3$, let $W \in \WW_{p,k}$ be a clique sentence. Then
		\begin{equation*}
			\frac{1}{N^{\frac{pk+k}{2}}}\# \Pi_{L_n}(W) \leq B^{pk},
		\end{equation*}
		where clique sentence is as defined in Definition \ref{defn:special_cluster}.
	\end{enumerate}
\end{proposition}

%

\begin{proposition}\label{lemma: Bpl general}
	Let $\{L_n\}$ be a sequence of link functions obeying Assumption \ref{Condition B} with the constant $B$. For $p \geq 1$ and $ k \geq 3$, let $W =(w_1,w_2,\ldots, w_k) \in \mathcal{W}_{p,k}$ be such that each letter in $S_W$ is repeated at least twice. If $G_W$ is connected but $W$ is not a clique sentence, then
\begin{equation*}
	\frac{1}{N^{\frac{pk+k}{2}}} \#\Pi_{L_n}^*(W) \leq \frac{B^{pk}}{\sqrt{N}},
\end{equation*}
where $B$ is the constant in Assumption \ref{Condition B}.
\end{proposition} 

Proposition \ref{lemma: Bpl general} is the backbone of this paper. The proof for Proposition \ref{lemma: Bpl general} is a bit long and involves some ideas from graph theory. For now, we shall skip the proofs of Propositions \ref{prop: W_k,p size 2} and \ref{lemma: Bpl general}, and move on to proving the main theorems. The proofs of these propositions 
are given in Section \ref{sec:proof of Bpl_prop}.

\subsection{\large \bf \textbf{Proof of Theorems \ref{thm:converg_LES_even} and \ref{thm:unified_moments_p_even}}:} \label{subsec:even_p}
With the help of the above propositions, we prove Theorem \ref{thm:converg_LES_even} using the method of moments and Wick formula. 

\begin{proof}[Proof of Theorem \ref{thm:converg_LES_even}]
	Let $(p_n)$ be a sequence of even positive integers. Recall  from (\ref{eqn:E w_p^k}) that
\begin{align*} 
	\E (\eta_{p_n}^k) 
	&= \frac{1}{N^{\frac{p_nk}{2}}\left(\varA\right)^{k/2}} \Big[ \sum_{W \in SP(p_n,k)} \sum_{\Pi_{L_n,\Delta_n}(W)} \mathbb{E} \big( \prod_{i=1}^k (x_{\pi_{i}}- \mathbb{E}x_{\pi_{i}} )\big) \\ 
	& \qquad \qquad \qquad \qquad+  \sum_{W \notin SP(p_n,k)} \sum_{\Pi_{L_n,\Delta_n}(W)} \mathbb{E} \big( \prod_{i=1}^k (x_{\pi_{i}}- \mathbb{E}x_{\pi_{i}} )\big)  \Big],
\end{align*}
where $(\pi_1,\pi_2,\ldots,\pi_k) \in \Pi_{L_n, \Delta_n}(W)$, $SP(p_n,k)$ is as defined in Definition \ref{def: special partition} and  
$$x_{\pi_i}=x_{L_n(\pi_i(0),\pi_i(1))}x_{L_n(\pi_i(1),\pi_i(2))}\cdots x_{L_n(\pi_i(p_n-1),\pi_i(p_n))}.$$ 

First, we look at the term 
\begin{equation}\label{eqn:contribution_SP(p,k)_even_1}
	\frac{1}{N^{\frac{p_nk}{2}}\left(\varA\right)^{\frac{k}{2}}} \times \sum_{W \notin SP(p_n,k)} \sum_{\Pi_{L_n,\Delta_n}(W)} \mathbb{E} \big( \prod_{i=1}^k (x_{\pi_{i}}- \mathbb{E}x_{\pi_{i}} )\big).
\end{equation}
Note that the summand in \eqref{eqn:contribution_SP(p,k)_even_1} is zero if there exists a letter in $S_W$ that appears only once. Further, by condition (iii) of Theorem \ref{thm:converg_LES_even}, \eqref{eqn:contribution_SP(p,k)_even_1} is bounded by
\begin{equation}\label{eqn:contribution_SP(p,k)_even_2}
	\frac{1}{c^{\frac{p_nk}{2}}N^{\frac{(p_n+1)k}{2}}} \sum_{W \notin SP(p_n,k)  \atop S_W(x) \neq 1 \forall x} \sum_{\Pi_{L_n,\Delta_n}(W)} \left|\mathbb{E} \big( \prod_{i=1}^k (x_{\pi_{i}}- \mathbb{E}x_{\pi_{i}}) \big) \right|.
\end{equation} 
 Further, since the input entries satisfy Assumption \ref{assum: subexponential}, it follows that for a fixed $k$,
\begin{equation}\label{eqn:moment_bound}
	 \left|\mathbb{E} \big( \prod_{i=1}^k (x_{\pi_{i}}- \mathbb{E}x_{\pi_{i}} )\big) \right| \leq \beta(\alpha p_nk)^{p_nk},
\end{equation}
where $\beta$ is a constant that does not depend on $p_n$.

Consider a sentence $W \notin SP(p_n,k)$ with clusters $C_1,C_2,\ldots,C_r$ of size $k_1,k_2,\ldots,k_r$, respectively, $k_i \geq 2$ for all $i$ and every letter appearing at least twice. Without loss of generality, suppose the subsentences $W_{C_1},W_{C_2},\ldots,W_{C_s}$ are elements of $\PP_{2}(p_n,p_n) \cup \PP_{2,4}(p_n,p_n)$ or a clique sentence. It then follows from \eqref{eqn: PiL star product} and Propositions \ref{prop: W_k,p size 2}, \ref{lemma: Bpl general} that
\begin{align*}
	\frac{1}{N^{\frac{p_nk+k}{2}}}\#\Pi_{L_n, \Delta_n}^*(W) &=\prod_{i=1}^r \frac{1}{N^{\frac{p_nk_i+k_i}{2}}}\# \Pi_{L_n, \Delta_n}^*(W_{C_i}) \\
	&\leq B^{p_n({k_1+k_2+\cdots+k_s})} \prod_{i=s+1}^{r}\frac{1}{N^{\frac{p_nk_i+k_i}{2}}}
	\# \Pi_{L_n, \Delta_n}^*(W_{C_i}) \\
	&\leq B^{p_n(k_1+k_2+\cdots+k_s)} \times \frac{B^{p_n(k_{s+1}+k_{s+2}+\cdots +k_r)}}{N^{\frac{r-s}{2}}} \leq \frac{B^{p_nk}}{\sqrt{N}},
\end{align*}
where the last inequality follows from the observation that $s<r$ for all sentences $W \notin SP(p_n,k)$. Therefore, it follows that \eqref{eqn:contribution_SP(p,k)_even_2} is bounded above by
\begin{equation}\label{eqn:contribution_SP(p,k)_even_3}
	\frac{1}{c^{\frac{p_nk}{2}} N^{\frac{(p_n+1)k}{2}}} \sum_{W \notin SP(p_n,k)  \atop S_W(x) \neq 1 \forall x} \#\Pi_{L_n,\Delta_n}(W) \beta(\alpha p_nk)^{p_nk} \leq \frac{\beta}{\sqrt{N}} B^{p_nk} (\alpha p_n k)^{p_nk} (p_nk)^{p_nk} \times \left(\frac{1}{c^{\frac{p_nk}{2}}}\right).
\end{equation}
Direct calculation implies that the right hand side of \eqref{eqn:contribution_SP(p,k)_even_3} is of the order of $o(1)$ when $p_n=o(\log N /\log \log N)$. This implies that 
\begin{equation}\label{eqn:Eeta_p^k2}
	\E (\eta_{p_n}^k) 
	= \frac{1}{N^{\frac{p_nk}{2}}} \times \frac{1}{\left(\varA\right)^{k/2}}\sum_{W \in SP(p_n,k)} \sum_{\Pi_{L_n,\Delta_n}(W)} \mathbb{E} \big( \prod_{i=1}^k (x_{\pi_{i}}- \mathbb{E}x_{\pi_{i}} )\big)+o(1).
\end{equation} 

Note from Remark \ref{rem:W_in_SP_peven} that $SP(p_n,k)$ is an empty set for $k$ odd and therefore 
\begin{equation*}
	\lim_{n \rightarrow \infty} 	\E (\eta_{p_n}^k) = 0
\end{equation*}
for odd $k$. Next, note from \eqref{eqn:Eeta_p^k2} that 
\begin{equation*}
	\E (\eta_{p_n}^2)= \frac{1}{N^{p_n}} \times \frac{1}{\varA}\sum_{W \in SP(p_n,2)} \sum_{\Pi_{L_n,\Delta_n}(W)} \mathbb{E} \big( \prod_{i=1}^2 (x_{\pi_{i}}- \mathbb{E}x_{\pi_{i}} )\big)+o(1).
\end{equation*}
By the definition of $\eta_{p_n}$, $\E (\eta_{p_n}^2)=1$ and therefore, we get
\begin{equation}\label{eqn:var=1}
	\frac{1}{N^{p_n}} \times \frac{1}{\varA}\sum_{W \in SP(p_n,2)} \sum_{\Pi_{L_n,\Delta_n}(W)} \mathbb{E} \big( \prod_{i=1}^2 (x_{\pi_{i}}- \mathbb{E}x_{\pi_{i}} )\big)=1+o(1).
\end{equation}

Next, we calculate the contribution of sentences $W \in SP(p_n,2k)$. It follows from Remark \ref{rem:W_in_SP_peven} that each cluster of $W$ belongs either to $\PP_{2}(p_n,p_n)$ or $\PP_{2,4}(p_n,p_n)$. Note from \eqref{eqn:Eeta_p^k2} that

\begin{equation*}
	\mathbb{E}\left(\eta_{p_n}^{2k}\right) = \frac{1}{N^{p_nk}} \frac{1}{\left(\varA\right)^{k}} \sum_{W \in SP(p_n,2k)} \sum_{\Pi_{L_n,\Delta_n}(W)} \mathbb{E} \left( \prod_{i=1}^{2k} (x_{\pi_{i}}- \mathbb{E}x_{\pi_{i}})\right) +o(1).
\end{equation*}

Note that for even $p_n$, each sentence in $SP(p_n,2k)$ is composed of $k$ sub-sentences each made of two words. Suppose that the sentences $W_{1},W_{2},\ldots, W_{k}$ are given, with $r$ sentences in $\PP_{2}(p_n,p_n)$ and the rest $(k-r)$ sentences belonging to $\PP_{2,4}(p_n,p_n)$. We aim to enumerate the number of sentences $W \in SP(p_n,2k)$ such that sentences $W_{1},W_{2}\ldots, W_{k}$ are the sub-sentences corresponding to the clusters of $W$. Note that the number of ways of choosing $k$ clusters of size two is $\#\PP_2(2k)$, the number of pair-partitions of $[2k]$. And the number of ways of choosing $r$ clusters among this particular choice of $k$ clusters such that the corresponding sub-sentences are elements of $\PP_{2}(p_n,p_n)$ is ${k \choose r}$. Once the above are fixed, we can uniquely determine the sentence $W \in SP(p_n,2k)$, by imposing the condition that different clusters do not have the same letters.
This implies that
\begin{align}\label{eqn: E eta p k choose r}
	&\mathbb{E}\left(\eta_{p_n}^{2k}\right) =	\frac{1}{N^{p_nk}}\frac{1}{\left(
	\varA\right)^k}  \#\mathcal{P}_2(2k)\sum_{r=0}^k \hspace{-2mm} \sum_{W_1, \ldots , W_r \in \mathcal{P}_2(p_n,p_n) \atop W_{r+1},\ldots ,W_k \in \mathcal{P}_{2,4}(p_n,p_n) } \nonumber\\ & \hspace{60mm} \sum_{\Pi_{L_n,\Delta_n}(W)} {k \choose r} \mathbb{E} \left( \prod_{i=1}^{2k} x_{\pi_i}- \mathbb{E}x_{\pi_i}\right)+o(1).
\end{align}

Before proceeding further, we make two observations:
\begin{enumerate}[(a)]
	\item Consider a sentence $W$ with subsentences $W_1,W_2,\ldots, W_k$. Then, note that $\Pi_{L_n,\Delta_n}(W)$ is a larger set compared to $\Pi_{L_n, \Delta_n}(W_1) \times \Pi_{L_n,\Delta_n}(W_2) \times \cdots \times \Pi_{L_n, \Delta_n}(W_k)$. We also have
	\begin{equation*}
		\Pi_{L_n, \Delta_n}(W) \setminus \left(\times_{i=1}^k \Pi_{L_n, \Delta_n}(W_i)\right)= \bigcup_{\tilde{W} \neq W} \Pi_{L_n,\Delta_n}(\tilde{W}),
	\end{equation*}
	where the union is over all words $\tilde{W}=(\tilde{w}_1,\tilde{w}_2,\ldots,\tilde{w}_k) \neq W$, such that $w_r[i]=w_s[j] \implies \tilde{w}_r[i]=\tilde{w}_s[j]$. In particular, note that such sentences $\tilde{W}$ are not elements of $SP(p_n,k)$ and therefore by \eqref{eqn:contribution_SP(p,k)_even_3}, the contribution due to such sentences is of the order $o(1)$.
	
	\item Suppose the vertex set of $W_i$ is $\{i_1,i_{2}\}$. Using the independence of the input entries $\{x_i\}$, we get that
	\begin{equation*}
		\mathbb{E} \left( \prod_{i=1}^{2k} \left(x_{\pi_i}- \mathbb{E}x_{\pi_i}\right)\right)=\prod_{i=1}^k \mathbb{E}\left((x_{\pi_{i_{1}}}-\E x_{\pi_{i_1}})(x_{\pi_{i_2}}-\E x_{\pi_{i_2}})\right).
	\end{equation*}
\end{enumerate}
Combining the above observations, we get that 
\begin{align*}
&\frac{1}{N^{p_n k}\left(\varA\right)^k} \sum_{\Pi_{L_n,\Delta_n}(W)} \E \left( \prod_{i=1}^{2k} \left(x_{\pi_i}-\E x_{\pi_i}\right) \right) \\ &=\frac{1}{N^{p_nk}\left(\varA\right)^k} \prod_{i=1}^k \sum_{(\pi_{i_1},\pi_{i_2}) \in \Pi_{L_n, \Delta_n}\left(W_i\right)} \mathbb{E} \left( \prod_{j=1,2} (x_{{\pi_{i_j}}}- \mathbb{E}x_{\pi_{i_j}}\right)+o(1) .
\end{align*}

Therefore (\ref{eqn: E eta p k choose r}) can be written as
\begin{align*}
	&\mathbb{E}\left(\eta_{p_n}^{2k}\right)  \\
		&=\frac{\#\mathcal{P}_2(2k)}{N^{p_nk}\left(\varA\right)^k}\displaystyle\sum_{r=0}^k  \sum_{W_1, \ldots , W_r \in \mathcal{P}_2(p_n,p_n) \atop W_{r+1},\ldots ,W_k \in \mathcal{P}_{2,4}(p_n,p_n) }  {k \choose r}  \prod_{ i=1}^k \sum_{\Pi_{L_n,\Delta_n}(W_i)} \mathbb{E} \big( \prod_{j=1}^2 (x_{\pi_{i_j}}- \mathbb{E}x_{\pi_{i_j}} )\big)  +o(1) \\
		  &=\frac{\#\mathcal{P}_2(2k)}{N^{p_nk}\left(\varA\right)^k}\displaystyle\sum_{r=0}^k   {k \choose r} \Bigg(\sum_{{W}_1 \in \mathcal{P}_2(p_n,p_n) } \sum_{(\pi_1,\pi_2) \in \Pi_{L_n,\Delta_n}(W_1)} \mathbb{E} \big( \prod_{i=1}^2 (x_{\pi_{i}}- \mathbb{E}x_{\pi_{i}} )\big) \Bigg)^r \\
		&\hspace{40mm}\times\Bigg(\sum_{ {W}_2 \in \mathcal{P}_{2,4}(p_n,p_n)}  \sum_{(\pi_1,\pi_2) \in \Pi_{L_n,\Delta_n}(W_2)} \mathbb{E} \big( \prod_{i=1}^2 (x_{\pi_{i}}- \mathbb{E}x_{\pi_{i}} )\big) \Bigg)^{k-r} +o(1) \\
		&= \#\mathcal{P}_2(2k)\Bigg( \frac{1}{N^{p_n}\varA}\sum_{{W}_1 \in SP(p_n,2) } \sum_{(\pi_1,\pi_2) \in \Pi_{L_n,\Delta_n}(W_1)} \mathbb{E} \big( \prod_{i=1}^k (x_{\pi_{i}}- \mathbb{E}x_{\pi_{i}} )\big) \Bigg)^k+o(1) \\
		&=\#\mathcal{P}_2(2k)+o(1),
	\end{align*}
	where the last equation follows from \eqref{eqn:var=1}. 
	This shows that  $\lim_{n \rightarrow \infty} \E (\eta_{p_n}^k)= \#\mathcal{P}_2(k)$ for all $k$. Thus using the Wick formula, it follows that limiting distribution is standard Gaussian and the proof of Theorem \ref{thm:converg_LES_even} is complete.
\end{proof}
\begin{proof}[Proof of Theorem \ref{thm:unified_moments_p_even}]
	The proof of Theorem \ref{thm:unified_moments_p_even} is similar to the proof of Theorem  \ref{thm:converg_LES_even}, and we only the major differences. Note that if the input sequence satisfies Assumption \ref{assum: 4 moments}, then in place of \eqref{eqn:moment_bound} we have
	\begin{equation*}
		\E \left(\prod_{ i=1}^k \left(x_{\pi_i}-\E x_{\pi_i}\right)\right) \leq C_W,
	\end{equation*}
	where $C_W$ is a constant that depends only on the sentence $W$.
	
	As a result, it follows that in this case the left-hand-side of \eqref{eqn:contribution_SP(p,k)_even_3} is bounded by $\frac{1}{\sqrt{N}}B^{pk}(pk)^{pk}\max \{C_W: W \notin SP(p,k)\}$. Since this term goes to zero, it follows that for fixed $p$, $\eta_p$ converges to standard Gaussian distribution, following the arguments in the proof of Theorem \ref{thm:converg_LES_even}.
\end{proof}

\subsection{\large \bf \textbf{Proof of Theorem \ref{thm:unified_moments_p_odd}}:} \label{subsec:odd_p}
In this section, we find the limiting moment sequence of LES when $p$ is an odd fixed positive integer.
\begin{proof}[Proof of Theorem \ref{thm:unified_moments_p_odd}]
		Recall that for $k=1$, we already have $\mathbb{E}(\xi_p^k)=0$. For $k \geq 2$, recall  from (\ref{eqn:E w_p^k}) and (\ref{eqn: Card Pi_L(W) zero}) that
	\begin{align*} 
		\E (\xi_p^k) 
		&= \frac{1}{N^{\frac{(p+1)k}{2}}} \sum_{W \in SP(p,k)} \sum_{\Pi_{L_n,\Delta_n}(W)} \mathbb{E} \big( \prod_{i=1}^k (x_{\pi_{i}}- \mathbb{E}x_{\pi_{i}} )\big)  + o(1).
	\end{align*} 
	For $k \geq 2$, consider $k$-tuples of circuits $(\pi_1,\pi_2,\ldots , \pi_k)$ and $(\pi_1^\prime,\pi_2^\prime,\ldots , \pi_k^\prime)$ such that $\pi_i$ and $\pi_i^\prime$ are equivalent for all $1 \leq i \leq k$. Note that since $\{x_i\}$ is an i.i.d. sequence, it follows that the expectation of the summand in (\ref{eqn: eta p power k}) are equal for $(\pi_1,\pi_2,\ldots , \pi_k)$ and $(\pi_1^\prime,\pi_2^\prime,\ldots , \pi_k^\prime)$, and so we denote $x_{\pi_i}$ by $x_{w_i}$. Thus
	\begin{equation}\label{eqn : E eta power k odd}
		\lim_{n \rightarrow \infty}	\mathbb{E}\left(\xi_p^k\right)= \lim_{n \rightarrow \infty} \frac{1}{N^{\frac{pk+k}{2}}}\sum_{W \in SP(p,k)} \prod_{C : \text{ cluster of } G_W } \mathbb{E} \left( \prod_{i \in V(C)} (x_{{w_i}}- \mathbb{E}x_{{w_i}})\right) \#\Pi_{L_n,\Delta_n}(W).
	\end{equation}
	Consider a sentence $W \in SP(p,k)$ and let $C$ be a cluster of $G_W$ such that $\#V(C)=2$. Since $p$ is odd, $\PP_{2,4}(p,p)$ is an empty set. Therefore, a non-zero contribution occurs only when $W_C \in \PP_{2}(p,p)$. For $W_C=(w_{i_1},w_{i_2}) \in \PP_{2}(p,p)$, as $\E( x_i)=0$ and $\E( x_i^2)=1$ for all $i$, we get $\E (x_{w_{i_1}})=\E( x_{w_{i_2}})=0$ and
	$$\mathbb{E}\left(\prod_{i \in V(C)} (x_{{w_i}}-\mathbb{E}x_{{w_i}})\right)=1=m_2.$$
	Now, suppose $C$ is  a cluster of $G_W$ of size $r \geq 2$. Then by the definition of special partition, $W_C$ is a clique sentence and it follows that $S_{W_C}$ contains a letter that appears $r=\# C$ times and all other letters appear twice. Since $\{x_i\}$ obeys Assumption \ref{assump: x_i all moments}, 
	\begin{align*}
		\mathbb{E}\left( \prod_{i \in V(C)} (x_{{w_i}}-\mathbb{E}x_{{w_i}})\right)= m_r.
	\end{align*}
	Substituting back in the expression of expectation in (\ref{eqn : E eta power k odd}), we get the result.
\end{proof}
\begin{remark}
Observe that if the entries are not identically distributed, then $\mathbb{E} \big(\prod_{i=1}^k (x_{\pi_{i}}- \mathbb{E}x_{\pi_{i}} )\big)$ need not be same for two different $k$-tuples of circuits from  $\Pi_{L_n,\Delta_n}(W)$. This is the reason that in comparison to Assumption \ref{assum: 4 moments} for even degree monomial ($p$ even), we need a stronger assumption (Assumption \ref{assump: x_i all moments}) to deal with odd degree monomial test functions ($p$ odd).
\end{remark}

\subsection{\large \bf \textbf{Proof of Theorem \ref{thm: always normal}}} \label{sec:LES_nestlink+ind}
In this section, we prove Theorem \ref{thm: always normal}. 
First, we make the following remark based on Proposition \ref{prop: W_k,p size 2}.
\begin{remark}\label{remark:Pi_Ln,Delta_n-star}
	Let $\{L_n:{[N(n)]^2 \rightarrow \Z^d}\}$ be a sequence of link functions obeying Assumption \ref{Condition B}. For every sentence  $W \in \WW_{p,k}$, note that the set $\Pi_{L_n}^*(W) \setminus \Pi_{L_n}(W)$ can be written as
		\begin{equation}\label{eqn:differnce Pil,PiL star}
			\Pi_{L_n}^*(W) \setminus \Pi_{L_n}(W)=\bigcup_{\tilde{W} \in Q(W)}\Pi_{L_n}(\tilde{W}),
		\end{equation}
	 $$\text{ where } Q(W)=\{\tilde{W}=(\tilde{w}_1,\tilde{w}_2,\ldots,\tilde{w}_k): w_{u_1}[r_1]= w_{u_2}[r_2] \implies \tilde{w}_{u_1}[r_1]= \tilde{w}_{u_2}[r_2] \text{ and } \tilde{W} \neq W\}.$$ Also observe that for each sentence $W \in \PP_{2,4}(p,p)$, $Q(W) \cap (\PP_{2,4}(p,p) \cup \PP_{2}(p,p))=\phi$ and hence by Proposition \ref{prop: W_k,p size 2}(a),
		$$\frac{1}{N^{p+1}} \# \left( \Pi_{L_n}^*(W) \setminus  \Pi_{L_n}(W) \right)\leq \frac{B^{2p}p^p}{\sqrt{N}}.$$
	The term $p^p$ follows from sum of all terms in multinominal expansion.
\end{remark}

\begin{proof}[Proof of Theorem \ref{thm: always normal}]
We show that when the conditions of the theorem are satisfied,  for sufficiently large $N$,
$$\frac{1}{N}\varA$$ is bounded below by $C^{p_n}$ for a constant $C>0$ and therefore $\eta_{p_n}$ converges to the standard Gaussian distribution.

For even $p$, we consider the sentence $W=(w_1,w_2) \in \mathcal{P}_{2,4}(p,p)$ of the following form: 
\begin{align} \label{eqn: condition non-zero P24}
	w_1[1] &=w_1[p]=w_2[1]=w_2[p], \nonumber \\
	w_1[i]&=w_{1}[p+1-i] \text{ for all } 2 \leq i \leq p \nonumber  \text{ and}\\
	w_2[i]&=w_2[p+1-i] \text{ for all } 2 \leq i \leq p,
\end{align}
where $w_1[i], w_2[i]$ denote the $i$-th letter of $w_1$ and $w_2$, respectively (see Figure \ref{fig:sentence_W}).

 Note that by \eqref{eqn:var=1} it is sufficient to show that $\liminf \frac{1}{N^{p_n+1}}\# \Pi_{L_n,\Delta_n}(W)$ is strictly greater than zero for $W$ given by \eqref{eqn: condition non-zero P24}. Further, by Remark \ref{remark:Pi_Ln,Delta_n-star}, we have that 
$$ \frac{1}{N^{p_n+1}} \# \left(\Pi_{L_n,\Delta_n}^*\left(W\right)  \setminus \Pi_{L_n,\Delta_n}\left(W\right)\right)\leq \frac{
B^{2p_n}p_n^{p_n}}{\sqrt{N}}=o(1)$$
for $p_n=o(\log N/\log \log N)$
and therefore it is sufficient to prove that $\liminf \frac{1}{N^{p_n+1}}\# \Pi_{L_n,\Delta_n}^*(W)$ is strictly greater than zero. We accomplish this by obtaining a lower bound on $\#\Pi_{L_n,\Delta_n}^*\left(W\right)$.

\begin{figure}[h!]
	\includegraphics[height=40mm,width=150mm]{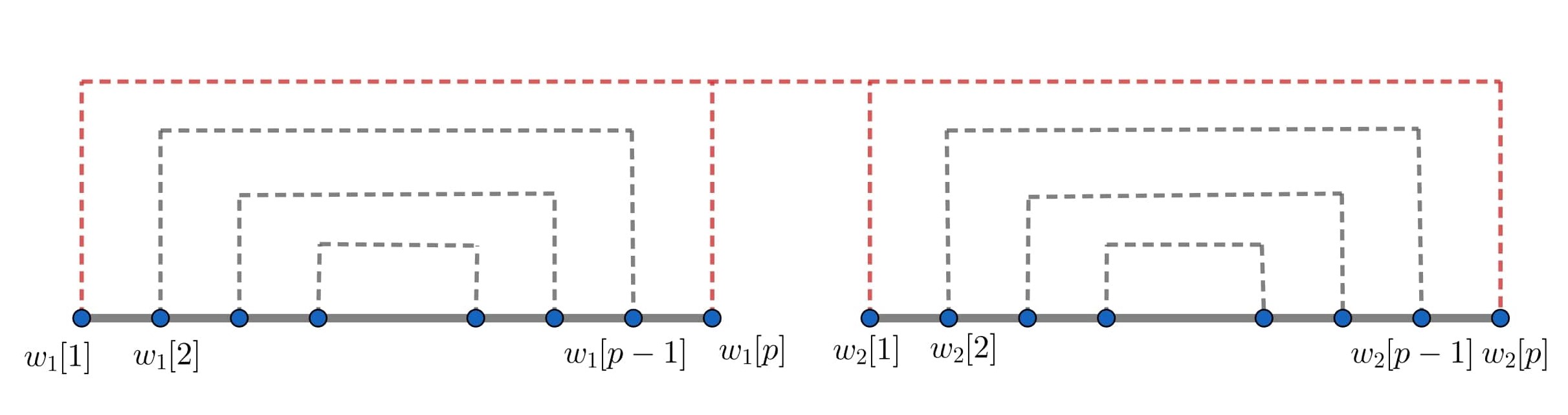}
	\caption{A pictorial representation of the sentence $W=(w_1,w_2)$ obeying (\ref{eqn: condition non-zero P24})}\label{fig:sentence_W}
\end{figure}

Let $n \in \Z_+$ and $p$ even be fixed. Since $\mathcal{S}_n$ is non-empty, there exists some row  $R_i$ such that $\mathcal{S}_n \cap R_i$ is non-empty. It follows from condition $(ii)$ that $\# ( \mathcal{S}_n \cap R_i) \geq c_2N$ and since $\mathcal{S}_n=\mathcal{S}_n^T$, $\# \{j: \mathcal{S}_n \cap R_j \neq \phi\}\geq c_2N$. Choose $\pi_1(0)$ from the set $ \{j: \mathcal{S}_n \cap R_j \neq \phi\}$ and therefore the number of choices for $\pi_1(0)$ is at least $c_2N$. Since $\mathcal{S}_n \cap R_{\pi_1(0)}$ is non-empty, $\mathcal{S}_n \cap R_{\pi_1(0)}$ has at least $c_2N$ elements. Choose $\pi_1(1)$ from the set, so that $\left(\pi_1(0),\pi_1(1)\right) \in \mathcal{S}_n$.

For $1 \leq i \leq p_n/2$, suppose $\pi_1(0),\pi_1(1),\ldots,\pi_1(i-1)$ are chosen such that $\left(\pi_1(j-1),\pi_1(j)\right) \in \mathcal{S}_n$ for all $1 \leq j \leq i-1$.
Since $\left(\pi_1(i-2),\pi_1(i-1)\right) \in \mathcal{S}_n$ and $\mathcal{S}_n=\mathcal{S}_n^t$, we get $\mathcal{S}_n \cap R_{\pi_1(i-1)}$ is non-empty and we choose $\pi_1(i)$ such that $\left(\pi_1(i-1),\pi_1(i)\right) \in \mathcal{S}_n$. Then the number of choices for $\pi_1(i)$ is at least $c_2 N$. For $p_n/2+1 \leq i \leq p_n$, we define $\pi_1(i)=\pi_1(p_n-i)$. Then it follows that for $p_n/2+1 \leq i \leq p$,
\begin{align*}
	w_1[i] &= L_n\left(\pi_1(i-1),\pi_1(i)\right) \\
	&= L_n\left(\pi_1(p_n-i+1),\pi_1(p_n-i)\right) \\
	&= L_n\left(\pi_1(p_n-i),\pi_1(p_n-i+1)\right)\\
	&=w_1[p_n -i+1].
\end{align*}
Furthermore, $\pi_1(p_n)=\pi_1(0)$ and since $A_n$ is symmetric, $(\pi_1(p_n-i+1),\pi_1(p_n-i)) \in \Delta_n$ for all $p_n/2+1 \leq i \leq p_n$. Therefore by construction, $\pi_1 \in \Pi_{L_n,\Delta_n}^*(w_1)$  and the number of choices of $\pi_1$ is bounded below by $(c_2 N)^{p_n/2+1}$.

Next, we calculate the number of choices of $\pi_2$ such that $(\pi_1,\pi_2) \in \Pi_{L_n,\Delta_n}^*(W)$. As we have already chosen $\pi_1$, we have that $r_u:=L_n\left(\pi_1(0),\pi_1(1)\right)$ is already fixed. Choose $\pi_2(0)$ such that $(\pi_2(0),j) \in \mathcal{S}_n$ and $L_n(\pi_2(0),j)=r_u$ for some $j$. By condition $(i)$, the number of choices for $\pi_2(0)$ is bounded below by $c_1 N/B$. We choose $\pi_2(1)$ such that $(\pi_2(0),\pi_2(1)) \in \mathcal{S}_n$ and $L_n(\pi_2(0),\pi_2(1))=r_u$. Note that such a choice for $\pi_2(1)$ always exist.

For $2 \leq i \leq p_n/2$, the letter $w_2[i]$ has not yet appeared. Suppose that $\pi_2(0),\pi_2(1),\ldots , \pi_2(i-1)$ are chosen such that $ \left(\pi_2(j-1),\pi_2(j)\right) \in \mathcal{S}_n$ for all $1 \leq j \leq i-1$. Then the number of choices of $\pi_2(i)$ such that $\left(\pi_2(i-1),\pi_2(i)\right) \in \mathcal{S}_n$ is bounded below by $c_2 N$. For $p_n/2+1 \leq i \leq p_n$, we define $\pi_2(i)=\pi_2(p_n-i)$. 

Therefore the number of choices for $\pi_2$ is bounded below by $c_1N(c_2N)^{p_n/2-1}/B$. It then follows that $\pi_2$ is a circuit and $(\pi_1,\pi_2) \in \Pi_{L_n,\Delta_n}^*(W)$. Thus, $\# \Pi_{L_n,\Delta_n}^*(W) \geq (c_2N)^{p_n/2+1} \times \frac{c_1N}{B}(c_2N)^{p_n/2-1}$. Hence for sufficiently large $N$, 
\begin{equation*}
	\frac{1}{N} \times \varA \geq \frac{1}{2N^{p_n+1}}\# \Pi_{L_n,\Delta_n}^*(W) \geq \frac{1}{2N^{p_n+1}}\frac{c_1N}{B}(c_2N)^{p_n} =\frac{c_1c_2^{p_n}}{2B}\geq \left(\frac{c_1c_2}{2B}\right)^{p_n}.
\end{equation*}
The last inequality follows since $c_1\leq 1$ and $B \geq 1$.
This completes the proof of theorem.
\end{proof}

\section{\large \bf \textbf{Proofs of Propositions \ref{prop: W_k,p size 2} and \ref{lemma: Bpl general}}} \label{sec:proof of Bpl_prop}
\subsection{\large Generating vertices:}
In this section, we introduce the concept of generating vertices and the role they would play in the calculations. First we give a definition of generating vertices. This is a more general version of the concept of generating vertices in \cite{bose_sen_LSD_EJP}.
\begin{definition}\label{defn:generating_vertex}
	\noindent (a) For a sentence $W \in \WW_{p,k}$, any combination $(u,j)$ where $1 \leq u \leq k$ and $0 \leq j \leq p$ is called a \textit{vertex}.
	\vskip2pt
	\noindent (b) For a sentence $W=(w_1,w_2,\ldots,w_k) \in \WW_{p,k}$, a set $A \subset S_W$ and a total ordering $`<$' on $\{(u,j): 1 \leq u \leq k, 0 \leq j \leq p\}$, a vertex $(u,j)$ called a \textit{generating} vertex of $W$ if 
	\begin{enumerate}[(i)]
		\item $j=0$ or
		\item $w_u[j] \notin A$ and $w_\ell[i] \neq w_u[j]$ for all $(\ell,i) < (u,j)$, where $w_u[j]$ denotes the $j$-th letter of $w_u$.
	\end{enumerate}
Note that generating vertices depend on the set $A$ and the ordering $`<$'. If not specified, we choose $A$ as the empty set and $<$ as the dictionary order on $\{(u,j): 1 \leq u \leq k, 0\leq j \leq p\}.$ 
\end{definition}

Recall the graph $G_W=(V_W,E_W)$ for $W \in \WW_{p,k}$, where $G_W$ as in Definition \ref{defn_cluster}. Consider a bijection $\zeta: V_W \rightarrow [k]$. The map $\zeta$ is called a \textit{renumbering} of vertices. We define an ordering with respect to $\zeta$ by $(\ell,i)<(u,j)$ if $(\zeta^{-1}\ell,i)$ is less than $(\zeta^{-1}u,j)$ under the dictionary ordering.

Observe from the definition of generating vertex that corresponding to each letters in $S_W$ that are not present in $A$ there exists a unique generating vertex, and further every generating vertex other than $(u,0)$ is of this form. Therefore given a sentence $W$ and a set $A \subset S_W$, the total number of generating vertices does not depend on the ordering, and we denote the number of generating vertices in $W$ by $F(W,A)$. In particular for a sentence $W \in \WW_{p,k}$ and $A =\phi$, the number of generating vertices in $W$ is $|W|+k$, where $|W|$ denotes the number of distinct letters in $S_W$.

\begin{example}
	Consider the sentence $W=(abca,addc,aedc,ebbd)$.  The set of generating vertices of $W$ is $\{(1,0),(1,1),(1,2),(1,3),(2,0),(2,2),(3,0),(3,2),(4,0)\}$.
\end{example}
Soon, we shall see that the concept of generating vertices play an important role in determining the cardinality of $\Pi_L(W)$. In this regard, the introduction of vertices of the form $(u,0)$ is intuitive, since for each word $w_i$ a circuit is a map from $\{0,1,2,\ldots , p\}$ and to count the number of different possibilities of circuits $\pi_i$, $\pi_i(0)$ also needs to be determined. We first present an example that contains important ideas required to determine the cardinality of $\Pi_L^*(W)$.

\begin{example}\label{eg:card_Pi_l_nosingle}
	Consider the word $w=abab$ and let $L: [N] \rightarrow \Z^d$ be a link function obeying Assumption \ref{Condition B}. Considering $w$ as a sentence in $\WW_{p,1}$, the set of generating vertices of $w$ is $\{(1,0),(1,1),(1,2)\}$. 
	
	To find the cardinality of $\Pi_L^*(w)$, we start counting the possible values of $\pi(i)$, starting from $\pi(0)$, for circuits $\pi \in \Pi_L^*(w)$. 
	As there are no constraints on the value of $\pi(0)$, the number of possible choices for $\pi(0)$ is $N$. Similarly, the number of choices for $\pi(1)$ and $\pi(2)$ are also $N$, as there are no constraints for $\pi(1)$ and $\pi(2)$. Once $\pi(0)$ and $\pi(1)$ are chosen, the condition $w[1]=w[3]$ dictates that $L(\pi(2),\pi(3))$ should be equal to $L(\pi(0),\pi(1))$ which is already fixed. Since $L$ obeys Assumption \ref{Condition B}, the maximum number of possibilities for $\pi(3)$ is bounded by $B$. As $\pi(4)=\pi(0)$, the value of $\pi(4)$ is already determined. Thus we obtain the following bound for the cardinality of $\Pi_L^*(w)$:
	\begin{equation*}
		\#\Pi_L^*(w) \leq N \times N \times N \times B \times 1 =BN^3= BN^{F},
	\end{equation*}
	where $F$ is the number of generating vertices in $w$. Intuition suggests that the same upper bound would hold for any sentence $W$.
\end{example}
Now, we present another example and show that the above bound is not tight and the upper bound can be reduced by an order of $N$.
\begin{example}\label{eg:card_Pi_l_singleelement}
	Consider the word $w=abacb$. Note that the letter $w[4]=c$ appears only once in $w$. We shall exploit this property to show that $\#\Pi_L^*(w) \leq B^pN^{F-1}$.
	
	We start by counting the number of choices for $\pi(0)$ and proceed in increasing order to assign values till $\pi(3)$. From Example \ref{eg:card_Pi_l_nosingle} it is clear that the number of choices for $\pi(0),\pi(1), \pi(2), \pi(3)$ is $N^3 \times B$. Now we move to the end, and assign value to $\pi(5)$. As $\pi(5)=\pi(0)$, there is only one choice for $\pi(5)$. Further, note that $L(\pi(4),\pi(5))=L(\pi(1),\pi(2))$ and the latter quantity is already fixed. Therefore, by Assumption \ref{Condition B}, the number of choices for  $\pi(4)$ is bounded above by $B$. Hence 
	\begin{equation*}
		\#\Pi_L^*(w) \leq N^3B \times B \times 1 = B^2 N^{F-1}.
	\end{equation*}
\end{example}
\noindent
In the next lemma, we extend the reasoning of above examples for arbitrary sentences.

\begin{lemma}\label{lemma:choices pi t}
	Let $L: [N] \rightarrow \Z^d$ be a link function obeying Assumption \ref{Condition B}. For $k \in \Z_+$ and $W \in \WW_{p,k}$, let $\pi_1,\pi_2,\ldots , \pi_{t-1}:\{0,1,2,\ldots,p\} \rightarrow [N]$ be  circuits such that $(\pi_1,\pi_2,\ldots , \pi_{t-1}) \in \Pi_L^*(w_1,w_2,\ldots,w_{t-1})$. Suppose $Q_t=\# \{\pi_t: (\pi_1,\pi_2,\ldots , \pi_{t}) \in \Pi_L^*(w_1,w_2,\ldots,w_{t})\}$. Then
	\begin{align*}
		Q_t
		\leq \begin{cases}
		B^{p}N^{F_t-1} &\text{ if } \exists \  a \in S_{w_t} \text{ such that } m_{S_{w_t}}(a)=1 \text{ and } m_{S_{w_\ell}}(a)=0 \text{ for all } \ell<t,\\
			B^{p}N^{F_t} &\text{ otherwise},
		\end{cases}
	\end{align*}
	where $F_t$ is the number of generating vertices of $w_t$ and $m_{S_{w_\ell}}$ denotes the multiplicity function of the multiset $S_{w_\ell}$.
\end{lemma}
\begin{proof}
	We first consider the case where $S_{w_t}$ does not contain any letter that appears only once in $S_{w_t}$. We start by assigning values for $\pi_t(0)$ and proceed in ascending order till $\pi_t(p)$. Note that the number of possibilities of $\pi_t(0)$ is $N$. Further, by the logic followed in Example \ref{eg:card_Pi_l_nosingle} it is clear that if $w_t[j]$ is the first appearance of a letter in the sentence $W$, then there is no prior constraint involving $L(\pi_t(j-1),\pi_t(j))$, and therefore the number of choices for $\pi_t(j)$ is $N$. Further if the letter $w_t[j]$ has already appeared, then note that the numerical value of $L(\pi_t(j-1),\pi_t(j))$ is already determined by the choice of $\pi_1,\pi_2,\ldots,\pi_{t-1} $ and $\pi_t(0),\pi_t(1),\ldots, \pi_t(j-1)$. Therefore, due to Assumption \ref{Condition B}, the number of values for $\pi_t(j)$ is bounded above by $B$. It follows from here that 
	\begin{align*}
		Q_t&\leq N^{F_t}B^{(p+1)-F_t} \leq  N^{F_t}B^{p}.
	\end{align*}
	
	Now suppose $S_{w_t}$ contains a new letter that is repeated only once in $S_{w_t}$. Let $x=w_t[i]$ be the last letter such that $x \notin S_{w_\ell}$ for all $\ell < t$ and $m_{S_{w_t}}(x)=1$.
	
	We count the number of choices for $\pi_t$ in the following way: We start by counting the number of choices for $\pi_t(0)$. Note that $\pi_t(0)$ has $N$ choices. Next for $j<i$, suppose the value of $\pi_t(0),\pi_t(1),\ldots , \pi_t(j-1)$ are chosen. If $(t,j)$ is a generating vertex, then note that the value of $L(\pi_t(j-1),\pi_t(j))$ is not fixed and therefore the number of choices for $\pi_t(j)$ is $N$. If $(t,j)$ is not a generating vertex, then by Assumption \ref{Condition B}, the number of choices for $\pi_t(j)$ is bounded above by $B$. 
	
	Once all $\pi_t(j)$ till $j \leq i-1$ are determined, we assign values to $\pi_t(j)$ starting from $\pi_t(p)$ and proceeding in the decreasing order of $j$ till $\pi_t(i)$. Note that the number of choices for $\pi_t(p)$ is 1, as $\pi_t(p)=\pi_t(0)$. Suppose $\pi_t(p),\pi_t(p-1),\ldots, \pi_t(p-j+1)$ are chosen for some $0 \leq j \leq p-i$ and the letter $w_t[p-j+1]$ has not appeared yet. Then it follows that the value of $L(\pi_t(p-j),\pi_t(p-j+1))$ is not fixed and therefore $\pi_t(p-j)$ has $N$ choices. And like in the previous case, if the letter $w_t[p-j+1]$ has already appeared, then the value of $L(\pi_t(p-j),\pi_t(p-j+1))$ is already fixed and therefore due to Assumption \ref{Condition B}, $\pi_t(p-j)$ has at most $B$ choices. As a result, the value of $\pi_t(i)$ is determined by the value of other $\pi_t(j)$, even though $(t,i)$ is a generating vertex.
	Note that by this construction, for all generating vertices other than $(t,i)$, there exist a component ($\pi_t(j)$ or $\pi_t(j-1)$ depending on whether $j < i$ or $j>i$) which has $N$ choices and the number of choices for all other $\pi_t(j)$ is bounded above by $B$. Hence, $Q_t$ is bounded above by $B^{p}N^{F_t-1}$. This completes the proof.
\end{proof}

Recall the notion of renumbering of vertices and ordering with respect to renumbering from Definition \ref{defn:generating_vertex}. The following lemma is an immediate generalization of Proposition
 \ref{lemma: Bpl general} and would follow from a similar argument.
\begin{lemma}\label{lemma:choice of pi_t reordered}
	Let $L: [N] \rightarrow \Z^d$ be a link function obeying Assumption \ref{Condition B}. For $k \in \Z_+$ and $W \in \mathcal{W}_{p,k}$, let $\zeta:V_W \rightarrow [k]$ be a bijection and let $\pi_{\zeta^{-1}(1)},\pi_{\zeta^{-1}(2)},\ldots , \pi_{\zeta^{-1}(t-1)}: \{0,1,2,\ldots,p\} \rightarrow [N]$ 
	be  circuits such that $$(\pi_{\zeta^{-1}(1)},\pi_{\zeta^{-1}(2)},\ldots , \pi_{\zeta^{-1}(t-1)}) \in \Pi_L^*(w_{\zeta^{-1}(1)},w_{\zeta^{-1}(2)},\ldots,w_{\zeta^{-1}(t-1)}).$$ Suppose $Q_t=\# \{\pi_{\zeta^{-1}(t)}: (\pi_{\zeta^{-1}(1)},\pi_{\zeta^{-1}(2)},\ldots , \pi_{\zeta^{-1}(t)}) \in \Pi_L^*(w_{\zeta^{-1}(1)},w_{\zeta^{-1}(2)},\ldots,w_{\zeta^{-1}(t)})\}$. Then
	\begin{align*}
		Q_t
		\leq \begin{cases}
			B^{p}N^{F_{\zeta^{-1}(t)}-1} &\hspace{-2mm}\text{if } \exists \  a \in S_{w_t}: m_{S_{w_{\zeta^{-1}(t)}}}(a)=1,  m_{S_{w_{\zeta^{-1}(\ell)}}}(a)=0 \ \forall \ \zeta^{-1}(\ell)<\zeta^{-1}(t),\\
			B^{p}N^{F_{\zeta^{-1}(t)}} &\hspace{-2mm}\text{otherwise},
		\end{cases}
	\end{align*}
	where $F_{\zeta^{-1}(t)}$ is the number of generating vertices in $w_{\zeta^{-1}(t)}$.
\end{lemma}


\subsection{\large Proof of Propositions \ref{prop: W_k,p size 2}:}
We prove Proposition \ref{prop: W_k,p size 2} for parts (a) and (b) separately. First, we prove part (a).
\begin{proof}[Proof of Proposition \ref{prop: W_k,p size 2}(a)]
	First consider $W=(w_1,w_2) \in  \mathcal{P}_2(p,p)$. Let $a \in S_{w_1} \cap S_{w_2}$ be a cross-matched letter. Note that $a$ appears only once in $S_{w_1}$ and thus by Lemma \ref{lemma:choices pi t}, it follows that the number of choices for $\pi_1$ is bounded above by $B^p N^{F_1-1}$, where $F_1$ is the number of generating vertices of $w_1$. Similarly, the number of choices for $\pi_2$ is bounded above by $B^p N^{F_2}$, where $F_2$ is the number of generating vertices of $w_2$. Thus we get
	\begin{equation*}
		\frac{1}{N^{p+1}}\# \Pi_{L_n}^*(W) \leq  \frac{1}{N^{p+1}} B^{2p} N^{F_1+F_2-1} = B^{2p}.
	\end{equation*}
	Here the last equality in the above expression follows from the fact that $F_1+F_2=|W|+2$, where $|W|$ is the number of distinct letters in $S_W$, which is $p$.
	
	Now for $W=(w_1,w_2)\in \mathcal{P}_{2,4}(p,p)$, note that each letter in $S_{w_i}$ is repeated exactly twice in $S_{w_i}$ and therefore from Lemma \ref{lemma:choices pi t}, we get that the number of choices for  $\pi_1$ is at most $B^pN^{F_1}$ and the number of choices for  $\pi_2$ is at most $B^pN^{F_2}$.
	Thus, it follows that for all $W \in \PP_{2,4}(p,p)$,
	\begin{equation*}
		\frac{1}{N^{p+1}}\# \Pi_{L_n}^*(W) \leq \frac{1}{N^{p+1}}B^{2p}N^{F_1+F_2}=B^{2p},
	\end{equation*}
	where the last equality follows as the number of distinct letters in $S_{W}$ is $(p-1)$.
	
	Next, consider $W=(w_1,w_2)\in \WW_{p,2} \setminus \PP_{2}(p,p)$ with at least one cross-matched letter, each letter in $S_W$ appearing at least twice and $\# \Pi_{L_n}^*(W) > B^{2p}N^{p+1/2}$. We shall show that this happens only if $W \in \mathcal{P}_{2,4}(p,p)$.
	
	Since each letter in $S_W$ is repeated at least twice, we get that $|W| \leq p$. It follows from Lemma \ref{lemma:choices pi t} that the number of elements in $\Pi_{L_n}^*(W)$ is at most $B^{2p}N^{|W|+2}$. Thus $\# \Pi_{L_n}^*(W) > B^{2p}N^{p+1/2}$ only if $|W| \geq p-1$. Hence, the only possibilities are $|W| \in \{p-1,p\}$.
	
	Suppose $|W|=p$, then since each letter in $S_W$ appears at least twice and the condition that $W$ has at least one cross-matched element, implies that $W \in \mathcal{P}_{2}(p,p)$, a contradiction.
	
	Now, let $|W|=p-1$. Suppose there exists a cross-matched letter $a \in S_{W}$ such that $a$ is repeated only once in $S_{w_i}$ for $i=1$ or $2$. Note that if $a$ appears only once in $S_{w_1}$, then by Lemma \ref{lemma:choices pi t}, it follows that the number of choices for $\pi_1$ is at most $B^pN^{F_1-1}$ and the number of choices for $\pi_2$ is at most $B^pN^{F_2}$. In the case if $a$ appears more than once in $S_{w_1}$ and only once in $S_{w_2}$, then from Lemma \ref{lemma:choice of pi_t reordered} the number of choices for $\pi_2$ is at most $B^pN^{F_2-1}$ and consequently, $\#\Pi_{L_n}^*(W) \leq B^{2p}N^{|W|+1}= B^{2p}N^{p}$ for both the cases. Thus for $|W|=p-1$, $\#\Pi_{L_n}^*(W) > B^{2p}N^{p+1/2}$, only if each cross-matched element is repeated at least four times, twice in both $S_{w_1}$ and $S_{w_2}$. Let $r$ be the number of cross-matched elements in $W$. Then we have 
	\begin{align*}
		4r+ 2(p-1-r) \leq 2p \implies r=1.
	\end{align*}
	This implies that $W \in \mathcal{P}_{2,4}(p,p)$ and this completes the proof for part (a).
\end{proof}
\begin{proof}[Proof of Proposition \ref{prop: W_k,p size 2}(b)]
	Let $W=(w_1,w_2,\ldots, w_k) \in \WW_{p,k}$ be a clique sentence with the letter $a \in S_W$ repeating exactly $k$ times. Then the number of distinct letters in $W$ is $|W|=\frac{pk-k}{2}+1$. From the definition of clique sentence it follows that $a$ appears only once in $w_1$ and therefore by an application of Lemma \ref{lemma:choices pi t}, the cardinality of $\Pi_{L_n}^*(W)$ is bounded above by $B^{pk}N^{|W|+k-1}$. Therefore
	\begin{equation*}
		\frac{1}{N^{\frac{pk+k}{2}}}\# \Pi_{L_n}^*(W) \leq \frac{1}{N^{\frac{pk+k}{2}}}B^{pk}N^{|W|+k-1} =B^{pk}.
	\end{equation*}
	Since $\#\Pi_{L_n}(W) \leq \#\Pi_{L_n}^*(W)$ and this completes the proof.
\end{proof}
\subsection{\large Proof of Proposition   \ref{lemma: Bpl general}:}

In this section, we prove Proposition \ref{lemma: Bpl general}. We first make some preliminary observations. 


Note that since $\E x_i=0$, it follows from standard calculations that for non-zero expectation, each letter must appear at least twice in $S_W$. Thus, we focus on sentences of that form and for each such  sentence, the number of distinct letters is at most $pk/2$. As a result, from Lemma \ref{lemma:choices pi t} we get that the cardinality of $\Pi_L^*(W)$ is at most $B^{pk}N^{\frac{pk}{2}+k}$.

For a sentence $W$ such that each letter appears at least twice and $G_W$ is connected, $G_W$ has at least $k-1$ edges. Consider an edge $\{i,j\} \in E_W$ and a cross-matched letter, say $a$, belonging to  $S_{w_i} \cap S_{w_j}$. Lemma \ref{lemma:choices pi t} says that if $a$ is a new letter that appears only once, then the degree of freedom for choosing $\pi_t$ reduces by $1$. Otherwise, $a$ must appear at least thrice in $S_W$, in vague terms, leading to a loss of half a degree of freedom. Thus, even in the worst case, we can expect a loss of $(k-1)/2$ degrees of freedom. We shall prove Proposition \ref{lemma: Bpl general} by showing that there is a loss of at least $(k+1)/2$ degrees of freedom whenever $W$ is not a clique sentence.

\vskip3pt
\noindent\textbf{Outline of the proof:} For better readability, we divide the proof of Proposition \ref{lemma: Bpl general} 
into three steps as given below.
\vskip3pt
\noindent\textbf{Step I:}
The discussion at the starting of this subsection highlights the role of cross-matched letters in proving Proposition \ref{lemma: Bpl general}. With this cue, we first consider one of the simplest cases possible, where there is a choice of distinct cross-matched letters for each edge in a spanning tree of $G_W$.
\begin{definition}\label{defn:C_k,p}
	For $k \geq 3$ and $ p \geq 1$, consider a sentence $W \in \WW_{p,k}$ such that $G_W$ is connected. We say $W \in \mathcal{C}_{p,k}$ if 
	\begin{enumerate}[(i)]
		\item each letter in $S_W$ appears at least twice in $S_W$, where $S_W$ is defined in Definition \ref{defn:SW}.
		\item $G_W$ has a spanning tree $T_W$, such that there exists a one-one function $\varphi: E(T_W) \rightarrow S_W$ with the property $\varphi(\{i,j\}) \in S_{w_i} \cap S_{w_j}$ for all $\{i,j\} \in E(T_W)$.
	\end{enumerate}
\end{definition}

We first show that for each $W \in \mathcal{C}_{p,k}$, the conclusion of Proposition \ref{lemma: Bpl general} holds (see Lemma \ref{prop:Pi C_k,p}). This is obtained by choosing a reordering of vertices such that there is maximum loss of degree of freedom. 

\vskip3pt
\noindent\textbf{Step II:} For $W \in \WW_{p,k}$, we consider a particular type of subgraph $H_W$ of $G_W$, which corresponds to the case where each connected component of $H_W$ is a tree which satisfies condition (iii) of Definition \ref{defn:C_k,p}, along with the additional conditions given in Definition \ref{defn:Type I subgraph}. 
\begin{definition}\label{defn:Type I subgraph}
	For a sentence $W=(w_1,w_2,\ldots, w_k)\in \WW_{p,k}$ with $m_{S_W}(a) \geq 2$ for all $a \in S_W$, a subgraph $ H_W \subset G_W$ with connected components $T_1,T_2,\ldots , T_r$ is called a \textit{Type I subgraph}, if 
	\begin{enumerate}[(i)]
		\item $V(H_W)=\{1,2,\ldots , k\}$ and $H_W$ is a forest,
		\item $\#V(T_1) \geq 3$ and if $\# V(T_j)=1$ for some $j$, then $\# V(T_i)=1$ for all $i >j$,
		\item for all $a \in S_{W_{T_m}}$, $1 \leq m \leq r$, $a$ appears at least twice in $\bigcup_{i \leq m} S_{W_{T_i}}$, 
		\item there exists a one-one map $\varphi: E(H_W) \rightarrow S_{W}$ such that $\varphi(\{i,j\}) \in S_{w_i} \cap S_{w_j}$ for each $\{i,j\} \in E(H_W)$.
	\end{enumerate}
\end{definition}
Pictorially, Type I subgraph can be represented as shown in Figure \ref{fig:Type1sub}.
\begin{figure}[h!]
	\includegraphics[height=42mm,width=150mm]{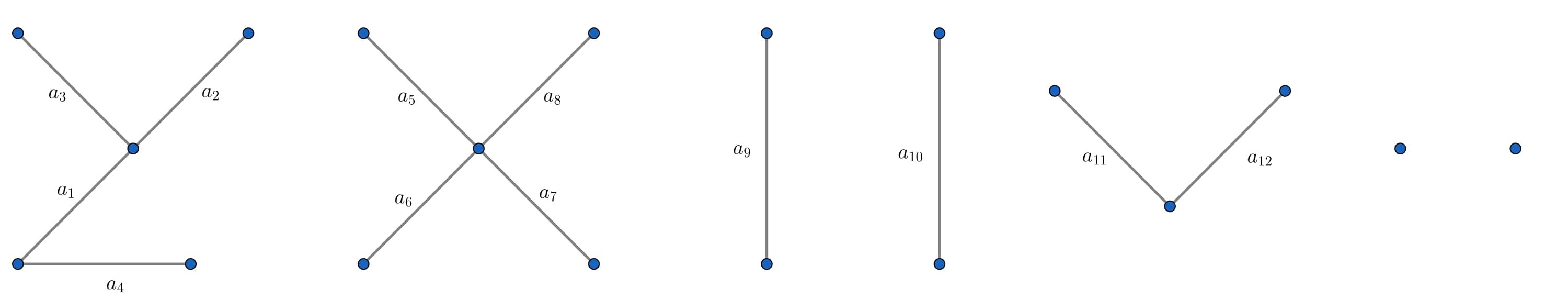}
	\caption{Example of a Type I subgraph with $a_e=\varphi(e)$ as distinct letters.} \label{fig:Type1sub}  
\end{figure}

In Lemma \ref{prop:existence_Type1subg}, we show that any $W \in \WW_{p,k}; k \geq 3$ such that $G_W$ is connected but $W$ is not a special sentence, has a Type I subgraph.
\vskip3pt
\noindent\textbf{Step III:} Finally, to prove Proposition \ref{lemma: Bpl general}, we show that if a sentence $W$ has a Type I subgraph, then  
\begin{equation*}
	\frac{1}{N^{\frac{pk+k}{2}}} \#\Pi_{L_n}^*(W) \leq \frac{B^{pk}}{\sqrt{N}}.
\end{equation*}
Combining Step III with Step II implies Proposition \ref{lemma: Bpl general}.
\vskip3pt
Now, we move on to the detailed proof:
\vskip3pt
\noindent\textbf{Step I:}
Before proceeding to this step in detail, we recall some terminologies from graph theory. For a finite graph $G=(V,E)$ and a vertex $v \in V$, the \textit{eccentricity} of $v$ is defined as $g[v]=\sup \{\operatorname{dist}(i,v):i \in V\}$, where $\operatorname{dist}(i,v)$ is the length of the shortest path from $i$ to $v$ through the edges of $G$. The \textit{center} of $G$ is defined as the set $$\mathbb{O}(G)=\{v \in V:g[v] \leq g[i] \,\, \forall \ i  \in V\}.$$ For a tree $T$, it is known that $\# \mathbb{O}(T) \leq 2$ and if $T$ has more than two vertices, then for each $o \in \mathbb{O}(T)$, $\operatorname{deg}(o) \geq 2$. 

Recall the notion of renumbering of vertices in Definition \ref{defn:generating_vertex}.
Let $W \in \CC_{p,k}$ such that $T_W$ is a spanning tree of $G_W$, $\varphi: E(T_W) \rightarrow S_W$ is a one-one map obeying condition (i) of Definition \ref{defn:C_k,p}, and let $\zeta$ be a renumbering of vertices obeying the following condition:
\begin{equation}\label{eqn: condition zeta}
	\operatorname{dist}\left(\mathbb{O}(T_W),j\right) < \operatorname{dist}\left(\mathbb{O}(T_W),i\right) \implies \zeta(i)<\zeta(j).
\end{equation}
Recall the multiset $S_W$ defined in Definition \ref{defn:multiset} and let $m_{S_W}$ and $m_{S_{w_i}}$ be the multiplicity functions of $S_W$ and $S_{w_i}$, respectively. For an edge $\{i,j\} \in E(T_W)$, we define 
\begin{equation}\label{eqn: defn ai,j}
	a_{i,j}:= \varphi(\{i,j\}) \text{ where } \zeta(i) <\zeta(j).
\end{equation}

Note that in the definition of $a_{i,j}$ the ordering of $i,j$ is important and this distinction will play an important role in the proofs later. We remark that for any finite graph $G$, there always exists a renumbering $\zeta$ of vertices obeying (\ref{eqn: condition zeta}).

For a set $A \subset S_W$ and a renumbering $\zeta$, consider the following sets.
\begin{align}\label{eqn: C A}
	&\mathcal{C}_A= \{a_{i,j}: a_{i,j} \notin S_{w_\ell} \,\, \forall \ \zeta(\ell)<\zeta(i), a_{i,j} \notin A \text{ and } m_{S_{w_i}}(a_{i,j})=1\}, \nonumber\\
	&\widetilde{\mathcal{C}}_A= \operatorname{Im}(\varphi) \setminus \mathcal{C}_A.
\end{align}
Note that the sets $\mathcal{C}_A$ and $\widetilde{\mathcal{C}}_A$ depend on the renumbering $\zeta$, but we are suppressing $\zeta$ for notational convenience. 
From the definition of $\widetilde{\mathcal{C}}_A$, it follows that each letter $ a_{i,j} \in \widetilde{\mathcal{C}}_A$ such that $a_{i,j} \notin A$ appears at least thrice in $S_W$.

The following lemma provides the contribution of the sentences in $\CC_{p,k}$ for $k\geq 3$,  where $\CC_{p,k}$ as in Definition \ref{defn:C_k,p}.
\begin{lemma}\label{prop:Pi C_k,p}
	For $p \geq 1$ and $ k \geq 3$, let $W \in \mathcal{C}_{p,k}$ and $\{L_n\}$ be a sequence of link functions obeying Assumption \ref{Condition B}. Then
	\begin{equation*}
		\frac{1}{N^{\frac{pk+k}{2}}} \#\Pi^*_{L_n}(W) \leq \frac{B^{pk}}{\sqrt{N}}.
	\end{equation*}
\end{lemma}

We first state a lemma required in the proof of Lemma \ref{prop:Pi C_k,p}.
\begin{lemma}\label{prop: Pi_L F-R}
	For $ k \geq 3$ and $p \geq 1$, let $W \in \WW_{p,k}$ be such that $G_W $ is connected. Suppose $T_W$ is a spanning tree of $G_W$, $\varphi: E(T_W) \rightarrow S_W$ is a one-one function, $\zeta$ is a renumbering obeying (\ref{eqn: condition zeta}), $\{L_n\}$ is a sequence of link functions obeying Assumption \ref{Condition B} and $A \subset S_W$. Then 
	\begin{equation}
		\#\Pi_{L_n}^*(W) \leq B^{pk} N^{F-R},
	\end{equation}
	where $F=F(W,A)$ is the number of generating vertices of $W$ given A and $R=\#  \CC_A$.
\end{lemma}
Assuming Lemma \ref{prop: Pi_L F-R}, we now prove Lemma \ref{prop:Pi C_k,p}. The proof of Lemma \ref{prop: Pi_L F-R} is given after the proof of Lemma \ref{prop:Pi C_k,p}.
\begin{proof}[Proof of Lemma \ref{prop:Pi C_k,p}] 
In this proof, we choose $A$ as the null set, and denote $\CC=\CC_\phi $ and $ \widetilde{\CC}=\widetilde{\CC}_\phi$. We prove the proposition by dividing into cases based on different values of $R= \# \CC$.
\vskip3pt

\noindent\textbf{Case 1: } $R \geq 2$ for some choice of renumbering $\zeta$ obeying $(\ref{eqn: condition zeta})$.

Fix a renumbering $\zeta$ obeying \eqref{eqn: condition zeta} such that $R \geq 2$. First note that each element of  $\widetilde{\mathcal{C}}$ appears at least thrice in $S_W$ and each element in $\CC$ appears at least twice. Considering the fact that every element in $S_W$ also appears at least twice, we get that the number of elements in the multiset $S_W \setminus \operatorname{Im}(\varphi)$ is at most $\frac{pk-3(k-1-R)-2R}{2}$. Clearly, the number of elements in $\operatorname{Im}(\varphi)$ is $k-1$. As a result, we get that the number of generating vertices of $W$ is
$$F=|W|+k \leq \frac{pk-3(k-1-R)-2R}{2}+(k-1)+k=\frac{pk}{2}+\frac{k}{2}+\frac{1+R}{2}.$$
Thus, it follows from Lemma \ref{prop: Pi_L F-R} that $\frac{1}{N^{\frac{pk+k}{2}}} \#\Pi_{L_n}^*(W) \leq B^{pk} N^{\frac{1-R}{2}}\leq B^{pk}/\sqrt{N}$ if $R \geq 2$. 

Before going to other cases, we define another quantity. Let
\begin{equation}\label{eqn:U,D}
	\mathcal{D}= \{a_{i,j} \in \widetilde{\mathcal{C}}: m_{S_W}(a_{i,j}) \geq 4\} \text{ and } U =\# \mathcal{D}.
\end{equation}
Note that $\mathcal{D}$ and $U$ depends on the sentence $W$, the map $\varphi$ and the renumbering $\zeta$.
\vskip3pt
\noindent\textbf{Case 2:} $R=0$ for some choice of renumbering $\zeta$ obeying $(\ref{eqn: condition zeta})$.

Let $\zeta$ be a renumbering of vertices obeying $(\ref{eqn: condition zeta})$ such that $R=0$.
We first consider the case $U \geq 2$, where $U$ is as defined in (\ref{eqn:U,D}). Since $R=0$, by Lemma \ref{prop: Pi_L F-R}, we have $\#\Pi_{L_n}^*(W)\leq B^{pk}N^F$. Note that each letter in $\mathcal{D}$ appears at least four times, each letter in $\CC$ appears at least thrice and all other letters appear at least twice. Therefore, 
$$F \leq \frac{pk-4U-3(k-1-U)}{2}+(k-1)+k=\frac{pk}{2}+\frac{k}{2}-\frac{U}{2}+\frac{1}{2}.$$
As a result, we get that $\frac{1}{N^{\frac{pk+k}{2}}} \#\Pi_{L_n}^*(W) \leq B^{pk}/\sqrt{N}$ if $U \geq 2$.

Next consider the case $U \in \{0,1\}$. As $k \geq 3$, $T_W$ has at least two leaves. Since $U \leq 1$, there exists a leaf $i$ such that $\{i,j\} \in E(T_W)$ and $m_{S_W}(a_{i,j}) \leq 3$. Note that since $R=0$, $a_{i,j} \in \widetilde{\mathcal{C}}$ and therefore $m_{S_W}(a_{i,j}) =3$. 

To determine $\# \Pi_{L_n}^*(W)$, we consider a reordering of vertices such that the first vertex is $j$. We start by counting the number of possible choices for the circuit $\pi_j$. Since $a_{i,j} \in \widetilde{\mathcal{C}}$ and $m_{S_W}(a_{i,j})=3$, by the construction of $\widetilde{\mathcal{C}}$ it follows that $m_{S_{w_j}}(a_{i,j})=1.$ Therefore by Lemma \ref{lemma:choice of pi_t reordered}, the maximum number of choice for $\pi_j$ is $B^pN^{F_j-1}$, where $F_j$ is the number of generating vertices of $w_j$. For the rest of the circuits $\pi_r$, the maximum number of possibilities for $\pi_r$ is $B^pN^{F_r}$, where $F_r$ is the number of generating vertices of $w_r$ with respect to the new ordering. Therefore we get that $\#\Pi_{L_n}^*(W)\leq B^{kp}N^{|W|+k-1}$ and
$$|W|+(k-1) \leq \frac{pk-3(k-1)}{2}+(k-1)+(k-1)=\frac{pk+k-1}{2}.$$
Hence $\frac{1}{N^{\frac{pk+k}{2}}} \#\Pi_{L_n}^*(W)\leq\frac{B^{pk}}{\sqrt{N}}$.
\noindent
Now, we are left with the following case.
\vskip3pt
\noindent\textbf{Case 3:} $R=1$ for all choice of renumbering $\zeta$ obeying (\ref{eqn: condition zeta}).

We fix a renumbering $\zeta$ obeying (\ref{eqn: condition zeta}). Suppose $U \geq 1$, then by Lemma \ref{prop: Pi_L F-R}, $\#\Pi_{L_n}^*(W)\leq B^{pk} N^{F-R}=B^{pk}N^{F-1}$. Also, we get
$$F-1=|W|+k-1 \leq \frac{pk-4U-3(k-2-U)-2}{2}+(k-1)+k-1=\frac{pk}{2}+\frac{k}{2}-\frac{U}{2}.$$
As a result, $\frac{1}{N^{\frac{pk+k}{2}}} \#\Pi_{L_n}^*(W)\leq B^{pk}N^{-U/2} \leq \frac{B^{pk}}{\sqrt{N}},$ for all $U \geq 1$. 

Now suppose $U=0$. Since $R=1$, there exists an edge $\{i,j\} \in E(T_W)$ with $\zeta(i) < \zeta(j)$, such that $a_{i,j} \in \mathcal{C}$. We fix the edge $\{i,j\}$ for the rest of this proof. Suppose $m_{S_W}(a_{i,j}) \geq 3$. Then we have $F-R \leq \frac{pk-3(k-1)}{2}+(k-1)+k-1$, and thus $\frac{1}{N^{\frac{pk+k}{2}}} \#\Pi_{L_n}^*(W) \leq \frac{B^{pk}}{\sqrt{N}}$.

Hence, it is sufficient to prove for the case where $m_{S_W}(a_{i,j})=2$. We state the following claim which is easy to see.

\noindent \textbf{Claim 1:} There exists a leaf $i^\prime$ such that $\{i^\prime,j^\prime\} \in E(T_W)$, $i \neq i^\prime, j^\prime$ and $\operatorname{dist}(i^\prime, \mathbb{O}(T_W))=\max\{\operatorname{dist}(\mathbb{O}(T_W),v):v \in V_W\}$.

Consider $i^\prime,j^\prime$ from the above claim. Since $m_{S_W}(a_{i,j})=2$, $a_{i,j}$ belongs to $\mathcal{C}$ for all choice of $\zeta$ obeying \eqref{eqn: condition zeta} and this implies $a_{i^\prime, j^\prime} \in \widetilde{\mathcal{C}}$ for all choice of $\zeta$ obeying (\ref{eqn: condition zeta}).
Using Claim 1, without loss of generality assume $\zeta(i^\prime)=1$. Combining $\zeta(i^\prime)=1$ with the deductions $a_{i^\prime,j^\prime} \in \widetilde{\CC}$ and $U=0$, we get $m_{S_{w_{i^\prime}}}(a_{i^\prime, j^\prime})=2$ and $m_{S_{w_{j^\prime}}}(a_{i^\prime, j^\prime})=1$. 

We determine the cardinality of $\Pi_{L_n}^*(W)$ in the following way. We start by counting the number of choices for $\pi_i \in \Pi_{L_n}^*(w_i)$. Since $m_{S_{w_i}}(a_{i,j})=1$, by Lemma \ref{lemma:choice of pi_t reordered}, the number of choices for $\pi_i$ is at most $B^pN^{F_i-1}$, where $F_i$ is the number of generating vertices of $w_i$. Next, we count the number of possibilities for $\pi_{j^\prime}$ such that $(\pi_i,\pi_{j^\prime}) \in \Pi_{L_n}^*(w_i,w_{j^\prime})$. Since $i \neq i^\prime,j^\prime$ and $m_{S_W}(a_{i^\prime,j^\prime})=3$ with $m_{S_{w_{i^\prime}}}(a_{i^\prime,j^\prime})=2$ and $m_{S_{w_{j^\prime}}}(a_{i^\prime,j^\prime})=1$, it follows that $a_{i^\prime,j^\prime}$ has not appeared in $w_i$. Thus, $a_{i^\prime,j^\prime}$ is a new letter that appears only once in $w_{j^\prime}$. 
Hence by Lemma \ref{lemma:choice of pi_t reordered}, the maximum number of choices for $\pi_{j^\prime}$ is $B^pN^{F_{j^\prime}-1}$, where $F_{j^\prime}$ is the number of generating vertices of $w_{j^\prime}$ with respect to the new ordering. For $r \neq i,j^\prime$, the maximum number of possibilities for $\pi_r$ is at most $B^pN^{F_r}$, where $F_r$ is the number of generating vertices of $w_r$ with respect to the new ordering. Therefore, we get	$\#\Pi_{L_n}^*(W) \leq B^{pk}N^{|W|+k-2}$. We also have the inequality
$$|W|+k-2 \leq \frac{pk-3(k-2)-2}{2}+k-1+k-2=\frac{pk}{2}+\frac{k}{2}-1.$$
Hence, it follows that $\frac{1}{N^{\frac{pk+k}{2}}}\#\Pi_{L_n}^*(W) \leq \frac{B^{pk}}{\sqrt{N}}$ and this completes the proof.
\end{proof}
\begin{figure}
	\includegraphics[height=40mm,width=150mm]{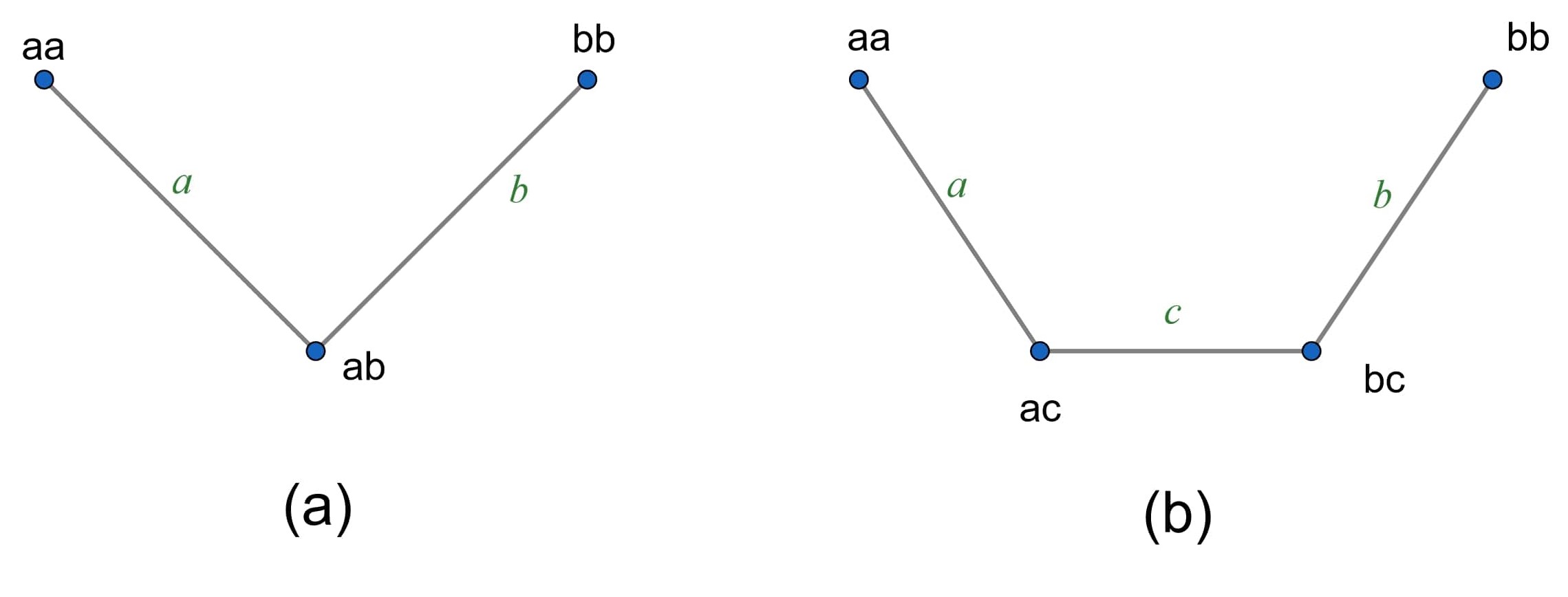}
	\caption{(a) is an example of a sentence such that $R=0$ for all choice of $\zeta$ obeying \eqref{eqn: condition zeta} and (b) is an example of a sentence such that $R=1$ for all choice of $\zeta$ obeying \eqref{eqn: condition zeta}. In both the figures, for an edge $e$, the letter $\varphi(e)$ is given along the edge.}\label{fig:R 0,1}
\end{figure}
\begin{remark}
	We remark that the cases 1 and 2 in the proof of Lemma \ref{prop:Pi C_k,p} are not disjoint. But, none of the above cases are also subcase of the other case. See Figure \ref{fig:R 0,1} for examples of sentences $W$ that fall under only one of the cases.
\end{remark}
\begin{remark}\label{remark:Pi L upper bound}
In Lemma \ref{prop:Pi C_k,p}, the degree of freedom for $ \Pi_{L_n}^*(W)$ is obtained by subtracting $R$ from the upper bound on the number of generating vertices in $W$, with respect to $A= \phi$. Suppose $W \in \mathcal{W}_{p,k}$, $A \subset S_W$ is non-empty and $m_{S_W}(a) \geq 2$ for all $a \notin A$. Suppose there exists a spanning tree $T_W$ of $G_W$ and a one-one function $\varphi: E(T_W) \rightarrow S_W$ obeying condition (iii) of Definition \ref{defn:C_k,p}. Then the upper bound on the number of generating vertices of $W$ in this case, is less than the case when $A = \phi$. Hence here also, $\frac{1}{N^{\frac{pk+k}{2}}}\# \Pi_{L_n}^*(W) \leq \frac{B^{pk}}{\sqrt{N}}$ for all $k \geq 3$. 
\end{remark}

\begin{figure}[h!]
	\includegraphics[height=42mm,width=120mm]{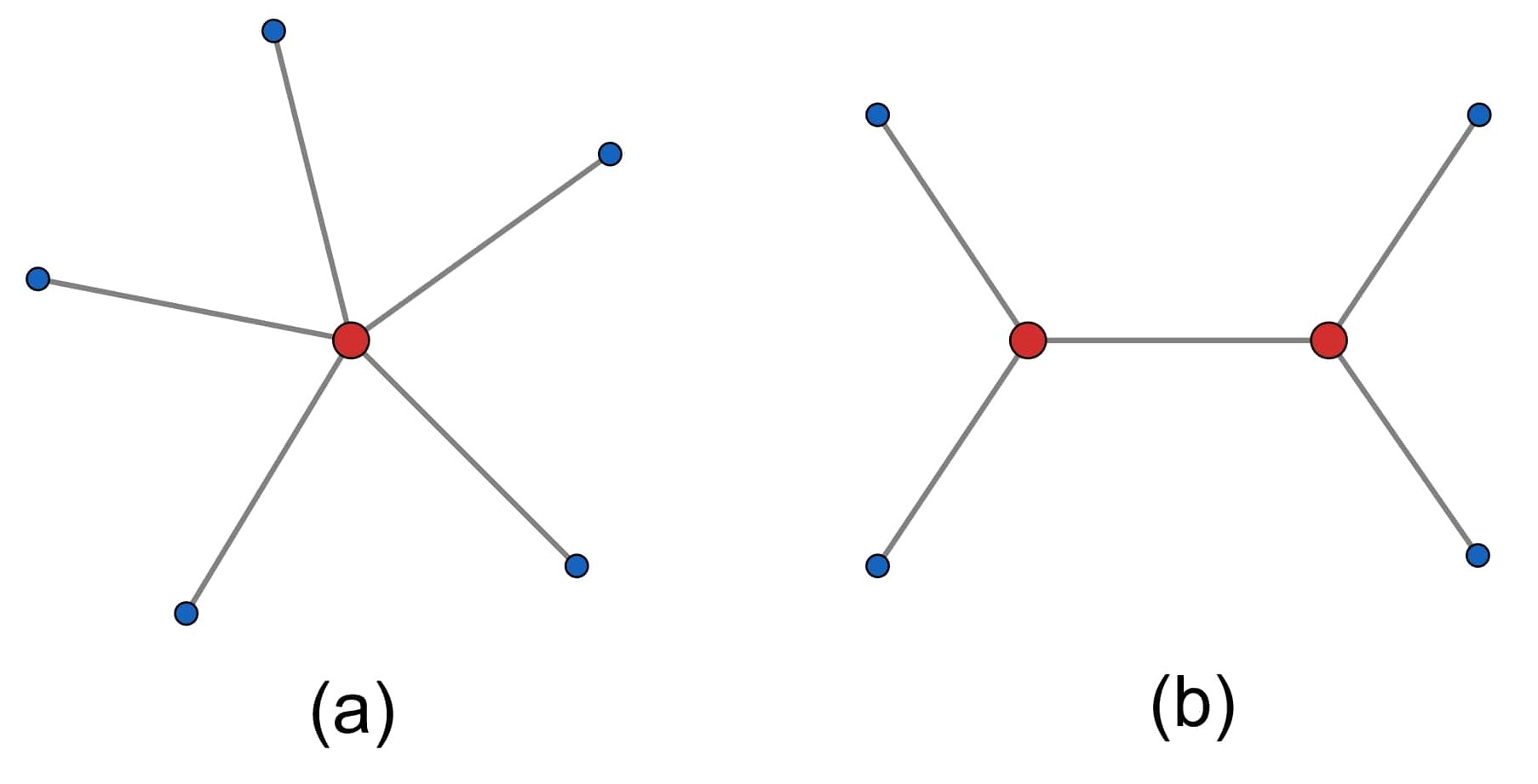}
	\caption{Trees such that the distance of all vertices from the center is one. (a) shows the case when the center consists of only one vertex and (b) shows the case when the center consists of two vertices.} \label{fig:CaseR=1}  
\end{figure}
Next we prove Lemma \ref{prop: Pi_L F-R}. The proof is by induction on $r=\max\{\operatorname{dist}\left(\mathbb{O}(T_W),v\right):v \in V_W\}$, where $\mathbb{O}(T_W)$ is the center of $T_W$.
\begin{proof}[Proof of Lemma \ref{prop: Pi_L F-R}]
	Let $r=\max\{\operatorname{dist}\left(\mathbb{O}(T_W),v\right):v \in V_W\}$ and $A \subseteq S_W$. For $r=1$, $T_W$ is one of the forms given in Figure \ref{fig:CaseR=1}. First suppose that $T_W$ is a star graph as considered in Figure \ref{fig:CaseR=1} (a). Without loss of generality assume that $\zeta$ is the identity map on $[k]$. Then the image set of the map $\varphi$ is given by $\{a_{i,k}:1 \leq i \leq k-1\}$, where $a_{i,k}$ is as defined in \eqref{eqn: defn ai,j}. Note that the number of choices for $\pi_1$ is bounded by $B^pN^{F_1}$ if $a_{1,k}$ appears more than once in $S_{w_1}$ and is bounded above by $B^pN^{F_1-1}$ if $a_{1,k}$ appears only once in $S_{w_1}$, where $F_1$ is the number of generating vertices of $w_1$ with respect to $\zeta$. For $2 \leq i \leq k$, once $\pi_1,\pi_2,\ldots,\pi_{i-1}$ are chosen, the number of choices for $\pi_i$ is bounded above by $B^pN^{F_i-1}$ if $a_{i,k} \in \mathcal{C}_A$ and is bounded above by $B^pN^{F_i}$ if $a_{i,k} \notin \mathcal{C}_A$, where $F_i$ is the number of generating vertices of $w_i$ with respect to $\zeta$. This implies that $\# \Pi_{L_n}(W) \leq B^{pk}N^{F-R}$, where $R=\# \CC_A$, completing the proof for this case.

Now, for $T_W$ as the graph considered in Figure \ref{fig:CaseR=1} (b), note that a similar argument again works. Combining the above observations, we get
the result for $r=1$.

Now, we move to the induction step: Suppose the result holds for all sentences $W$ with a spanning tree $T_W$ such that $\max\{\operatorname{dist}\left(\mathbb{O}(T_W),v\right):v \in V_W\} \leq r $ and all choices of $A \subset S_W$. Let $W$ be a sentence with a spanning tree $T_W$ such that $\max\{\operatorname{dist}\left(\mathbb{O}(T_W),v\right):v \in V_W\}=(r+1)$ and assume without loss of generality that $\zeta$ is the identity map. Let $W^\prime$ be the sentence obtained from $W$ after the deletion of all words $w_i$ such that $\operatorname{dist}(i,\mathbb{O}(T_W))=r+1$ (see Figure \ref{fig: construction TW prime}). Consider the spanning tree $T_{W^\prime}$ of $G_{W^\prime}$ obtained by the restriction of $T_W$ to $V(G_{W^\prime})$ and define $\varphi^\prime= \varphi|_{E(W^\prime)}$. Define $\zeta^\prime: V(G_{W^\prime}) \rightarrow
\{1,2,\ldots ,k-q \}$ given by $\zeta^\prime(i):= \zeta (i)-q$, where $q$ is the number of vertices in $T_W$ at a distance of $r+1$ from the center of $T_W$. It follows that $\zeta^\prime$ is a bijection obeying (\ref{eqn: condition zeta}). Also, define $$A^\prime= \left(A \cup  \bigcup_{i: \operatorname{dist}(i,\mathbb{O}(T_W))=r+1}  S_{w_i} \right)\cap S_{W^\prime}.$$ 
By induction hypothesis it follows that $\# \Pi_{L_n}^*(W^\prime)\leq B^{(k-q)p}N^{F_1-R_1}$, where $F_1$ is the number of generating vertices of $W^\prime$ given $A^\prime$ and 
$$R_1= \# \{a_{i,j} \in \operatorname{Im}(\varphi^\prime): a_{i,j} \notin S_{w_\ell} \,\, \forall \ \zeta^\prime(\ell)<\zeta^\prime(i), a_{i,j} \notin A^\prime \text{ and } m_{S_{w_i}}(a_{i,j})=1\}.$$

Note that $\zeta^\prime(\ell)<\zeta^\prime(i)$ if and only if $\zeta(\ell)<\zeta(i)$ and $\operatorname{dist}(\mathbb{O}(T_W),\ell) \leq r$. Further for $a_{i,j} \in \operatorname{Im}(\varphi^\prime)$, $a_{i,j} \notin A^\prime$ if and only if $a_{i,j} \notin A$ and $a_{i,j} \notin S_{w_\ell}$ for any $\ell$ such that $\operatorname{dist}(\mathbb{O}(T_W),\ell)=r+1$. Thus $R_1$ can be given by $$R_1= \# \{a_{i,j} \in \operatorname{Im}(\varphi^\prime): a_{i,j} \notin S_{w_\ell} \,\, \forall \ \zeta(\ell)<\zeta(i), a_{i,j} \notin A \text{ and } m_{S_{w_i}}(a_{i,j})=1\}.$$ Therefore $R=R_1+R_2$, where $$R_2= \# \{a_{i,j} \in \operatorname{Im}(\varphi) \setminus \operatorname{Im}(\varphi^\prime): a_{i,j} \notin S_{w_\ell} \,\, \forall \ \zeta(\ell)<\zeta(i), a_{i,j} \notin A \text{ and } m_{S_{w_i}}(a_{i,j})=1\}.$$
Now by applying Lemma \ref{lemma:choice of pi_t reordered}, we get that
$$\# \Pi_{L_n}^*(W) \leq \# \Pi_{L_n}^*(W^\prime) \times B^{pq}N^{F_2-R_2},$$
where $F_2$ is the number of generating vertices of the form $(u,j)$ where $\operatorname{dist}(\mathbb{O}(T_W),u)=r+1$. Hence, we have $\# \Pi_{L_n}^*(W)\leq B^{(k-q)p}N^{F_1-R_1} \times B^{pq}N^{F_2-R_2}$ and the result follows since $F_1+F_2=F$ and $R_1+R_2=R$.
\end{proof}

\begin{figure}
	\includegraphics[height=40mm,width=150mm]{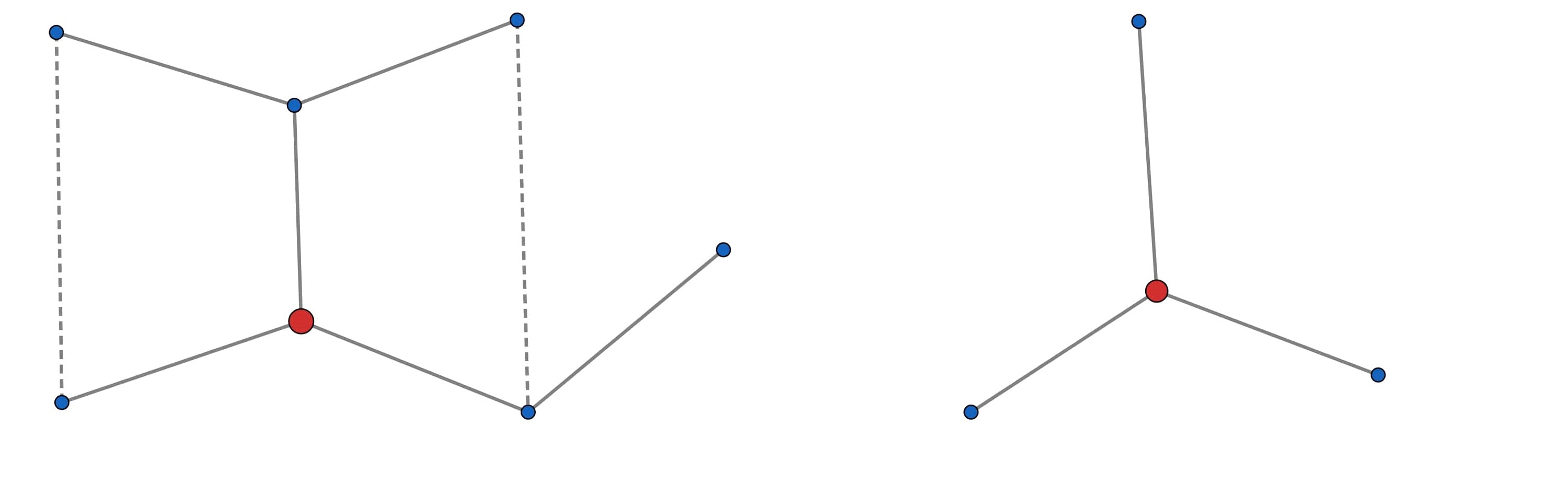}
	\caption{(The construction of the subgraph $T_{W^\prime}$ from the graph $G_W$) The figure in the left represents the graph $G_W$. The center of $T_W$ is denoted by a larger red dot. The solid lines represent the edges of a spanning tree $T_W$ and the dotted lines the edges of $G_W$ not present in $T_W$. The figure in the right, $T_{W^\prime}$, is obtained from $T_W$ by removing the vertices at the distance two from the center of $T_W.$}\label{fig: construction TW prime}
\end{figure}
\noindent\textbf{Step II:} In this step we establish relation between clique cluster and Type I subgraph.
\begin{lemma}\label{prop:existence_Type1subg}
	For $p \geq 1$ and $ k \geq 3$, let $W =(w_1,w_2,\ldots, w_k) \in \mathcal{W}_{p,k}$ be a sentence such that each letter in $S_W$ is repeated at least twice. If $G_W$ is connected but not a clique cluster, then $G_W$ has a Type I subgraph.
\end{lemma}
\begin{proof}
	The idea of the proof is to provide an explicit construction of the subforest $H_W$ with components $T_1,T_2,\ldots, T_r$ along with the one-one map $\varphi$.

Since $G_W$ is connected, but not a clique cluster, there exist distinct vertices $i_1,i_2,i_3$ and letters $a_1 \in S_{w_{i_1}} \cap S_{w_{i_2}}$ and $a_2 \in S_{w_{i_2}} \cap S_{w_{i_3}}$ with $a_1 \neq a_2$. Add the vertices $i_1,i_2,i_3$ to $V(T_1)$ and add the edges $\{i_1,i_2\},\{i_2,i_3\}$ to $E(T_1)$ and define $\varphi\left(\{i_1,i_2\}\right)=a_1$, $\varphi\left(\{i_2,i_3\}\right)=a_2$.

Suppose for all $x \in  S_{w_{i_1}} \uplus  S_{w_{i_2}} \uplus  S_{w_{i_3}}$, $x$ appears at least twice in $S_{w_{i_1}} \uplus  S_{w_{i_2}} \uplus  S_{w_{i_3}}$. Then define $V(T_1)=\{i_1,i_2,i_3\}$. If not, then there exists a letter $a_3 \in S_{w_{i_1}} \uplus  S_{w_{i_2}} \uplus  S_{w_{i_3}}$ such that $a_3$ appears only once in $S_{w_{i_1}} \uplus  S_{w_{i_2}} \uplus  S_{w_{i_3}}$. Since every letter appears at least twice in $S_W$, there exists $i_4 \neq i_1,i_2,i_3$ such that $a_3 \in S_{w_{i_4}}$. Let $i \in \{i_1,i_2,i_3\}$ be such that $a_3 \in S_{w_i}$, add the edge $\{i,i_4\}$ to the graph and define $\varphi\left(\{i,i_4\}\right)=a_3$. Since $G_W$ is a finite graph, this process should terminate in at most $k_1$ steps where $k_1 \leq k$. Define $T_1$ as the graph with vertex set $\{i_1,i_2,\ldots, i_{k_1}\}$ constructed as above. Note that by construction, $T_1$ is a tree such that each letter in  $S_{W_{T_1}}$ is repeated at least twice and $\varphi: E(T_1) \rightarrow S_W$ is a one-one map. If $\# V(T_1)=k$, then $H_W=T_1$ is the Type I subgraph.

Suppose we constructed disjoint subtrees  $T_1,T_2,\ldots, T_m$ with $\# V(T_i)=k_i$, $\kappa=(k_1+k_2+\cdots + k_m) <k$ and $k_i \geq 2$ for all $1 \leq i \leq m$ such that $T_1,T_2,\ldots ,T_m$ obey condition (iii) of Definition \ref{defn:Type I subgraph} and there exists a one-one function $\varphi: \cup_{i=1}^m E(T_i) \rightarrow S_W$ obeying condition (iv). Define
\begin{align*}
	M_{m}:=  \{i  &\in V(G_W) \setminus \left( V(T_1) \cup V(T_2) \cup \cdots \cup V(T_m) \right) :  \\ 
	&\exists \ a_{\kappa+1} \in S_{w_i} \text{ such that }  m_{S_{w_i}}(a_{\kappa+1})=1 \text{ and } a_{\kappa+1} \notin S_{W_{T_1}} \cup S_{W_{T_2}} \cup \cdots \cup S_{W_{T_m}}\}.
\end{align*}
If $M_{m}$ is empty, then define $H_W$  by choosing all the vertices in $G_W$ not belonging to $V(T_1) \cup V(T_2) \cup\cdots \cup V(T_m)$ as isolated vertices in $H_W$. By construction, $H_W$ is a forest. Further, note that as all isolated vertices in $G_W$ occur at the last by construction, condition (ii) is also satisfied. Then since $M_{m}$ is empty and by the assumption on $T_1,T_2,\ldots ,T_m$, $H_W$ satisfies condition (iii) of Definition \ref{defn:Type I subgraph}. Condition (iv) is satisfied as no new edges are added to $H_W$ afterwards. As a result, we get that $H_W$ is a Type I subgraph.

Suppose $M_{m}$ is not empty. Then define $i_{\kappa+1}= \min M_{m}$. Since each letter in $S_W$ appears at least twice, there exists $i_{\kappa +2} \notin  V(T_1) \cup \cdots \cup V(T_m) $, $i_{\kappa +2} \neq i_{\kappa+1}$ such that $a_{\kappa+1} \in S_{w_{i_{\kappa+2}}}$. Construct the edge $\{i_{\kappa+1},i_{\kappa+2}\}$ and define $\varphi\left(\{i_{\kappa+1},i_{\kappa+2}\}\right)=a_{\kappa+1}$. 

Suppose for all  $x \in  S_{w_{i_{\kappa+1}}} \uplus  S_{w_{i_{\kappa+2}}}$, $x$ appears at least twice in $S_{W_{T_1}} \uplus  S_{W_{T_2}} \uplus \cdots \uplus S_{W_{T_m}} \uplus  S_{w_{i_{\kappa+1}}} \uplus  S_{w_{i_{\kappa+2}}} $, we define $V(T_{m+1})= \{i_{\kappa+1},i_{\kappa+2}\}$ and $E(T_{m+1})=\{\{i_{\kappa+1},i_{\kappa+2}\}\}$. If not, then there exists $a_{\kappa+2}$ that appears only once in $S_{W_{T_1}} \uplus  S_{W_{T_2}} \uplus \cdots \uplus  S_{W_{T_m}} \uplus  S_{w_{i_{\kappa+1}}} \uplus  S_{w_{i_{\kappa+2}}} $, and hence there exists $i_{\kappa+3}$ such that $i_{\kappa+3} \notin V(T_1) \cup \cdots \cup V(T_m), i_{\kappa+3} \neq i_{\kappa+1},i_{\kappa+2}$ and $a_{\kappa+2} \in S_{w_{i_{\kappa+3}}}$. Let $i \in \{i_{\kappa+1},i_{\kappa+2}\}$ such that $a_{i_{\kappa+2}} \in S_{w_i}$. Construct the edge $\{i,i_{\kappa+3}\}$ and define $\varphi\left(\{i,i_{\kappa+3}\}\right)=a_{i_{\kappa+2}}$. Continue this construction to obtain a tree $T_{m+1}$ such that for all $x \in S_{W_{T_{m+1}}}$, $x$ appears at least twice in $S_{W_{T_1}} \uplus S_{W_{T_2}} \uplus \cdots \uplus S_{W_{T_{m+1}}}$. Thus, we have obtained a tree $T_{m+1}$ such that $T_1,T_2,\ldots ,T_{m+1}$ obey condition (iii) of Definition \ref{defn:Type I subgraph} and there exists a one-one function $\varphi: \cup_{i=1}^{m+1} E(T_i) \rightarrow S_W$ obeying condition (iv).  

Continue the construction of trees, and since $G_W$ is finite, the construction of trees should stop at some finite step. Define $H_W$ as the union of the subtrees along with the isolated vertices. Then by construction, all conditions of Type I subgraph are satisfied by $H_W$. Therefore $H_W$ is a Type I subgraph and the proof is complete.
\end{proof}

\noindent\textbf{Step III:} In this final step, using Lemmas \ref{prop:Pi C_k,p}, \ref{prop: Pi_L F-R} and \ref{prop:existence_Type1subg} we prove Proposition \ref{lemma: Bpl general}.

\begin{proof}[Proof of Proposition \ref{lemma: Bpl general}]
	First, consider a sentence $W \in \WW_{p,k}$ such that $G_W$ is a clique cluster but $W$ is not a clique sentence and each letter in $S_W$ appearing at least twice. Let $x$ denote the common letter in $W$. If there exists $i$ such that $m_{S_{w_i}}(x)=1$, then from Lemma \ref{lemma:choices pi t}, we have that $\# \Pi_{L_n}^*(W)=B^{kp}N^{|W|+k-1}$. Since each letter appears at least twice and $W$ is not a clique sentence, it follows that $|W| \leq \frac{pk-k-1}{2}+1$, proving the result. Now, if $m_{S_{w_i}}(x) \geq 2$ for all $i$, then $|W| \leq \frac{pk-2k}{2}+1$ and the result follows.

Consider a sentence $W =(w_1,w_2,\ldots, w_k) \in \mathcal{W}_{p,k}$ such that $G_W$ is connected but not a clique cluster, and each letter in $S_W$ appearing at least twice. Then by Lemma \ref{prop:existence_Type1subg}, $G_W$ always has a Type I subgraph. Let $H_1$ and $H_2$ be Type I subgraphs of $G_W$ with one-one maps $\varphi_1$ and $\varphi_2$, respectively. We say $H_1 \leq H_2$, if $H_1$ is a subgraph of $H_2$ and the restriction of $\varphi_2$ to the edge set of $H_1$ is $\varphi_1$, i.e. $\varphi_2|_{E(H_1)}=\varphi_1$. The relation $\leq$ is a partial order on the set of all Type I subgraphs of $G_W$ and we call a maximal element with respect to $\leq$, a maximal Type I subgraph. Since $G_W$ is a finite graph, by Lemma \ref{prop:existence_Type1subg}, maximal Type I subgraphs always exist. 

Let $H_W$ be a maximal Type I subgraph of $G_W$ with associated one-one function $\varphi_W$ and let $T_1,T_2,\ldots, T_r$ be the connected components of $H_W$ with $\# V(T_i)=k_i$ for all $i$.

Since $H_W$ is a Type I subgraph, $\# V(T_1) \geq 3$. Furthermore, by condition (iii) of Type I subgraph, each letter in $S_{W_{T_1}}$ appears at least twice and therefore, $W_{T_1} \in \CC_{p,k_1} $ where $k_1 \geq 3$ and $\CC_{p,k_1}$ is as defined in Definition \ref{defn:C_k,p}. Thus by Lemma \ref{prop:Pi C_k,p}, it follows that $\frac{1}{N^{(pk_1+k_1)/2}}\# \{\text{choices of } \pi_{T_1}\}$ is of the order $o(1)$. Suppose $\pi_{T_1},\pi_{T_2},\ldots, \pi_{T_{m-1}}$ are chosen. Then, by Remark \ref{remark:Pi L upper bound}, it follows that the number of choices of $\pi_{T_{m}}$ are given by
\begin{align*}
	\frac{1}{N^{\frac{pk_m+k_m}{2}}}\# \{\text{choices of } \pi_{T_m}\} \leq
	\left\{\begin{array}{lll}
		B^{pk_m} & \ \text{ if } k_m=2, \\
	\frac{B^{pk_m}}{\sqrt{N}} & \ \text{ if } k_m \geq 3.
	\end{array} \right.
\end{align*}
Next we prove a claim for the case when $k_m=1$.
\vskip3pt
\noindent\textbf{Claim 1:} Let $m$ be such that $k_m=1$. Then there exists $a \in S_{W_{T_m}}$ such that $a$ appears at least once in $\uplus_{i \leq m-1} S_{W_{T_i}}$.
\begin{proof}[Proof of Claim 1]
	Suppose for all $a \in S_{W_{T_m}}$, $a$ has not appeared in $\uplus_{i \leq m-1} S_{W_{T_i}}$. Since $G_W$ is connected there exist $q>m$ and $x \in S_{W_{T_m}}$ such that $x \in S_{W_{T_q}}$. Join the vertices in $T_m$ and $T_q$ by an edge, say $e$, and define $\varphi(e)=x$. Using the construction in Lemma \ref{prop:existence_Type1subg}, extend this graph to a Type I subgraph $H$ with the one-one function $\varphi$. Then it follows that $H_W$ is strictly contained in $H$, contradicting the maximality of $H_W$. Thus we have obtained a contradiction.
\end{proof}

For $k_m=1$, it follows from Claim 1 that there exists $x \in S_{W_{T_m}}$ such that $x \in S_{W_{T_i}}$ for some $i<m$. Thus, once all $\pi_{T_i}$ such that $i<m$ are chosen, the $L$-value of the letter $x$ is fixed. Furthermore, by condition (iii) of Type I subgraph, each new letter in $S_{W_{T_m}}$ appears at least twice in $S_{W_{T_m}}$. Thus by Lemma \ref{lemma:choice of pi_t reordered}, the number of choices for $\pi_{T_m}$ is at most $B^{p}N^{\frac{p-1}{2}+1}= B^{pk_m}N^{\frac{pk_m+k_m}{2}}$. Hence, 
\begin{align*}
	\text{the number of choices for }(\pi_1,\pi_2,\ldots, \pi_k) &\leq B^{pk_1} N^{\frac{pk_1+k_1-1}{2}} \times B^{pk_2}N^{\frac{pk_2+k_2}{2}} \times \cdots \times B^{pk_r}N^{\frac{pk_r+k_r}{2}}  \\
	&= B^{pk}N^{\frac{pk+k-1}{2}},
\end{align*}
where the last equality follows since $k_1+k_2+\cdots+ k_r=k$. This proves the lemma.
\end{proof}




\section{ \textbf{Applications of Theorem \ref{thm: always normal}}}\label{sec:LES+example_ind}

In this section, we show that Theorem \ref{thm: always normal} can be used to establish CLT for the LES of a large class of random matrices for even degree monomial test functions.
In Section \ref{subsec:common_patterned_eg}, we prove LES results for classical patterned random matrices and in Section \ref{subsec:block_patterned_eg}, we prove LES results for block patterned matrices. In Section \ref{subsec:other_model_eg}, we discuss LES of $d$-disco matrix, the swirl matrix, checkerboard matrix and a class of Hankel-type matrices. Through out the section, we always treat the constants $c_1,c_2$ and the sets $\mathcal{Z}_n, \mathcal{S}_n$ as used in Theorem \ref{thm: always normal}, and we use the notations
	\begin{align*}
		\eta_{p} &=\frac{1}{\sqrt{\var\left(\Tr\left(\frac{A_n}{\sqrt{N}}\right)^p\right)}}\left(\Tr\left(\frac{A_n}{\sqrt{N}}\right)^p-\E \Tr \left(\frac{A_n}{\sqrt{N}}\right)^p\right),\\
		\xi_p &=\frac{1}{\sqrt{N}}\left(\Tr\left(\frac{A_n}{\sqrt{N}}\right)^p-\E \Tr \left(\frac{A_n}{\sqrt{N}}\right)^p\right),
\end{align*}
to denote the varying $p$ and fixed $p$ cases, respectively.

\subsection{Classical patterned random matrices}\label{subsec:common_patterned_eg}
In this subsection,  we prove CLT for the LES of classical patterned random matrices. For all the matrices considered in this subsection,  $N(n)=n$. We write the detailed proof for the non-banded case, but it can also be easily seen that the same arguments hold for the banded case as well as other choices of $\Delta_n$ satisfying Assumption \ref{assump: existence Delta}. The approach we adopt here is to present a brief literature overview on the LES of the patterned random matrix along with the new proof using Theorem \ref{thm: always normal}. 

\noindent \textbf{(i) Symmetric Toeplitz matrix:} Chatterjee  \cite{chatterjee2009fluctuations} was the first to study the LES of symmetric Toeplitz matrices with random entries. He showed that for test functions $\phi(x) = x^{p_n}$, where $p_n = o(\log n/ \log \log n)$, the LES of symmetric Toeplitz matrices with independent Gaussian input entries converges to a Gaussian distribution under the total variation norm. 
Later, Liu et al.  \cite{liu2012fluctuations} studied the LES of symmetric band Toeplitz matrices for non-Gaussian entries and diagonal entries as zero. They showed that for fixed polynomials, the LES converges in distribution to a Gaussian distribution. 
The techniques in \cite{liu2012fluctuations} are only applicable for band Toeplitz matrices with the bandwidth $b_n$ such that $b_n/n \rightarrow b$ as $n \rightarrow \infty$, and are not applicable for anti-band Toeplitz matrices. 

To simplify the calculations, we present the case of the usual symmetric Toeplitz matrix but the same idea also works for banded and anti-banded Toeplitz matrices of bandwidth $\Theta(n)$. Consider the sets
\begin{align*}
\mathcal{Z}_n &=\{0,1,\ldots,\lfloor n/2 \rfloor \} \text{ and}\\
\mathcal{S}_n &= \{(i,j):1 \leq i \leq \lfloor n/2 \rfloor , 1 \leq j \leq \lfloor n/2 \rfloor+i \} \\ & \hspace{30mm}\cup 
\{(i,j): \lfloor n/2 \rfloor+ 1\leq i \leq n, i+1-\lfloor n/2 \rfloor  \leq j \leq n \}.  
\end{align*} 
Then it is easy to see that (see Figure \ref{fig:Toeplitz_Hankel_Colour}(b)) that
\begin{align*}
&\# \{(i,j):L(i,j)=r_u\} \geq n/2 \text{ for all } r_u \in \ZZ_n \text{ and} \\
&\# (\mathcal{S}_n \cap \mathcal{R}_i) \geq n/2 \text{ for all rows } \mathcal{R}_i.
\end{align*}
 Now, Theorem \ref{thm: always normal} implies that $\eta_{p_n}$ converges in distribution to standard Gaussian distribution for even natural numbers $p_n=o(\log n / \log \log n)$. 
\begin{figure}[h!]
\centering
\includegraphics[width=160mm,height=80mm]{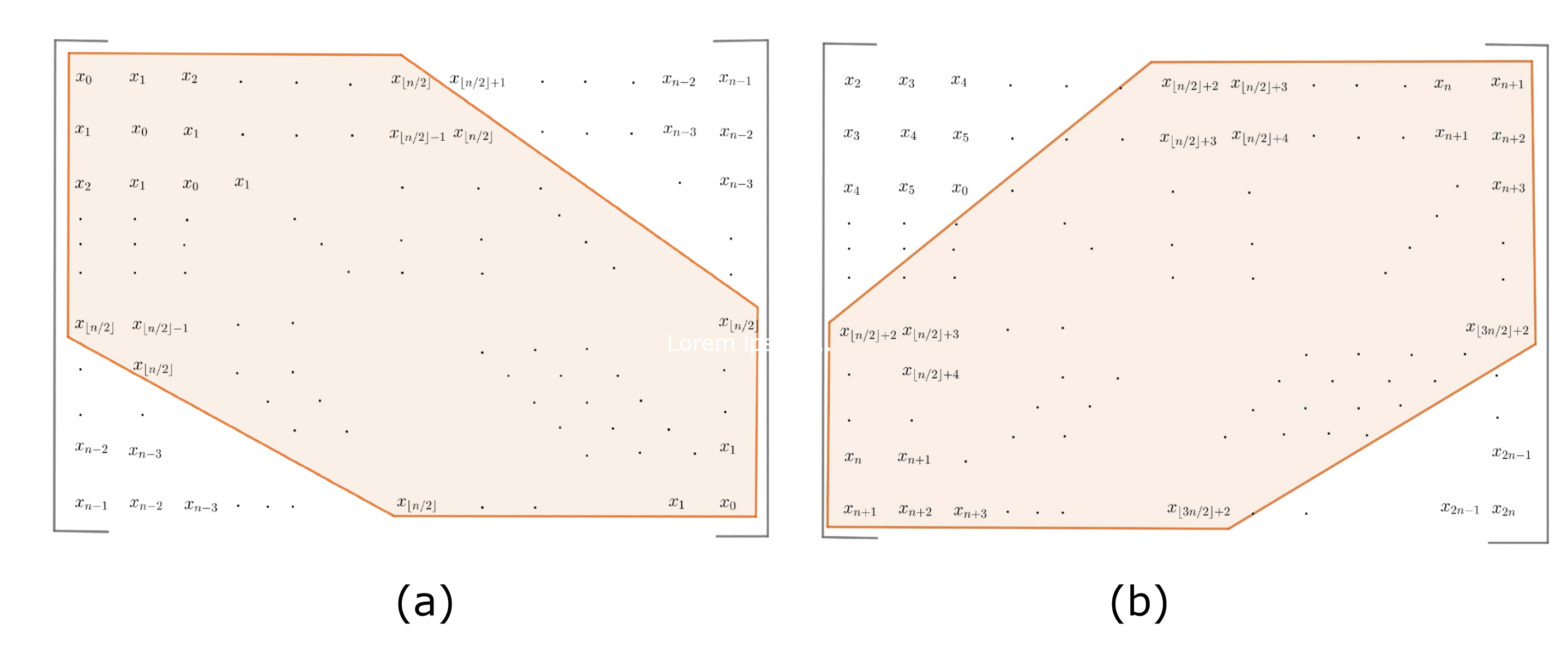}
\caption{Set $\mathcal{S}_n$ is colored in orange for (a) Toeplitz and (b) Hankel.}
\label{fig:Toeplitz_Hankel_Colour}
\end{figure}

Further, it follows that for fixed even $p$ with input entries obeying Assumption \ref{assum: 4 moments}, $\xi_p$ converges in distribution to a Gaussian distribution. For the odd case, the convergence is trickier. In \cite{liu2012fluctuations}, it was shown that for odd $p$, $\xi_p$ converges to the distribution $N(0,\sigma^2)+x_0$, where $N(0,\sigma^2)$ is a normal distribution. Clearly, the $k$-th moment of $N(0,\sigma^2)+x_0$ includes contribution from  all lower moments of $x_0$ as predicted by Theorem \ref{thm:unified_moments_p_odd}. 

\vskip2pt
\noindent \textbf{(ii) Hankel matrix:} It was proved in \cite{adhikari_saha2017} that for Hankel matrices with i.i.d. Gaussian entries and a sequence of even positive integers $p_n=o(\log n / \log \log n)$, $\eta_{p_n}$ converges in total variation norm to the standard Gaussian distribution. The LES of Hankel matrices with general entries was first studied in \cite{liu2012fluctuations}, and it was showed that for fixed even $p$, $\xi_p$ converges to a normal distribution. For odd $p$, the case is different from the Toeplitz case. It was proved in \cite{k+m+s_han_22}, that fixed odd $p$, $\xi_p$ does not converge to a Gaussian random variable. To apply Theorem \ref{thm: always normal}, we choose 
\begin{align*}
\mathcal{Z}_n &=\{\lfloor n/2 \rfloor +2, \lfloor n/2 \rfloor+3, \ldots, \lfloor 3n/2 \rfloor+2 \} \text{ and } \\
\mathcal{S}_n &= \{(i,j): \lfloor n/2 \rfloor+2 \leq i+j \leq \lfloor 3n/2 \rfloor+2 \}. 
\end{align*} 
As a result, we get that $c_1=c_2=1/2$ (see Figure \ref{fig:Toeplitz_Hankel_Colour} for a pictorial explanation). Also, we get that for random Hankel matrices with input entries obeying Assumption \ref{assum: subexponential} and any sequence of even positive integers $p_n=o(\log n/ \log \log n)$, $\eta_{p_n}$ converges in distribution to the standard Gaussian random variable.  

\vskip2pt
\noindent \textbf{(iii) Palindromic Toeplitz and Hankel matrices:}
Palindromic Toeplitz matrices are another important class of patterned random matrices, well-known for having a Gaussian limiting spectral distribution \cite{Massey_Miller_2007_Palin_Toeplitz}. Palindromic Toeplitz matrix is obtained from Toeplitz matrix by adding the additional condition $x_{i}=x_{n-1-i}$ (so that the first row is a palindrome). The techniques in \cite{liu2012fluctuations} used for Toeplitz matrices are not applicable for palindromic Toeplitz matrices and to the best of our knowledge, central limit theorem type results for the LES of palindromic Toeplitz matrix are not known. 
To apply Theorem \ref{thm: always normal}, we choose $\mathcal{Z}_n,\mathcal{S}_n$ as $\mathcal{Z}_n=\{0,1,2,\ldots, \frac{n}{2}\}$, $\mathcal{S}_n=[n]^2$ and $c_1=c_2=1$, we get that for Palindromic Toeplitz matrix, $\eta_{p_n}$ and $\xi_p$ converges to Gaussian distributions.

Palindromic Hankel matrices are obtained from Hankel matrices by imposing the additional condition $x_i=x_{n+3-i}$. Arguments similar to the Palindromic Toeplitz case can be employed for Palindromic Hankel matrix to show that for a sequence of even natural numbers $p_n=o(\log n / \log n \log n)$, the $\eta_{p_n}$ converges to the standard Gaussian distribution. The result for Palindromic Hankel matrices for fixed $p$ case is also a new addition to the literature.
\vskip2pt
\noindent \textbf{(iv) Generalized Toeplitz and Hankel matrices:} Several generalizations and alterations of Toeplitz and Hankel matrices, exists in the literature. We talk about the following generalizations introduced in \cite{Miller_Swanson_2015_General_Toeplitz}.
\begin{enumerate}[(a)]
\item Generalized Toeplitz matrix: For fixed positive integers $\alpha, \beta$, Generalized Toeplitz matrix is the $n \times n$ patterned random matrix given by the link function
$$L_{ \alpha, \beta}(i, j):= \begin{cases}\alpha i-\beta j & \mbox{if } i \leq j, \\ -\beta i+\alpha j & \mbox{if } i>j.\end{cases}$$
For Generalized Toeplitz matrix, we choose 
\begin{align*}
	\ZZ_n &=\{\alpha-\beta,\alpha-2\beta,\ldots,\alpha-\lfloor n/2 \rfloor \beta\},\\
	\mathcal{T}_n &=\{\left(1+k\beta/d,r+k\alpha/d\right)\in [n] \times [n], k \in \N, 0 \leq r \leq \lfloor n/2 \rfloor\}, d=\gcd(\alpha,\beta),\\
	\mathcal{S}_n &=\mathcal{T}_n \cup \mathcal{T}_n^C,\\
	c_1 &= \min \{\frac{dn}{2\alpha},\frac{dn}{2\beta}\} \text{ and } c_2= 1/2.
\end{align*} 
The above choice implies that for even $p_n=o(\log n /\log \log n)$, $\eta_{p_n}$ converges to a Gaussian distribution for Generalized Toeplitz matrix for all choices of positive integers $\alpha$ and $ \beta$.
\item Generalized Hankel matrix: For fixed positive integers $\alpha, \beta$, the link function is given by
$$L_{ \alpha, \beta}(i, j):= 
\begin{cases}\alpha i+\beta j &  \mbox{if } i \leq j, \\ \beta i+\alpha j & \mbox{if } i>j.
\end{cases}$$
For Generalized Hankel matrix, we can choose $\mathcal{Z}_n,\mathcal{S}_n$ in a similar fashion as Generalized Toeplitz matrix and this would imply that for Generalized Hankel matrices $\eta_{p_n}$ converges to Gaussian distribution for all sequences of even positive integers $p_n=o(\log n / \log  \log n)$.
\end{enumerate} 
\vskip2pt
\noindent \textbf{(v) Reverse circulant matrix and symmetric circulant matrix:} Reverse circulant and symmetric circulant matrices of dimension $n \times n$ is defined by the link functions
\begin{align*}
	L(i,j) &=(i+j-2)\bmod n, \\
	L(i,j) &= n/2-|n/2-|i-j||,
\end{align*} 
respectively. The LES of reverse circulant and symmetric circulant matrices were first studied in \cite{adhikari_saha2017}, and it is was proved that for symmetric circulant (similarly, reverse circulant) matrices with i.i.d. Gaussian entries, $\eta_{p_n}$ converges under total variation norm to standard Gaussian distribution for every sequence of positive (even positive) integers $p_n=o(\log n / \log \log n)$. For general entries, the convergence of $\xi_p$ for the above matrices was established in \cite{m&s_RcSc_jmp} (even $p$ for reverse circulant matrices). For reverse circulant matrices, $\xi_p$ converges to zero almost surely for odd values of $p$. Without the scaling $1/\sqrt{n}$, the limiting behavior of the LES of reverse circulant matrix depends on the parity of the order of matrix $n$, see Remark 21 of \cite{m&s_RcSc_jmp}.

The LES for reverse circulant and symmetric circulant matrices were studied in \cite{m&s_RcSc_jmp} using the eigenvalues formula well-known in literature. Thus, the technique of \cite{m&s_RcSc_jmp} cannot be utilized to obtain the LES of banded versions or other choices of $\Delta_n$. From Theorem \ref{thm: always normal}, one can derive the convergence of LES for banded and anti-banded, reverse circulant and symmetric circulant matrices.
To establish the convergence to Gaussian for general entries, we choose $\mathcal{Z}_n =\{0,1,\ldots,n-1\}$, $\mathcal{S}_n=[n]^2$ and the constants $c_1=c_2=1$ and apply Theorem \ref{thm: always normal}.

The new results of this subsection for the fixed even positive integer case are summarized in Table \ref{tab:Classical_new}.


\subsection{Block patterned matrices}\label{subsec:block_patterned_eg}

In this section, we consider different types of symmetric block patterned matrices, as defined in (\ref{eqn:defn_block_patterned}). To the best of our knowledge, all the results in this section are new to the literature.
All the block patterned matrices considered in this section will be of the form:
$$A=(B_{L_1(i,j)})_{i,j=1}^d \text{ with } B_\ell= (x_{\ell,L_2(i,j)})_{i,j=1}^m,$$
where $L_1, L_2$ are appropriate functions. 

Recall the definition of $\Pi_{L}(W)$ from \eqref{defn: Pi L(W)} and note that the cardinality of $\Pi_{L}(W)$ depends on $N$, where $[N]$ is the range set of the circuits in $\Pi_L(W)$. For the above block patterned random matrix, the range set of the circuits in $\Pi_{L_1}(W)$ is $[d]$, of circuits in $\Pi_{L_2}(W)$ is $[m]$, and of circuits in $\Pi_{L}(W)$ is $[md]$. We first prove the following equality:
\begin{proposition}\label{prop: Pi l,L1,L2}
For all $W \in \mathcal{W}_{p,k}$,
\begin{equation*}
	\#\Pi_L^*(W)=\#\Pi_{L_1}^*(W) \times \#\Pi_{L_2}^*(W).
\end{equation*}
\end{proposition}
\begin{proof}
We prove the proposition by showing that the map
\begin{align*}
	\Pi_L^*(W) &\rightarrow \Pi_{L_1}^*(W) \times \Pi_{L_2}^*(W) \\
	\pi_u(i) &\rightarrow (\pi_u^{(1)}(i),\pi_u^{(2)}(i)) \text{ for } u=1,2,\ldots, k
\end{align*}
is a bijection,	where 
\begin{align*}
	\pi_u^{(1)}(i)=& \Bigl\lfloor \frac{\pi_u(i)+m-1}{m} \Bigr\rfloor\ \text{ and } 
	\pi_u^{(2)}(i)=\begin{cases}
		\operatorname{rem}(\pi_u(i),m) &\text{ if } m \nmid \pi_u(i), \\
		m &\text{ if } m | \pi_u(i),
	\end{cases}
\end{align*}
where $\operatorname{rem}(a,b)$ denotes the reminder when $a$ is divided by $b$. Note that as the map $\pi_u:\{0,1,\ldots , p\} \rightarrow \{1,2,\ldots ,md\}$ is a circuit, $\pi_u^{(1)}:\{0,1,\ldots , p\} \rightarrow \{1,2,\ldots, d\}$ and $\pi_u^{(2)}:\{0,1,\ldots , p\} \rightarrow \{1,2,\ldots ,m\}$ are circuits. 

To prove that the map is well-defined, we use the following reasoning: note that for $W \in \mathcal{W}_{p,k}$ and $(\pi_1,\pi_2,\ldots , \pi_k) \in \Pi_{L}^*(W)$, 
\begin{align}\label{eqn:Pi L,L1,L2}
	w_{u_1}[i_1]=w_{u_2}[i_2] \implies &L(\pi_{u_1}(i_1-1),\pi_{u_1}(i_1))= L(\pi_{u_2}(i_2-1),\pi_{u_2}(i_2)) \nonumber\\
	\iff &L_1(\pi_{u_1}^{(1)}(i_1-1),\pi_{u_1}^{(1)}(i_1))=L_1(\pi_{u_2}^{(1)}(i_2-1),\pi_{u_2}^{(1)}(i_2)) \text{ and } \nonumber\\ &L_2(\pi_{u_1}^{(2)}(i_1-1),\pi_{u_1}^{(2)}(i_1))=L_2(\pi_{u_2}^{(2)}(i_2-1),\pi_{u_2}^{(2)}(i_2)).
\end{align}
This implies that $(\pi_1^{(1)},\pi_2^{(1)},\ldots , \pi_k^{(1)}) \in \Pi_{L_1}^*(W)$ and $(\pi_1^{(2)},\pi_2^{(2)},\ldots , \pi_k^{(2)}) \in \Pi_{L_2}^*(W)$. 
Note that the values of $\pi_u^{(1)}(i)$ and $\pi_u^{(2)}(i)$ determine the value of $\pi_u(i)$ uniquely and therefore the map is one-one. To prove that the map is onto consider $(\pi_1^{(1)},\pi_2^{(1)},\ldots , \pi_k^{(1)}) \in \Pi_{L_1}^*(W)$ and $(\pi_1^{(2)},\pi_2^{(2)},\ldots , \pi_k^{(2)}) \in \Pi_{L_2}^*(W)$. Note that for proving that the map is onto, we need to show that $(\pi_1,\pi_2,\ldots,\pi_k)$, where $\pi_u(i)=(\pi_u^{(1)}(i)-1)m+\pi_u^{(2)}(i)$, is an element of $\Pi_{L}^*(W)$. This is equivalent to proving the first line of \eqref{eqn:Pi L,L1,L2}, and that follows from swapping the two sides of the if and only if condition in \eqref{eqn:Pi L,L1,L2}.
\end{proof}

Now we look at special types of block patterned random matrices and obtain their LES results.
\vskip2pt
\noindent \textbf{(i) Block Wigner matrix:} Let $m: \N \to \N$ be an increasing function. Consider the block Wigner matrix given by
$$	A_n=\left[\begin{array}{cccccc}
B_{11}(n) & \cdots & B_{1 d}(n) \\
\vdots & \ddots &\vdots \\
B_{1 d}(n) &  \cdots & B_{dd}(n)
\end{array}\right],
$$
where $d \in \N$ is fixed and for each $n \in \N$, $\{B_{i,j}(n)\}$ are $m(n) \times m(n)$ independent patterned random matrices with link function $L_2$. 

Observe that for each sentence $W \in \WW_{p,k}$, using Proposition \ref{prop: Pi l,L1,L2},
\begin{equation*}
\frac{1}{N^{p+1}} \# \Pi_{L}^*(W)=\frac{1}{d^{p+1}} \# \Pi_{L_1}^*(W) \times \frac{1}{(m(n))^{p+1}} \# \Pi_{L_2}^*(W).
\end{equation*}
Here $L_1$ is the link function of Wigner matrix. Note that the tuple $(\pi_1,\pi_2)$ such that $\pi_u(j)=1$ for all $u,j$, belongs to $\Pi_{L_1}^*(W)$ for all $W$ and since the number of blocks of $A_n$ is fixed, $\frac{1}{d^{p+1}} \# \Pi_{L_1}^*(W)$ is a constant independent of $n$ and strictly greater than zero. 

Suppose $\{B_{ij}(n)\}$ belong to the class of classical patterned random matrices discussed in Section \ref{subsec:common_patterned_eg}.
Then it follows from the discussions in Section \ref{subsec:common_patterned_eg} that $ \liminf \frac{1}{(m(n))^{p+1}} \# \Pi_{L}^*(W)>0$ for the sentence $W$ defined in \eqref{eqn: condition non-zero P24}, and as a result by Remark \ref{remark:Pi_Ln,Delta_n-star}
$$\liminf \frac{1}{N^{p_n+1}} \# \Pi_{L}(W)>0$$
for the sentence $W$ defined in \eqref{eqn: condition non-zero P24}. 
Therefore $$\liminf \frac{1}{N}\varA = \liminf \frac{1}{N^{p_n+1}} \sum_{W \in SP(p_n,2)} \# \Pi_{L}(W)>0.$$ Hence, by Theorem \ref{thm:converg_LES_even}, it follows that $\eta_{p_n}$ converges to standard Gaussian distribution for even positive integers $p_n=o(\log md / \log \log md)$. Similar conclusion also holds for fixed positive integer $p$.



\vskip2pt
\noindent \textbf{(ii) Block Toeplitz and Block Hankel matrices:} Let $\{B_i: i \in \Z_+ \}$ be a sequence of independent random matrices of size $m \times m$ of the form \eqref{eqn:A_ij_general}. The block Toeplitz matrix of size $md \times md$ with blocks $(B_i)$, is defined as $T_{i,j}=B_{i+j}$. For block Toeplitz matrix with independent Wigner blocks of fixed size, the limiting spectral distribution (LSD) was derived in \cite{Basu_Bose_Ganguly_Block_Toeplitz_LSD}, and LSD for block size growing to infinity was derived in \cite{Li_Liu_wang_Block_Toepliz_2011},\cite{far_spectra_block_toeplitz}. The fluctuations of the LES of block Toeplitz and block Hankel matrices have not been studied yet. 

Consider a block Toeplitz matrix with independent Wigner blocks of fixed size $m$ and $d \rightarrow \infty$. It is easy to see that the argument for Toeplitz matrices works for this case and conditions (i) and (ii) of Theorem \ref{thm: always normal} hold for block Toeplitz matrices as well. From Proposition \ref{prop: Pi l,L1,L2}, we have
\begin{equation*}
\frac{1}{N^{p+1}} \# \Pi_{L}^*(W)=\frac{1}{d^{p+1}} \# \Pi_{L_1}^*(W) \times \frac{1}{m^{p+1}} \# \Pi_{L_2}^*(W).
\end{equation*} 
Note that here $m$, the size of the blocks are fixed, and therefore $\frac{1}{m^{p+1}} \# \Pi_{L_2}^*(W)$ is a constant strictly greater than zero. Further, since $L_1$ is the link function of Toeplitz matrix, we have that $\liminf\frac{1}{d^{p+1}}\Pi_{L_1}^*(W)>0$ is greater than zero for the sentences $W$ defined in \eqref{eqn: condition non-zero P24}.
Now following the argument in the block Wigner case, it follows that for block Toeplitz matrices with independent Wigner blocks of fixed size, $\eta_{p_n}$ converges in distribution to the standard Gaussian distribution for sequence of even positive integers $p_n=o(\log n / \log \log n)$. 

For the case where the blocks are other classical patterned random matrices given in Section \ref{subsec:common_patterned_eg}, a broader result is possible. Consider the block Toeplitz matrix of size $d(n) \times d(n)$ with symmetric circulant blocks of size $m(n) \times m(n)$. In this case, note that both $\liminf \frac{1}{m(n)^{p+1}} \# \Pi_{L_2}^*(W)$ and $\liminf \frac{1}{d(n)^{p+1}} \# \Pi_{L_1}^*(W)$ are greater than zero for the case $d,m \rightarrow \infty$ as well. Thus, it follows that for all choice of increasing functions $d,m: \N \rightarrow \N$, $\eta_{p_n}$ converges to the standard Gaussian distribution for even poitive integer sequence $p_n=o(\log n /\log \log n)$. The same conclusion also hold for other patterned blocks, such as Hankel, reverse circulant or circulant.

For block Hankel matrices too, the same conclusions hold, and the calculations are similar to the block Toeplitz case. 

\vskip2pt
\noindent \textbf{(iii) Block versions of circulant-type matrix:} The LSD of block symmetric circulant matrix with Wigner blocks was derived in \cite{Oraby_block_circulant},\cite{Murat_Kopp_Miller_Block_Circulant_2013}. The problem of fluctuations of LES for block symmetric circulant matrices has not been addressed yet. Clearly, Theorem \ref{thm:converg_LES_even} can be applied here for different choices of the blocks, such as Wigner blocks, Toeplitz blocks or Hankel blocks. One can use a similar idea as used for Block Toeplitz matrices to conclude the LES results for block symmetric circulant matrices and block reverse circulant matrices.
Table \ref{table:block} gives the different possibilities of block patterned matrices for which Theorem \ref{thm:converg_LES_even} is applicable.


\subsection{Extension to other models}\label{subsec:other_model_eg}

In this section, we discuss some other models of random matrices (apart from the patterned matrices) for which Theorem \ref{thm: always normal} can be applied. The LES for the random matrices discussed in this section is new to the literature.
\vskip2pt
\noindent \textbf{(i) $d$-disco matrix:} The random version of $d$-disco matrix was introduced in \cite{blackwell_etal_disco_2021} to study the transition of limiting spectral distribution from Gaussian (for Palindromic Toeplitz) to semi-circle law (for Wigner). 
Consider a symmetric patterned random matrix $A$ of dimension $M \times M$ and a sequence of independent patterned random matrices $\{B_n: n \in \N\}$, independent of $A$, such that $B_n$ is of the order $2^n M \times 2^n M$. The $d$\textit{-disco} of $A$ and $\{B_n\}$ is defined inductively as
$$\mathcal{D}_n=\left(\begin{array}{cc}
\mathcal{D}_{n-1} & B_{n-1} \\
B_{n-1} & \mathcal{D}_{n-1}
\end{array}\right) \mbox{ with } \mathcal{D}_0=A.$$ 

Suppose the link functions of $\{B_n\}$ obey Assumption \ref{Condition B}, the entries of $A,B_n$ obey Assumption \ref{assum: subexponential} and there exist fixed constants $c_1,c_2>0$, independent of $n$, such that $B_n$ obeys conditions (i) and (ii) of Theorem \ref{thm: always normal}. Then it follows that $\mathcal{D}_n$ also obeys conditions (i) and (ii) of Theorem \ref{thm: always normal}. In particular, it can easily be seen that the new constants are $c_1^\prime= c_1/2$ and $c_2^\prime=c_2/2$. Hence if $B_n$ belongs to the class of matrices discussed in Section \ref{subsec:common_patterned_eg}, then by Theorem \ref{thm: always normal} we get that for even positive integer sequence $p_n=o(\log n /\log \log n)$, $\eta_{p_n}$ converges to a Gaussian distribution for $d$-disco matrices.

\vskip2pt
\noindent \textbf{(ii) Checkerboard matrix:}
Let $\{x_{i}: i \in \Z^d\}$ be an input sequence. For  $k \in \mathbb{N}, w \in \mathbb{R}$, and a link function $L_n : \mathbb{Z}^2 \rightarrow \Z^d$, the $N \times N$ $(k, w)$-checkerboard ensemble is the matrix $M=\left(m_{i j}\right)$ given by
$$
m_{i j}= \begin{cases}x_{L_n(i,j)} & \text { if } i \not \equiv j \bmod k, \\ w & \text { if } i \equiv j \bmod k .\end{cases}
$$
 In \cite{Miller_etal_Checkerboard_2018_RMTA}, Burkhardt et al. introduced Checkerboard matrix and studied its LSD.

For $w=0$, the checkerboard matrix is a special case of the matrix defined in \eqref{eqn:A_ij_general}. It can be seen that when $L_n$ is one of the classical link functions discussed in Section \ref{subsec:common_patterned_eg}, then the conditions of Theorem \ref{thm: always normal} is satisfies. Therefore, it follows that for checkerboard matrices with classical link functions and $p_n$ a sequence of even positive integers of order $o(\log n /\log \log n)$, $\eta_{p_n}$ converges to standard Gaussian distribution for $(k,0)$-checkerboard matrices with appropriate assumptions on input entries. 
\vskip2pt
\noindent \textbf{(iii) Swirl Matrix:} The swirl of two $n \times n$ matrices $A$ and $X$ is defined as
\begin{equation*}
\operatorname{swirl}(A,X)=\left(\begin{array}{cc}
	AX & A \\
	XAX & XA
\end{array}\right).
\end{equation*}
Swirl matrices were introduced in \cite{Dunn_Miller_Swirl_2022} to study the LSD of circulant Hankel matrices. Note that when $X$ is the backward identity permutation matrix ($X_{ij}=\delta_{n}(i+j)$) and $A$ is a bisymmetric (symmetric and centrosymmetric) patterned random matrix. Then $\operatorname{swirl}(A,X)$ is a symmetric patterned random matrix. Suppose the link function of $A$ obeys Assumption \ref{Condition B}, the entries of $A$ obey Assumption \ref{assum: 4 moments} and $A$ obeys conditions (i) and (ii) of Theorem \ref{thm: always normal}. Then it follows that $\operatorname{swirl}(A,X)$ also obeys all the conditions of Theorem \ref{thm: always normal}, and as a result for even positive integer sequence $p_n=o(\log n / \log \log n)$, $\eta_{p_n}$ converges to standard Gaussian distribution for $\operatorname{swirl}(A,X)$.
\vskip2pt
\noindent\textbf{(iv) Hankel-type matrices:} Hankel matrices are closely related to reverse circulant matrices. The following family of Hankel-type matrices given by the link function
$$L_\theta(i,j)=(i+j) \bmod \lfloor n/\theta \rfloor ; \theta \in (0,\infty), $$ was introduced in \cite{Basak_Bose_Sunder_Hankeltype_2015} to study the transition from the LSD of Hankel matrix to LSD of reverse circulant matrix. Note that for $\theta < 1/2$, $L_\theta$ is the link function of Hankel matrices and for $\theta=1$, $L_\theta$ is the link function of reverse circulant matrix. 

The technique used to establish the convergence of $\eta_{p_n}$ for Hankel matrices in Section \ref{subsec:common_patterned_eg} also works for fixed $\theta$. It therefore follows from Theorem \ref{thm: always normal} that $\eta_{p_n}$ converges to standard Gaussian distribution for even positive integers $p_n=o(\log n/ \log \log n)$, for all fixed $\theta$.

\section*{\textbf{Appendix}} 

In this section, we prove Remark \ref{rem:normalization sqrt N} for some well-known patterned random matrices. Note that to prove Remark \ref{rem:normalization sqrt N}, it is sufficient to show that $\frac{1}{N}\varAp$ converges to a constant, and the remark follows from Slutsky's theorem. Note that from \eqref{eqn:var=1}
\begin{equation*}
	\frac{1}{N}\varAp=\frac{1}{N^{p+1}} \sum_{W \in SP(p,2)} \sum_{\Pi_{L_n,\Delta_n}(W)} \mathbb{E} \big( \prod_{i=1}^2 (x_{\pi_{i}}- \mathbb{E}x_{\pi_{i}} )\big)+o(1).
\end{equation*}
For input entries obeying Assumption \ref{assum: 4 moments} and for a subsentence $W_C$ of $W$, 
\begin{align*}
	\E \left(\prod_{ i=1}^2 x_{\pi_i}-\E x_{\pi_i}\right)=\begin{cases}
		1 &\text{ if } W_C \in \PP_{2}(p,p) \\
		\kappa-1 &\text{ if } W_C \in \PP_{2,4}(p,p).
	\end{cases}
\end{align*}
Hence, to prove the convergence of $\frac{1}{N}\varAp$, it is sufficient to show that
$\lim _{n \rightarrow \infty} \frac{1}{N^{p+1}} \# \Pi_{L_n, \Delta_n}(W)$ exists for all $W \in P_2(p, p) \cup P_{2,4}(p, p)$. 

Let $N: \N \rightarrow \N$ be a strictly increasing function and $\Delta_n \subseteq [N(n)]^2$ for each $n$. We make the following assumption on the sequence of sets $\{\Delta_n\}_{n \geq 1}$.
\begin{assumption}\label{assump: existence Delta}
	There exists a set $\Delta \subseteq [0,1]^2$ such that $\mathbb{U}_n\left((\frac{\Delta_n}{N} \setminus \Delta) \cup (\Delta \setminus \frac{\Delta_n}{N})\right)$ converges to zero as $n$ goes to infinity, where $\mathbb{U}_n$ is the uniform measure on $\{\frac{1}{N},\frac{2}{N},\ldots, 1\}^2$.
\end{assumption}

The following theorem helps us to prove the convergence of $\frac{1}{N}\varAp$ for the classical patterned random matrices and their banded versions.

\begin{theorem}\label{thm:Lij_secondmoment}
	Let $\left\{L_n\right\}$ be a sequence of link functions in one of the following forms:
	\begin{enumerate}[(i)]
		\item $L_n(i, j)=i \pm j.$
		\item $L_n(i, j)=|i - j|$.
		\item $L_n(i, j)=(i \pm j) \bmod N$ or $|i \pm j| \bmod N$, where $N$ is the dimension of the matrix.
	\end{enumerate}
	Suppose $\{\Delta_n\}$ obeys Assumption \ref{assump: existence Delta}.
	Then for each $W \in \PP_2(p, p) \cup \PP_{2,4}(p, p)$,
	$$\theta(W)=\lim _{n \rightarrow \infty} \frac{1}{N^{p+1}} \# \Pi_{L_n,\Delta_n}(W) \text { exists. }$$
\end{theorem}
\begin{proof}[Proof of (i)]
	We shall show that the result is true for $L_n(i,j)=i+j$ and the proof for $L_n(i,j)=i-j$ follows along similar lines.
	For the purpose of the proof, we denote
	\begin{equation*}
		v_{u,j}=\frac{\pi_{u}(j)}{N},
	\end{equation*}
	for $(\pi_1,\pi_2) \in \Pi_{L_n,\Delta_n}^*(W)$.
	
	We prove the theorem for $W \in \PP_{2,4}(p,p)$ and $W \in \PP_{2}(p,p)$ separately. 
	First, fix $W=(w_1,w_2) \in \PP_{2,4}(p,p)$. 
	Since $L_n(i,j)=i+j$, it follows that if $w_{u_1}[j_1]=w_{u_2}[j_2]$, then
	\begin{equation}\label{eqn:sys_of_eqn_vuj_PP2,4}
		v_{u_1,j_1-1}+ v_{u_1,j_1}= v_{u_2,j_2-1}+ v_{u_2,j_2}.
	\end{equation}
	Note that since $\pi_u(j) \in [N]$, the value of $v_{u,j}$ must be an integer multiple of $1/N$ less than or equal to one. Also observe that since $W \in \PP_{2,4}(p,p)$, \eqref{eqn:sys_of_eqn_vuj_PP2,4} is a system of equations with $p+1$ equations, where $(p-2)$ equations correspond to blocks of size two and three equations correspond to the block of size four. 
	Additionally, the following equations are also satisfied.
	\begin{align}\label{eqn:pi_1,pi_2,PP2,4}
		v_{1,0} =v_{1,p}, \
		v_{2,0}=v_{2,p}.
	\end{align}
	Our aim is to find all values of $v_{u,j}$, multiples of $1/N$, such that \eqref{eqn:sys_of_eqn_vuj_PP2,4} and \eqref{eqn:pi_1,pi_2,PP2,4} are satisfied. Recall the definition of generating vertices from Definition \ref{defn:generating_vertex}. Define
	\begin{equation*}
		S=\{\text{set of all generating vertices of } W\}
	\end{equation*}
	and let $v_S=\{v_{u,j}:(u,j) \in S\}$.
	Note that the cardinality of $S$ is $p+1$. We make the following claim:
	
		\noindent \textbf{Claim 1.}	For $u=1,2$ and $0 \leq j \leq p$, there exists a linear function $f_{u,j}$ with integer coefficients such that $v_{u,j}=f_{u,j}(v_S)$.
	\begin{proof}[Proof of Claim 1]
		Note that the claim holds trivially for all generating vertices $(u,j)$. Let $(u_1,j_1)$ be the first non-generating vertex with respect to the dictionary order. Then there exists a generating vertex $(u_0,j_0)$ such that $v_{u_0,j_0-1}+ v_{u_0,j_0}= v_{u_1,j_1-1}+ v_{u_1,j_1}$. Since $(u_1,j_1)$ is the first non-generating vertex, $v_{u_0,j_0-1},v_{u_0,j_0}$ and $v_{u_1,j_1-1}$ are elements of $v_S$ and therefore $v_{u_1,j_1}=f_{u_1,j_1}(v_S)$, where $f_{u_1,j_1}(v_S)$ is a linear function with integer coefficients.
		
		Now suppose $(u_1,j_1)$ is a non-generating vertex and for every vertex $(u,j)$ such that $(u,j)<(u_1,j_1)$ with respect to the dictionary order, there exist linear functions $f_{u,j}$ with integer coefficients such that $v_{u,j}=f_{u,j}(v_S)$. Further note that there exists a generating vertex $(u_0,j_0)$ such that $v_{u_0,j_0-1}+ v_{u_0,j_0}= v_{u_1,j_1-1}+ v_{u_1,j_1}$. Since $v_{u_0,j_0-1},v_{u_0,j_0}$ and $v_{u_1,j_1-1}$ are of the form $f_{u,j}(v_S)$, the claim follows.
	\end{proof}
	Now, we use Claim 1 to show that the limit $\lim _{n \rightarrow \infty} \frac{1}{N^{p+1}} \# \Pi_{L_n,\Delta_n}^*(W)$ exists. 
	Suppose elements of $\{v_S\}$ are chosen independently as  multiples of $1/N$, and each element lies between 0 and 1. Define $v_{u,j}=f_{u,j}(v_S)$, where $f_{u,j}$ is as defined in Claim 1. Then, note that by the construction of $f_{u,j}$, $\{v_{u,j}: u=1,2 , 0 \leq j \leq p\}$ obeys the system of equations \eqref{eqn:sys_of_eqn_vuj_PP2,4}. Further, since all elements of $v_S$ are multiples of $1/N$, it follows from Claim 1 that every $v_{u,j}$ is a multiple of $1/N$. 
	Hence, to check if $\{\pi_u(j)=Nv_{u,j}\}$ are elements of $\Pi_{L_n,\Delta_n}^*(W)$, it is sufficient to check the following: 
	\begin{enumerate}[(i)]
		\item $0 < f_{u,j}(v_S) \leq 1$ for all $u,j$,
		\item  the set of equations \eqref{eqn:pi_1,pi_2,PP2,4} are satisfied, and
		\item  $(v_{u,j-1},v_{u,j}) \in \frac{\Delta_n}{N}$ for $u=1,2$ and $1 \leq j \leq p$.
	\end{enumerate}
	Therefore, we get that $\#\frac{1}{N^{p+1}}\Pi_{L_n,\Delta_n}^*(W)$ is the Riemann sum of the integral
	\begin{align*}
		\int_{[0,1]^{p+1}} \prod_{j=1 \atop u=1,2}^{p} &\chi_{[0,1]} \left(f_{u,j}(v_S)\right) \delta(f_{1,p}(v_S)=v_{1,0})\\ &\times\delta(f_{2,p}(v_S)=v_{2,0})
		\prod_{j=1 \atop u=1,2}^p \chi_{\Delta}\big(f_{u,j-1}(v_S),f_{u,j}(v_S)\big)dv_S,
	\end{align*}
	with an error term bounded by $2 \times \frac{\#\left((\Delta_n \setminus \Delta) \cup (\Delta \setminus \Delta_n)\right)}{N^2}$. Since $\{\Delta_n\}$ obeys Assumption \ref{assump: existence Delta}, it follows that $\frac{1}{N^{p+1}}\#\Pi_{L_n,\Delta_n}^*(W)$ converges to the above integral as $n \rightarrow \infty$. Now, by Remark \ref{remark:Pi_Ln,Delta_n-star}, the result follows.
	
	Next, we consider $W=(w_1,w_2) \in \PP_2(p,p)$. Note that if $w_{u_1}[j_1]=w_{u_2}[j_2]$ and $(\pi_1,\pi_2) \in \Pi_{L_n,\Delta_n}^*(W)$, then \eqref{eqn:sys_of_eqn_vuj_PP2,4} and \eqref{eqn:pi_1,pi_2,PP2,4} are satisfied. Let $\{(1,r),(2,r^\prime)\}$ be the last cross-matched block in $W$ (i.e. $r$ is the largest number such that $(1,r)$ is an element of a cross-matched block). For $(w_1,w_2) \in \PP_{2}(p,p)$, we consider the following ordering `$\prec$' on $\{(u,j):u=1,2;j=0,1,2,\ldots, p\}$ given by
	\begin{align}\label{eqn:ordering_PP2}
		(1,0) \prec (1,1)\prec(1,2) \prec\cdots \prec(1,r-1) &\prec(1,p)\prec  (1,p-1)\prec\cdots \prec(1,r)  \nonumber \\
		&\prec(2,0)\prec(2,1)\prec \cdots \prec(2,p).
	\end{align}
	We shall use $(i,j)\preceq (r,s)$ to denote that either $(i,j)\prec (r,s)$ or $(i,j)=(r,s)$, i.e. $i=r$ and $j=s$.
	Recall from Definition \ref{defn:generating_vertex} that generating vertices are determined by the total ordering on $\{(u,j):u=1,2;j=0,1,2,\ldots, p\}$. Consider the generating vertices of $W$ with respect to the ordering in \eqref{eqn:ordering_PP2} and define 
	\begin{align}\label{eqn:set S_PP2(p,p)}
		\mathcal{S} = \{(1,j): j<r \text{ and }  &(1,j) \text{ is a generating vertex}\}\nonumber 
		\cup \{(1,j-1): j > r \text{ and }  (1,j) \text{ is }\\ & \text{ a generating vertex}\} \cup \{(2,j): (2,j) \text{ is a generating vertex}\}.
	\end{align}
	We define $v_\mathcal{S}=\{v_{u,j}=\pi_u(j)/N: (u,j) \in \mathcal{S}\}$, where $\mathcal{S}$ is as defined above. Similar to the case $W \in \PP_{2,4}(p,p)$, the following claim holds.
	
	\vskip2pt
	\noindent \textbf{Claim 2.}		For $u=1,2$ and $0 \leq j \leq p$, there exists a linear function $f_{u,j}$, with integer coefficients such that $v_{u,j}=f_{u,j}(v_\mathcal{S})$, where $\mathcal{S}$ is as defined in \eqref{eqn:set S_PP2(p,p)}.
\begin{proof}[Proof of Claim 2]
	Note that the claim holds trivially for vertices $(u,j) \in \mathcal{S}$. In the proof of this claim, we use the total ordering `$\prec$' defined in \eqref{eqn:ordering_PP2}. Similar to the proof of Claim 1, for all non-generating vertices $(1,j)$ such that $j \leq r-1$, it can be shown that $v_{1,j}=f_{1,j}(v_\mathcal{S})$ where $f_{1,j}$ is a linear function with integer coefficients. 
	
	We choose $v_{1,p}=v_{1,0}$ due to \eqref{eqn:pi_1,pi_2,PP2,4}. Now if $(1,p)$ is a generating vertex, then by the definition of $\mathcal{S}$, $v_{1,p-1} \in v_\mathcal{S}$ and the claim holds for $(1,p-1)$. Alternatively, if $(1,p)$ is not a generating vertex with respect to $\prec$, then there exists a generating vertex $(1,j_0)\prec(1,p)$ such that  $v_{1,j_0-1}+v_{1,j_0}=v_{1,p-1}+v_{1,p}$. Since it is already shown that for all $(1,j) \preceq (1,p)$, $v_{1,j}=f_{1,j}(v_\mathcal{S})$ where $f_{1,j}$ are linear functions with integer coefficients, it follows that $v_{1,p-1}$ is also of the form given in Claim 2. Next, consider $(1,p-1)\prec (1,j) \preceq (1,r)$ and suppose $v_{1,j^\prime}=f_{1,j^\prime}(v_\mathcal{S})$ for all $(1,j^\prime) \prec (1,j)$. For $(1,j) \notin \mathcal{S}$, it follows from the definition of $\mathcal{S}$ that there exists $(1,j_0) \prec (i,j)$ such that $v_{1,j_0-1}+v_{1,j_0}=v_{1,j}+v_{1,j+1}$. Note that as $v_{1,j_0-1}, v_{1,j_0},v_{1,j+1} \prec v_{1,j}$, the claim follows for $(1,j)$. The case of $(2,1) \preceq (2,j) \preceq (2,p)$ is similar to the case $(1,1) \preceq (1,j) \preceq (1,r-1)$ and hence the claim is proved for all $(u,j)$.
\end{proof}
Note that the cardinality of the set $\mathcal{S}$ is $p+1$, as $(1,r)$ is a generating vertex but there is no vertex in $\mathcal{S}$ that corresponds to $(1,r)$. Further, by construction, all equations in \eqref{eqn:sys_of_eqn_vuj_PP2,4} are satisfied and additionally the equation $v_{1,0}=v_{1,p}$ is also satisfied. Therefore, we get that $\#\frac{1}{N^{p+1}}\Pi_{L_n,\Delta_n}^*(W)$ is the Riemann sum of the integral
\begin{align*}
	\int_{[0,1]^{p+1}} \prod_{j=1 \atop u=1,2}^{p} &\chi_{[0,1]} \left(f_{u,j}(v_\mathcal{S})\right) \delta(f_{2,p}(v_\mathcal{S})= v_{2,0}) \prod_{j=1 \atop u=1,2}^p \chi_{\Delta}\left(f_{u,j-1}(v_\mathcal{S}),f_{u,j}(v_\mathcal{S})\right)dv_\mathcal{S},
\end{align*}
with an error term bounded by $2 \times \frac{\#\left((\Delta_n \setminus \Delta) \cup (\Delta \setminus \Delta_n)\right)}{N^2}$. Therefore, $\frac{1}{N^{p+1}}\#\Pi_{L_n,\Delta_n}^*(W)$ converges to the above integral as $n \rightarrow \infty$. Now, it is sufficient to show that this implies that  $\frac{1}{N^{p+1}}\Pi_{L_n}(W)$ converges. Recall the definition of $Q(W)$ from \eqref{eqn:differnce Pil,PiL star}.
For $W \in \PP_{2}(p,p)$, note that it is possible for $\tilde{W} \in Q(W)$ to belong to $\PP_{2,4}(p,p)$. Hence, the cardinality of the quantity in the right-hand side of \eqref{eqn:differnce Pil,PiL star} could be of the order $\Theta(N^{p+1})$. Suppose $\lim_{n \rightarrow \infty} \frac{1}{N^{p+1}}\# \Pi_{L_n}(W)$ exists for all $W \in \PP_{2,4}(p,p)$. Then it follows that the cardinality of the right hand side of \eqref{eqn:differnce Pil,PiL star} converges. Hence for all $W \in \PP_2(p,p)$, $\frac{1}{N^{p+1}}\# \Pi_{L_n}(W)$ exists provided $ \lim_{n \rightarrow \infty} \frac{1}{N^{p+1}}\# \Pi_{L_n}^*(W)$ exists.

Hence, it follows that  $\frac{1}{N^{p+1}}\Pi_{L_n}(W)$ converges.
\end{proof}
\begin{proof}[Proof of (ii) and (iii)]
For the case $L_n(i,j)=|i - j|$, a sentence $W=(w_1,w_2) \in \PP_{2}(p,p) \cup \PP_{2,4}(p,p)$ and $(\pi_1,\pi_2) \in \Pi_{L_n,\Delta_n}^*(W)$, note that if $w_{u_1}[j_1]=w_{u_2}[j_2]$, then $L_n(\pi_{u_1}(j_1-1),\pi_{u_1}(j_1))=L_n(\pi_{u_2}(j_2-1),\pi_{u_2}(j_2))$, i.e. $\pi_{u_1}(j_1-1)-\pi_{u_1}(j_1)=\pm \left(\pi_{u_2}(j_2-1)-\pi_{u_2}(j_2)\right)$. For each equation, fix a choice from $\{+1,-1\}$. Note that in this case, the number of distinct choices are $2^p$ for $W \in \PP_{2}(p,p)$ and $2^{p+1}$ for $W \in \PP_{2,4}(p,p)$. For each choice, we have a system of equations, as \eqref{eqn:sys_of_eqn_vuj_PP2,4}, and the number of solutions of the system of equation can be approximated by a Riemann sum as in part (i). This implies that $\lim_{n\to\infty}\frac{1}{N^{p+1}}\#\Pi^*_{L_n,\Delta_n}(W)$ is a finite sum of integrals. 

For the case of $L_n=(i+j)\mod N$, note that the system of equations can be written as
\begin{align}\label{eqn:sys_of_eqn_vuj_PP2,4_(iii)}
	\pi_{u_1}(j_1-1)+ \pi_{u_1}(j_1) &= \pi_{u_2}(j_2-1)+ \pi_{u_2}(j_2)\mod N. \nonumber \\
	\iff	\pi_{u_1}(j_1-1)+ \pi_{u_1}(j_1) &-\left( \pi_{u_2}(j_2-1)+ \pi_{u_2}(j_2)\right)=kN \text{ for } -2 \leq k \leq 2.
\end{align}
For each equation in \eqref{eqn:sys_of_eqn_vuj_PP2,4_(iii)}, fix a choice of $k$. This leads to a total of $5^{p}$ choice of system of equations for $W \in \PP_{2}(p,p)$ and $5^{p+1}$ for $W \in \PP_{2,4}(p,p)$. For each choice of the system of equations, $\frac{1}{N^{p+1}}\#\Pi^*_{L_n,\Delta_n}(W)$ can be approximated by an appropriate Riemann sum as in part (i). As the number of choices are finite, in this case, $\lim_{n\to\infty}\frac{1}{N^{p+1}}\#\Pi^*_{L_n,\Delta_n}(W)$ is a finite sum of integrals. Note that similar arguments also work for $L_n(i,j)=|i-j|\mod N$. This completes the proof.
\end{proof}
\begin{remark}
The proof of Theorem \ref{thm:Lij_secondmoment} would also work for link functions, $L(i,j)=(i \pm j) \mod M$ for $M=\omega(N)$.
\end{remark}

\vskip5pt
\noindent \textbf{Acknowledgment:} The research work of S.N. Maurya is supported by the fund:
NBHM Post-doctoral Fellowship (order no. 0204/10/(25)/2023/R$\&$D-II/2803). A part of this research was carried out  while attending the conference "Topics in High Dimensional Probability" (code: ICTS/thdp2023/1) at International Centre for Theoretical Sciences (ICTS).


\providecommand{\bysame}{\leavevmode\hbox to3em{\hrulefill}\thinspace}
\providecommand{\MR}{\relax\ifhmode\unskip\space\fi MR }
\providecommand{\MRhref}[2]{%
	\href{http://www.ams.org/mathscinet-getitem?mr=#1}{#2}
}
\providecommand{\href}[2]{#2}

\end{document}